\documentclass[10pt,a4paper]{article}

\usepackage[english]{babel}
\usepackage[utf8x]{inputenc}
\usepackage[T1]{fontenc}
\usepackage[a4paper,top=5cm,bottom=2cm,left=4cm,marginparwidth=1.75cm]{geometry}

\usepackage{amsmath,amsthm,amsfonts}
\usepackage{amssymb}
\usepackage{mathtools}
\usepackage{graphicx}
\usepackage{stmaryrd}
\usepackage[colorinlistoftodos]{todonotes}
\usepackage[colorlinks=true, allcolors=blue]{hyperref}
\usepackage{scalerel}
\usepackage{mathrsfs}
\usepackage{relsize}
\usepackage{enumerate}   
\usepackage{bm} 
\usepackage{dsfont}

\newtheorem{theorem}{Theorem}[section]
\newtheorem{lemma}[theorem]{Lemma}
\newtheorem{proposition}[theorem]{Proposition}
\newtheorem{corollary}[theorem]{Corollary}
\theoremstyle{definition}
\newtheorem{definition}[theorem]{Definition} 

\newtheorem{remark}[theorem]{Remark}

\theoremstyle{remark}

\newcommand{\C}{\mathbb{C}}
\newcommand{\R}{\mathbb{R}}
\newcommand{\N}{\mathbb{N}}
\newcommand{\Z}{\mathbb{Z}}
\newcommand{\SO}{\textnormal{\textbf{S}}_{0}}

\newcommand{\ghat}{\widehat{G}}

\newcommand{\s}{\textnormal{s}}
\newcommand{\NumWin}{n}
\newcommand{\NumDim}{d}
\newcommand{\La}{\Lambda}
\newcommand{\la}{\lambda}
\newcommand{\Lac}{\Lambda^\circ}
\newcommand{\lac}{\lambda^\circ}

\newcommand\blfootnote[1]{%
  \begingroup
  \renewcommand\thefootnote{}\footnote{#1}%
  \addtocounter{footnote}{-1}%
  \endgroup
}

\newcommand{\cocy}{\textsf{c}}

\newcommand{\tr}{\textnormal{tr}}

\newcommand{\lmodule}{\mathcal{A}}
\newcommand{\rmodule}{\mathcal{B}}

\DeclareBoldMathCommand\boldlangle{\left\langle} 
\DeclareBoldMathCommand\boldrangle{\right\rangle}

\newcommand{\lhs}[2]{\prescript{}{\lmodule}{\boldlangle #1,#2\boldrangle}}
\newcommand{\rhs}[2]{{\boldlangle #1,#2\boldrangle}_{\!\rmodule}}

\newcommand{\mlhs}[2]{\prescript{}{\lmodule}{{\textnormal{\textbf{[}}} #1,#2 \textnormal{\textbf{]}}}}
\newcommand{\mrhs}[2]{ \textnormal{\textbf{[}} #1, #2 {\textnormal{\textbf{]}}}_{\rmodule}}

\newcommand{\mvfun}[3]{
\ifthenelse{\equal{#2}{}}
	{\ifthenelse{\equal{#3}{}}  
    	{#1_{\bullet,\bullet}}{  
    	{#1_{\bullet,#3}}
    }}{ 
    {\ifthenelse{\equal{#3}{}}  
    	{#1_{#2,\bullet}}
    {  
    	{#1_{#2,#3}}}}}
} 

\newcommand{\vvfun}[2]{
\ifthenelse{\equal{#2}{}}
	{#1_{\bullet}}
  {#1_{#2}}
} 
\newcommand{\cE}{\mathcal{E}}
\usepackage{fancyhdr}
\title{Duality of Gabor frames and Heisenberg modules}
\author{Mads S.\ Jakobsen \thanks{Norwegian University of Science and Technology, Department of Mathematical Sciences, Trondheim, Norway, \mbox{E-mail: \protect\url{mads.jakobsen@ntnu.no}; \protect\url{franz.luef@ntnu.no}}}, Franz Luef\footnotemark[1]}

\begin{document}
\date{}
\maketitle


\begin{abstract}
Given a locally compact abelian group $G$ and a closed subgroup $\La$ in $G\times\ghat$, Rieffel associated to $\La$ a Hilbert $C^*$-module $\cE$, known as a Heisenberg module. He proved that $\cE$ is an equivalence bimodule between the twisted group $C^*$-algebra $C^*(\Lambda,\cocy)$ and $C^*(\Lambda^\circ,\overline{\cocy})$, where $\Lambda^{\circ}$ denotes the adjoint subgroup of $\Lambda$. Our main goal is to study Heisenberg modules using tools from time-frequency analysis and pointing out that Heisenberg modules provide the natural setting of the duality theory of Gabor systems. More concretely, we show that the Feichtinger algebra $\SO(G)$ is an equivalence bimodule between the Banach subalgebras $\SO(\La,\cocy)$ and $\SO(\Lac,\overline{\cocy})$ of $C^*(\Lambda,\cocy)$ and $C^*(\Lambda^\circ,\overline{\cocy})$, respectively. Further, we prove that $\SO(G)$ is finitely generated and projective exactly for co-compact closed subgroups $\Lambda$. In this case the generators $g_1,\ldots,g_n$ of the left $\SO(\La)$-module $\SO(G)$ are the Gabor atoms of a multi-window Gabor frame  for $L^2(G)$. We prove that this is equivalent to $g_1,\ldots,g_n$ being a Gabor super frame for the closed subspace generated by the Gabor system for $\Lac$. This duality principle is of independent interest and is also studied for infinitely many Gabor atoms. We also show that for any non-rational lattice $\La$ in $\R^{2m}$ with volume $s(\La)<1$ there exists a Gabor frame generated by a single atom in $\SO(\R^m)$. 
\end{abstract}

\section{Introduction}

In this paper we revisit Rieffel’s construction of Heisenberg modules for locally compact abelian groups \cite{ri88}, which followed Connes’ construction of these modules for non-commutative 2-tori in his seminal work \cite{co80}. Heisenberg modules have been the subject of many investigations in noncommutative 
geometry \cite{bo99-2,chlu17,ecluphwa10,hu18-1,la17-1,la18,lemo16,li04-4,lali12-1,lapa13,lapa15,lapa17}, 
but the interplay between the left and the right module structure has not been addressed at all. A link between applied harmonic analysis and noncommutative geometry was described in \cite{lu09,lu11} by relating the Heisenberg modules over noncommutative tori, \cite{ri88}, with (multi-window) Gabor frames for $L^2(\mathbb{R}^m)$. In this way we get to see the natural interplay between the duality theory of Gabor analysis and the Morita equivalence of twisted group algebras.

In order to best describe and motivate our results let us give a brief account of the theory in the classical case of $L^{2}(\R)$, \cite{co80}. We consider the modulation operator and the translation operator
\[ E_{\beta} f(t) = e^{2\pi i \beta t} f(t), \ T_{\alpha} f(t) = f(t-\alpha), \ \ \alpha,\beta\in\R\backslash\{0\}, \ f\in L^{2}(\R),\]
that act unitarily on $L^{2}(\R)$. In \emph{time-frequency analysis} one wants to find $\alpha,\beta$ and two functions $g$ and $h$ in $L^{2}(\R)$ such that
\begin{equation} \label{eq:1605a} f = \sum_{m,n\in\Z} \langle f, E_{m\beta}T_{n\alpha} g \rangle \, E_{m\beta}T_{n\alpha} h \ \ \text{for all} \ \ f\in L^{2}(\R).\end{equation}
Here $\langle \, \cdot \, , \, \cdot \, \rangle$ is the usual $L^{2}$-inner-product with the linearity in the left entry. The theory of \emph{frames} allows us to describe when reproducing formulas of the form \eqref{eq:1605a} are possible. A system of the form $\{E_{m\beta}T_{n\alpha}g\}_{m,n\in\Z}$ is a \emph{Gabor} system. These systems have been studied extensively, see \cite{grrost18,grst13} for recent results on the frame set of totally positive functions and \cite{jale16-1,jale16-2} for the theory in the setting of locally compact abelian groups. 
We will introduce frames and Gabor systems thoroughly in Section~\ref{sec:mws-gabor}. 

In the theory of noncommutative tori sums of the form \eqref{eq:1605a} appear in the following way:
the two operators $E_{\beta}$ and $T_{\alpha}$ are used to construct a closed subspace of the linear and bounded operators on $L^{2}(\R)$,
\begin{align*} 
\lmodule & = \big\{ {\bf a} \in \mathsf{B}(L^{2}(\R)) \, : \, {\bf{a}} =  \sum_{m,n\in\Z} a(m,n) \, E_{m\beta}T_{n\alpha} \, , \ a\in \ell^{1}(\Z^{2})\big\}.
\end{align*}
The norm $\Vert \mathbf{a} \Vert_{\mathcal{A}} = \Vert a \Vert_{1}$, where $\mathbf{a}$ and $a$ are related as above, turns 
 $\mathcal{A}$ into an involutive Banach algebra with respect to composition of operators and the taking of adjoints.
$\mathcal{A}$ is a faithful representation of the twisted group algebra $\ell^{1}(\Z^{2})$ with the twisted convolution and involution (with a phase factor) given by
\begin{subequations} \label{eq:1605c}
\begin{align} a_{1} \,\natural\, a_{2} (m,n) & = \sum_{m',n'\in\Z} a_{1}(m',n') \, a_{2}(m-m',n-n') \, e^{2\pi i \beta \alpha  (m-m')  n'}, \\
a^{*}(m,n) & = e^{2\pi i \alpha\beta mn} \, \overline{a(-m,-n)}. \end{align}
\end{subequations}
The left-action that $\mathbf{a}\in \lmodule$ has on functions $f\in L^{2}(\R)$ is given by $\mathbf{a}\cdot f=\sum_{m,n\in\Z}a(m,n)E_{m\beta}T_{n\alpha}f$. \\
For functions in $\SO(\R)$ we define an $\lmodule$-valued inner-product in the following way:
\begin{align*} &\lhs{\,\cdot\,}{\,\cdot\,} : \SO(\R)\times \SO(\R) \to \lmodule, \ \lhs{f}{g} = \sum_{m,n\in\Z} \langle f, E_{m\beta}T_{n\alpha} g \rangle \, E_{m\beta} T_{n\alpha}.
\end{align*}
Here $\SO(\R)$ is \emph{Feichtinger's algebra}: a suitable Banach space of test-functions, which is widely used in time-frequency analysis. For now readers may think of it as a space akin to the Schwartz space. We summarize its relevant theory and properties in Section \ref{sec:fei-alg}.
With the inner-product $\lhs{\cdot}{\cdot}$ we can write the desired equality in \eqref{eq:1605a} as $f = \lhs{f}{g}\cdot h$.

Another question in time-frequency analysis is to find values of $\alpha$ and $\beta$, a function $g$, and a process that make it possible to recover a given sequence $c\in \ell^{2}(\Z^{2})$ from a function of the form
\[ f = \sum_{m,n\in\Z} c(m,n) \, E_{m\beta}T_{n\alpha} g.\]
The solution to this problem is to make sure that $\{E_{m\beta}T_{n\alpha}g\}_{m,n\in\Z}$ is a \emph{Riesz sequence}. We give a definition of Riesz sequences in Section \ref{sec:mws-gabor}. For now we wish to emphasize the following:
there is an intimate relation between the frames and Riesz sequences of the form $\{E_{m\beta}T_{n\alpha} g\}_{m,n\in\Z}$. This is known as the \emph{duality  principle for Gabor systems} (we provide references for this result in the last paragraph of the introduction):

\begin{theorem} \label{th:frame-riesz-dual-intro} For any choice of $\alpha,\beta\in\R\backslash\{0\}$ and any $g\in L^{2}(\R)$ the following statements are equivalent. \begin{enumerate}
\item[(i)] The Gabor system $\{E_{m\beta}T_{n\alpha} g\}_{m,n\in\Z}$ is a frame for $L^{2}(\R)$,
\item[(ii)] The Gabor system $\{(E_{m/\alpha}T_{n/\beta})^{*} g\}_{m,n\in\Z}$ is a Riesz sequence for $L^{2}(\R)$.
\end{enumerate}
\end{theorem}
\noindent Note: in time-frequency analysis assertion (ii) is typically expressed as the (equivalent) statement that $\{E_{m/\alpha} T_{n/\beta} g\}_{m,n\in\Z}$ is a Riesz sequence.

The key behind the equivalence in Theorem \ref{th:frame-riesz-dual-intro} is the following identity in \eqref{eq:janssen-intro} that was 
established independently by Rieffel \cite[Proposition 2.11]{ri88} and seven years later in the  time-frequency analysis community by Daubechies, Laundau, Landau \cite[Theorem 3.1]{dalala95}, and Janssen \cite[Proposition 2.4]{ja95}. For $\alpha,\beta\ne0$ and suitably chosen functions $f,g$ and $h$ (e.g.\ in the Schwartz space, or more generally in $\SO(\R)$, or under slightly more general assumption) one has the equality 
\begin{equation} \label{eq:janssen-intro} \sum_{m,n\in\Z} \langle f, E_{m\beta}T_{n\alpha} g \rangle \, E_{m\beta} T_{n\alpha} h = \frac{1}{\vert\alpha\beta\vert} \sum_{m,n\in\Z} \langle h, \big( E_{m/\alpha} T_{n/\beta}\big)^{*}g\rangle \,\big( E_{m/\alpha} T_{n/\beta}\big)^{*} f. \end{equation}

The connection between Gabor systems generated by the operators $E_{\beta}$ and $T_{\alpha}$ and the Gabor systems generated by the operators $E_{1/\alpha}$ and $T_{1/\beta}$ appear in the study of Morita equivalence for noncommutative tori as follows \cite{lu09}: resembling the definition of $\lmodule$ from before, we set
\begin{align*}
\rmodule & = \big\{ {\bf b} \in \mathsf{B}(L^{2}(\R)) \, : \, {\bf{b}} = \frac{1}{\vert\alpha\beta\vert}\sum_{m,n\in\Z} b(m,n) \, \big(E_{m/\alpha}T_{n/\beta}\big)^{*}, \ b\in \ell^{1}(\Z^{2})\big\} \end{align*}
with the norm $\Vert \mathbf{b} \Vert_{\rmodule} = \Vert b \Vert_{1}$. $\rmodule$ is an involutive Banach algebra with respect to composition and the taking of adjoints and a faithful representation of a twisted group algebra $\ell^{1}(\Z^{2})$. The twisted convolution and involution on $\ell^{1}(\Z^{2})$ is slightly different here compared to the one in \eqref{eq:1605c} due to the differences between $\lmodule$ and $\rmodule$. We define a $\rmodule$-valued inner product 
\[ \rhs{\,\cdot\,}{\,\cdot\,} :\SO(\R)\times\SO(\R)\to \rmodule, \ \rhs{f}{g} = \frac{1}{\vert\alpha\beta\vert} \sum_{m,n\in\Z} \langle g, \big( E_{m/\alpha} T_{n/\beta}\big)^{*}f\rangle \,\big( E_{m/\alpha} T_{n/\beta}\big)^{*} \]
that has a right-action on functions $h\in L^{2}(\R)$ given by
\[ h \cdot \rhs{f}{g} = \frac{1}{\vert\alpha\beta\vert} \sum_{m,n\in\Z} \langle g, \big( E_{m/\alpha} T_{n/\beta}\big)^{*}f\rangle \,\big( E_{m/\alpha} T_{n/\beta}\big)^{*} h.\]
With the $\lmodule$- and $\rmodule$-valued inner products we can write the crucial equality \eqref{eq:janssen-intro} as
\[ \lhs{f}{g}\cdot h = f \cdot \rhs{g}{h} \ \ \text{for all} \ \ f,g,h\in\SO(\R).\]
In Rieffel's theory of Morita equivalence for $C^*$-algebras this is the key ingredient in the argument that $\SO(\R)$ is an $\lmodule$-$\rmodule$-equivalence bimodule. This implies that $\lmodule$ and $\rmodule$ are Morita equivalent. A non-trivial question in the theory of $C^*$-algebras is whether or not there exist elements $p\in\lmodule$ such that $p^2 = p$. This is related with the construction of Gabor Riesz sequences, and by Theorem \ref{th:frame-riesz-dual-intro} also with the construction of Gabor frames \cite{lu11}. In fact, from the general theory of frames and Riesz sequences and Theorem \ref{th:frame-riesz-dual-intro} it is possible to deduce the following.

\begin{theorem} \label{th:duality-for-mws-module-intro} For any pair of functions $g,h\in\SO(\R)$ and parameters $\alpha,\beta\ne 0$ the following statements are equivalent.
\begin{enumerate}[(i)]
\item $f = \lhs{f}{g} \cdot h \ $ for all $\ f\in \SO(\R)$, i.e.,
\begin{equation} \label{eq:1605d} f = \textstyle\sum\limits_{m,n\in\Z} \langle f, E_{m\beta}T_{n\alpha} g\rangle \, E_{m\beta}T_{n\alpha} h.
\end{equation}
\item The operator 
\[\rhs{g}{h} : L^{2}(\R)\to L^{2}(\R), \ f\cdot \rhs{g}{h} = \vert \alpha\beta\vert^{-1} \textstyle\sum\limits_{m,n\in\Z} \langle h, (E_{m/\alpha}T_{n/\beta})^{*}g\rangle \, (E_{m/\alpha}T_{n/\beta})^{*} f\] 
is the identity operator, i.e, $\vert \alpha\beta \vert^{-1} \, \langle h, (E_{m/\alpha}T_{n/\beta})^{*}g \rangle = \delta_{(m,n),(0,0)}$ for all $(m,n)\in\Z^{2}$.
\item The operator
\[ \lhs{g}{h} : L^{2}(\R)\to L^{2}(\R), \ \lhs{g}{h} \cdot f = \textstyle\sum\limits_{m,n\in\Z} \langle g, E_{m\beta}T_{n\alpha} h\rangle \, E_{m\beta}T_{n\alpha} f\] 
is an idempotent operator from $L^{2}(\R)$ onto $ \overline{\textnormal{span}}\{ (E_{m/\alpha}T_{n/\beta})^{*} g\}_{m,n\in\Z}$.
\item $f = g \cdot \rhs{h}{f}$ for all $ f\in\SO(\R)\cap \overline{\textnormal{span}}\{ (E_{m/\alpha}T_{n/\beta})^{*} g\}_{m,n\in\Z}$, i.e.,
\begin{equation} \label{eq:1605e} f = \vert \alpha \beta \vert^{-1} \textstyle\sum\limits_{m,n\in\Z} \langle f, (E_{m/\alpha}T_{n/\beta})^{*}h\rangle \,  (E_{m/\alpha}T_{n/\beta})^{*}g.\end{equation}
\end{enumerate}
The closure in (iii) and (iv) is with respect to the $L^{2}$-norm. The equalities in \eqref{eq:1605d} and \eqref{eq:1605e} extend to all $f\in L^{2}(\R)$ and to all $f\in\overline{\textnormal{span}}\{ (E_{m/\alpha}T_{n/\beta})^{*} g\}_{m,n\in\Z}$, respectively. All statements are equivalent to the ones where $g$ and $h$ are interchanged.
\end{theorem}

This finishes our review of the theory for the traditional setting. In this work we expand the above theory as follows.

\begin{enumerate}[(A)]
\item \textbf{Multi-window Gabor systems \& finitely generated modules.} \label{lab:mw} 
Instead of just one pair of functions $g$ and $h$ so that $f=\lhs{f}{g}\cdot h$ holds for all $f\in \SO(\R)$, we ask to find collections of functions $\{g_{j}\}_{j=1}^{n}$ and $\{h_{j}\}_{j=1}^{n}$ in $\SO(\R)$ such that 
\begin{equation} \label{eq:1805a} f = \sum_{j=1}^{n} \lhs{f}{g_{j}} \cdot h_{j} \ \ \text{for all} \ \ f \in \SO(\R).\end{equation}
In time-frequency analysis these representations arise from \emph{multi-window} Gabor systems and their frame theory. In the theory of Hilbert $C^*$-modules the equality above implies that the equivalence bimodule $\SO(\R)$ is \emph{finitely generated}. In particular, we show that \eqref{eq:1805a} holds if and only if $(\lhs{g_{j}}{h_{j'}})_{j,j'}$ is an idempotent  $n\times n$ $\lmodule$-valued matrix.
\item \textbf{Super Gabor systems.} \label{lab:super}  
The theory of \emph{super} Gabor frames asks to find collections of functions $\{g_{k}\}_{k=1}^{d}$ and $\{h_{k}\}_{k=1}^{d}$ in $\SO(\R)$ such 
that
\begin{equation} \label{eq:1805b} f_{k} = \sum_{k'=1}^{d} \lhs{f_{k'}}{g_{k'}} \cdot h_{k} \ \ \text{for all} \ \ f \in \SO(\R,\C^{d}),\end{equation}
where $\SO(\R,\C^{d})$ are the $\C^{d}$-vector valued functions over $\R$ that belong to $\SO$. To our knowledge these expansions do not show up in the theory of Hilbert $C^*$-modules. We show that \eqref{eq:1805b} holds if and only $\sum_{k=1}^{d} \lhs{h_{k}}{g_{k}}$ is an idempotent element of $\lmodule$.

\item \textbf{Multi-window super Gabor systems.} \label{enum:mws} In order to take care of the theory of multi-window and super Gabor systems described in \eqref{lab:mw} and \eqref{lab:super} in the most general way, we consider them simultaneously and wish to describe when two families of functions $\{g_{k,j}\}_{k=1,j=1}^{d,n}$ and $\{h_{k,j}\}_{k=1,j=1}^{d,n}$ in $\SO(\R,\C^{d\times n})$ are such that
\[ f_{k} = \sum_{j=1}^{n} \sum_{k'=1}^{d} \lhs{f_{k'}}{g_{k',j}} \cdot h_{k,j} \ \ \text{for all} \ \ f \in \SO(\R,\C^{d}).\]
We then generalize Theorem \ref{th:frame-riesz-dual-intro} to a duality principle for these \emph{multi-window super} Gabor systems (Theorem \ref{th:duality}). In particular, we show that there is a duality between the parameters $n$ and $d$. We extend Theorem \ref{th:duality-for-mws-module-intro} for the associated operator valued inner-products as well (Theorem \ref{th:duality-for-mws-module}). 
\item \textbf{General locally compact abelian groups.} Rather than focusing on (vector/matrix valued) functions on $\R$ we consider the theory presented above for functions on locally compact abelian (LCA) groups $G$ and where the translations and modulations need not arise from a lattice (e.g., of the form $\alpha\Z\times\beta\Z$ as in \eqref{eq:1605a}) but may originate from any closed subgroup $\La$ of the time-frequency plane. In particular, we do not require this subgroup to be discrete. For example, in case  $G=\R^{m}$, we can consider $\La = A (\R^{m_{1}}\times \Z^{m_{2}}\times\{0\}^{m_{3}})$, where $m_{1}+m_{2}+m_{3}= 2m$ and $A\in \textnormal{GL}_{2m}(\R)$ (of course there are far more elaborate examples to be made, e.g., the theory is applicable to the group of the adeles \cite{enjalu18}).
\item  \textbf{$\SO$ vs.\ Schwartz-Bruhat.} We follow the vision of Feichtinger \cite{fe81-2} and use the Segal algebra $\SO$ as a space of test-functions on locally compact abelian groups rather than the technical Schwartz-Bruhat space. This Banach space of functions is easily defined on any LCA group and behaves very much like the Schwartz-Bruhat space. We go over the relevant theory in Section \ref{sec:fei-alg}.
\item \label{enum:existence} \textbf{Existence.} Our simultaneous description of the theory of Gabor frames and Riesz sequences together with that of Heisenberg modules allows us to combine results from time-frequency analysis and operator algebras. We use this to characterize the subgroups $\La$ from which the translations and modulations may arise such that $\SO(G)$ is a finitely generated projective module, i.e., such that there exist multi-window Gabor frames for $L^{2}(G)$ with generators in $\SO(G)$, i.e., such that  there exist functions $\{g_{j}\}_{j=1}^{n}$ and $\{h_{j}\}_{j=1}^{n}$ in $\SO(G)$ such that
\[ f = \sum_{j=1}^{n} \lhs{f}{g_{j}}\cdot h_{j} \ \ \text{for all} \ \ f\in \SO(G)\]
(with a properly defined Banach algebra $\lmodule$ and an associated $\lmodule$-valued inner product that both depend on $\La$). In Theorem~\ref{th:S0-fin-gen} and Theorem~\ref{th:exist} we show that this is the case \emph{exactly} if the quotient group $(G\times\ghat)/\La$ is compact (that this is a necessary condition was established in \cite{jale16-2}). \\ Further, in Theorem \ref{th:existence-lattice-less-than-1}, we contribute to a long standing problem for $G=\R^{m}$: If $\La$ is a lattice in $\R^{2m}$ with lattice size $\s(\La)< 1$, does there exist a function $g\in \SO(\R^{m})$ that generates a Gabor frame for $L^{2}(\R^{m})$ with respect to that lattice? We use results from Rieffel and the here established relationships between Gabor frames and projections in Heisenberg modules to give a completely non-constructive proof of this fact for non-rational lattices.
\end{enumerate}

The paper is structured as follows.
Section \ref{sec:abstract-HA} introduces the reader to  necessary notions and definitions concerning Fourier analysis on locally compact abelian groups. Section \ref{sec:fei-alg} gives a short review of the Feichtinger algebra and highlights some of its properties that are important to us.

Section \ref{sec:module} states and expands upon the Heisenberg module theory by Rieffel. 
Section \ref{sec:module-mws} shows how this theory is generalized to capture the theory of multi-window and super Gabor systems, which we develop in detail in Section 
\ref{sec:mws-gabor}. The main results in Section \ref{sec:module} are Theorem \ref{th:S0-fin-gen} and Theorem \ref{th:exist} (as mentioned under point \eqref{enum:existence} above) and Theorem \ref{th:duality-for-mws-module} (see point \eqref{enum:mws}).


Section \ref{sec:mws-gabor} begins with the definition of the (multi-window super) Gabor systems that we consider in this paper.
Section \ref{sec:frame-and-riesz} gives a brief account of the  theory of frames and Riesz sequences. In Section \ref{sec:continuous-case} the simplest case of Gabor analysis is considered: where the subgroup $\Lambda$ is the entire phase-space. In Section \ref{sec:duality} we state another main result: the duality principle for multi-window super Gabor systems (the generalization of Theorem \ref{th:frame-riesz-dual-intro} from above). In Section \ref{sec:figa-wex-raz} 
we develop the necessary tools for the proof of the duality principle, namely the \emph{fundamental identity of Gabor analysis} (the generalization of \eqref{eq:janssen-intro}) and the Wexler-Raz relations (they yield a characterization of dual Gabor frames).
Finally, Section \ref{sec:proof-duality} contains the proof of the duality principle.

In Section \ref{sec:existence-mws-frames-for-R} we consider Gabor analysis for functions of $\R^{m}$. 
We give a much more concrete proof of Theorem \ref{th:exist}: the existence of multi-window super Gabor frames for any subgroup $\Lambda$ of the time-frequency plane $\R^{2m}$ for which $\R^{2m}/\La$ is compact, Theorem \ref{th:existence-mwsframes-rn}. 
Furthermore, a major result is Theorem \ref{th:existence-lattice-less-than-1} as described in point \eqref{enum:existence}. 

For the Feichtinger algebra, we show the new result that it forms a Banach algebra with respect to the twisted convolution, Lemma \ref{le:SO-twist-conv-algebra}.



\textbf{Related literature on the duality principle for Gabor systems.} For $G=\R$ and where $\Lambda = \alpha\Z \times \beta\Z$ ($\alpha,\beta\ne 0$) the duality principle for Gabor systems (Theorem \ref{th:frame-riesz-dual-intro}) was shown by Daubechies, Landau and Landau \cite{dalala95} and Janssen \cite{ja95}. The more general case for $G=\R^{n}$, $n\in \N$ with $\Lambda = A \Z^{n} \times B\Z^{n}$ ($A,B\in \textnormal{GL}_{n}(\R)$) is due to Ron and Shen \cite{rosh97}. It should be noted that these three papers proved the result independently from one another and with rather different approaches. The duality principle between super/multi-window, Gabor frames and multi-window/super Riesz sequences in the Euclidean setting can be found and is used in \cite{ab10,dogr11}, \cite[Section 5.3]{gr07-2} and \cite[Theorem 2.6]{grly09}.
Recently, in \cite{jale16-2}, the duality principle has been shown for Gabor systems on general locally compact abelian groups $G$ where $\Lambda$ is a closed subgroup in the time-frequency plane $G\times\ghat$.  
Our result, Theorem \ref{th:duality}, contains all the above as special cases. 
An extension of the duality principle beyond Gabor systems is the \emph{R-duality} theory for frames proposed in \cite{cakula04,cakula05}, which is further investigated in \cite{chst15,chst16}. Another generalization to  collections of functions other than Gabor systems is the work by Fan et al.\ \cite{fahesh16}. \\
Furthermore, in \cite{duhala09} a duality principle between frames and Riesz sequences that arise as samples of projective representations is proven. In particular, this contains (single window) Gabor systems with time-frequency shifts from lattices. Further, our work has overlap with the unpublished manuscript \cite{baduhala18}, where the authors establish a duality between ``multi-window'' and ``super'' systems as we consider them here. The system that they consider need not be of the Gabor type but can be any family of functions that arises as the samples of a projective representation that acts on a given finite collection of functions (similar to the setting in \cite{duhala09}). Their setting can only consider the case where $\Lambda$ and $\Lambda^{\circ}$ both are discrete and co-compact subgroups, whereas our approach allows for $\Lambda$ and $\Lambda^{\circ}$ to be any closed subgroup. Furthermore, in the case of Gabor systems, the duality principle that we show here is more general and it also details how the bounds of the frames and of the Riesz sequences are related by the duality.

\section{Preliminaries}
\label{sec:prelim}
\subsection{Abstract harmonic analysis}
\label{sec:abstract-HA}
We let $G$ be a locally compact Hausdorff abelian topological (LCA) group and we let $\ghat$ denote its dual group. 
The action of a character $\omega\in \ghat$ on an element $x\in G$ is written as $\omega(x)$. We assume some fixed Haar measure $\mu_{G}$ on $G$ and we normalize the Haar measure $\mu_{\ghat}$ on $\ghat$ in the unique way such that the Fourier inversion holds. That is, such that it is possible to reconstruct a continuous representative of a function $f\in L^{1}(G)$ by its Fourier transform 
\[ \mathcal{F}f(\omega) = \int_{G} f(t) \, \overline{\omega(t)} \, dt, \ \ \omega\in \ghat,\]
in case $\mathcal{F}f\in L^{1}(\ghat)$, by the formula
\[ f(t) = \int_{\ghat} \mathcal{F}f(\omega) \, \omega(t) \, d\omega \ \ \text{for all} \ \ t\in G.\]
We equip $L^{2}(G)$ with the inner product $\langle f, g\rangle = \int_{G} f(t) \overline{g(t)} \, dt$ which is linear in the first entry. The Fourier transform extends to a unitary operator on $L^{2}(G)$.

For any $x\in G$ and $\omega\in\ghat$ we define the translation operator (time-shift) $T_{x}$ and the modulation operator (frequency-shift) $E_{\omega}$ by
\[ T_{x}f(t) = f(t-x) \ \ \text{and} \ \ E_{\omega} f(t) = \omega(t) f(t), \ \ t\in G,\]
where $f$ is a complex-valued function on $G$.
Observe that 
\[ \mathcal{F}T_{x} = E_{-\omega} \mathcal{F} \ \ , \ \ \mathcal{F}E_{\omega}=T_{\omega} \mathcal{F} \ \ , \ \  E_{\omega}T_{x} = \omega(x) T_{x} E_{\omega} .\]
For any $\chi = (x,\omega)\in G\times\ghat$
we define the \emph{time-frequency shift operator} 
\[ \pi(\chi) \equiv \pi(x,\omega) := E_{\omega} T_{x}.\]
It is clear that time-frequency shift operators are unitary on $L^{2}(G)$. 
For two elements $\chi_{1}=(x_{1},\omega_{1})$ and $\chi_{2}=(x_{2},\omega_{2})$ in $G\times\ghat$ we define the \emph{cocycle}
\begin{equation} \label{eq:def-cocy} \cocy : (G\times\ghat)\times (G\times\ghat) \to \mathbb{T}, \ \cocy(\chi_{1},\chi_{2}) = \overline{\omega_{2}(x_{1})} \end{equation}
and the associated \emph{symplectic cocyle}
\begin{equation} \label{eq:def-cocy-sym} \cocy_s : (G\times\ghat)\times (G\times\ghat) \to \mathbb{T}, \ \cocy_s(\chi_{1},\chi_{2}) = \cocy(\chi_{1},\chi_{2}) \, \overline{\cocy(\chi_{2},\chi_{1})} = \overline{\omega_{2}(x_{1})} \, \omega_{1}(x_{2}). \end{equation}
For any $\chi,\chi_{1},\chi_{2},\chi_{3}\in G\times\ghat$ the cocycle and time-frequency shift satisfy the following,
\begin{align*}
\overline{\cocy(\chi_{1},\chi_{2})} & = \cocy(-\chi_{1},\chi_{2}) = \cocy(\chi_{1},-\chi_{2}), \\
\cocy(\chi_{1}+\chi_{2},\chi_{3}) & = \cocy(\chi_{1},\chi_{3}) \, \cocy(\chi_{2},\chi_{3}), \ \ \cocy(\chi_{1},\chi_{2}+\chi_{3}) = \cocy(\chi_{1},\chi_{2}) \, \cocy(\chi_{1},\chi_{3}), \\
 \pi(\chi_{1}) \, \pi(\chi_{2}) &  = \cocy(\chi_{1},\chi_{2}) \, \pi(\chi_{1}+\chi_{2}), \\
 \pi(\chi_{1}) \, \pi(\chi_{2}) & = \cocy_s(\chi_{1},\chi_{2}) \,\pi(\chi_{2}) \, \pi(\chi_{1}), \\
 \pi(\chi)^* & = \cocy(\chi,\chi) \, \pi(-\chi), \\
 \pi(\chi_{1})^{*} \, \pi(\chi_{2})^{*} & = \overline{\cocy(\chi_{2},\chi_{1})} \, \pi(\chi_{1}+\chi_{2})^{*}.
\end{align*}
The \emph{short-time Fourier transform} with respect to a given function $g\in L^{2}(G)$ is the operator
\begin{equation} \label{eq:def-STFT} \mathcal{V}_{g} : L^{2}(G)\to L^{2}(G\times\ghat), \ \mathcal{V}_{g}{f}(\chi) = \langle f, \pi(\chi) g\rangle, \ \chi\in G\times\ghat. \end{equation}
The operator $\mathcal{V}_{g}^{*} \circ \mathcal{V}_{g}$ is a multiple of the identity. Specifically, for all $f_{1},f_{2},g,h\in L^{2}(G)$
\begin{align} \label{eq:STFT} \langle f_{1},f_{2} \rangle \, \langle h, g\rangle & = \langle \mathcal{V}_{g}f_{1}, \mathcal{V}_{h}f_{2}\rangle \\
& = \int_{G\times\ghat} \langle f, \pi(\chi) g \rangle \, \langle \pi(\chi) h , f_{2}\rangle \, d\mu_{G\times\ghat}(\chi). \nonumber \end{align}

The symbol $\Lambda$ will always denote a closed subgroup of the time-frequency plane $G\times\ghat$. The induced topology and group action on $\Lambda$ and on the quotient group $(G\times\ghat)/\Lambda$ turn those into LCA groups as well, and can therefore be equipped with their own Haar measures. If the measures on $G$, $\ghat$ and $\Lambda$ are fixed, then the Haar measure $\mu_{(G\times\ghat)/\Lambda}$ on the quotient group $(G\times\ghat)/\Lambda$ can be uniquely scaled such that, for all $f\in L^{1}(G\times\ghat)$,
\begin{equation} \label{eq:2503a} \int_{G\times\ghat} f(\chi) \, d\mu_{G\times\ghat}(\chi) = \int_{(G\times\ghat)/\Lambda} \int_{\Lambda} f(\chi+\lambda) \, d\mu_{\Lambda}(\lambda) \, d\mu_{(G\times\ghat)/\Lambda}(\dot{\chi}) \ \ \dot{\chi} = \chi + \Lambda, \ \chi\in G\times\ghat.\end{equation}
If \eqref{eq:2503a} holds we say that $\mu_{G\times\ghat}$, $\mu_{\Lambda}$ and $\mu_{(G\times\ghat)/\Lambda}$ are \emph{canonically related} and the equality in \eqref{eq:2503a} is called \emph{Weil's formula}. We always choose the measures $\mu_{G\times\ghat}$, $\mu_{\Lambda}$ and $\mu_{(G\times\ghat)/\Lambda}$ in this way. For more on this, see \cite[p.87-88]{rest00} and \cite[Theorem~3.4.6]{rest00}.
With the uniquely determined measure $\mu_{(G\times\ghat)/\Lambda}$ we define the \emph{size} or the \emph{covolume} of $\Lambda$, by $\s(\Lambda)=\int_{(G\times\ghat)/\Lambda} 1 \, d\mu_{(G\times\ghat)/\Lambda}$. Note that $\s(\Lambda)$ is finite if and only if $\Lambda$ is a \emph{co-compact} subgroup of $G\times\ghat$, i.e., the quotient group $(G\times\ghat)/\Lambda$ is compact. If $\Lambda$ is discrete, co-compact (hence a lattice), and equipped with the counting measure, then $\s(\Lambda)$ is exactly the measure of any of its fundamental domains. The \emph{adjoint} group of $\Lambda$ is the closed subgroup of $G\times\ghat$ given by
\begin{align*} \Lambda^{\circ} 
& = \{ \chi \in G\times\ghat \, : \, \cocy_{s}(\chi,\lambda) = 1 \ \ \textnormal{for all} \ \ \lambda\in\Lambda  \ \}. \end{align*}

For any closed subgroup $\Lambda$ one has $(\Lambda^{\circ})^{\circ}=\Lambda$ and $\widehat{\Lambda^{\circ}}\cong (G\times\ghat)/\Lambda$. Given these identifications, we take the Haar measure $\mu_{\Lambda^{\circ}}$ on $\Lambda^{\circ}$ such that the Fourier inversion between functions on $\Lambda^{\circ}$ and $(G\times\ghat)/\Lambda$ holds. This unique measure on $\Lambda^{\circ}$ is called the \emph{orthogonal measure} relative to $\mu_{\Lambda}$ \cite[Definition 5.5.1]{rest00}. We now choose the Haar measure on $(G\times\ghat)/\Lambda^{\circ}$ such that the measures $\mu_{(G\times\ghat)}$, $\mu_{\Lambda^{\circ}}$ and $\mu_{(G\times\ghat)/\Lambda^{\circ}}$ are canonically related. This ensures that also the Fourier inversion formula between functions on $\Lambda$ and $(G\times\ghat)/\Lambda^{\circ}$ holds \cite[Theorem 5.5.12]{rest00}. 
\begin{remark} \label{rem:orth-measure-for-cocompact}
For a closed subgroup $\Lambda$ with measure $\mu_{\Lambda}$ it is in general difficult to say more about the orthogonal measure on $\mu_{\Lambda^{\circ}}$ on $\Lambda^{\circ}$. However, if the quotient group $(G\times\ghat)/\Lambda$ is compact (equivalently $\Lambda^{\circ}$ is discrete), then the orthogonal measure on $\Lambda^{\circ}$ satisfies
\begin{equation} \label{eq:cocompact-orth-measure} \int_{\Lambda^{\circ}} f(\lambda^{\circ}) \, d\mu_{\Lambda^{\circ}}(\lambda^{\circ}) = \frac{1}{\s(\Lambda)} \sum_{\lambda^{\circ}\in\Lambda^{\circ}} f(\lambda^{\circ}) \ \ \textnormal{for all} \ \ f\in \ell^{1}(\Lambda^{\circ}).\end{equation}
\end{remark}

We will be considering functions on phase space. For such functions the Fourier transform is the operator
\begin{align*}
& \mathcal{F} : L^{1}(G\times\ghat) \to C_{0}(\ghat\times G), \\
& \quad (\mathcal{F} F )(\omega,x) = \int_{G\times\ghat} F(t,\xi) \, \underbrace{\cocy_{s}\big( (t,\xi), (-x,\omega) \big)}_{=\overline{\omega(t)} \, \overline{\xi(x)}} \, d\mu_{G\times\ghat}(t,\xi), \ (\omega,x) \in \ghat\times G.\end{align*}
Since phase space is a self-dual group it is convenient to use the \emph{symplectic Fourier transform}
\begin{align*}
& \mathcal{F}_{s} : L^{1}(G\times\ghat) \to C_{0}(G\times\ghat), \\
& \quad (\mathcal{F}_{s} F )(x,\omega) = \int_{G\times\ghat} F(t,\xi) \, \underbrace{ \cocy_{s}\big( (t,\xi),(x,\omega)\big)}_{=\overline{\omega(t)} \, \xi(x)} \, d\mu_{G\times\ghat}(t,\xi), \ (x,\omega) \in G\times\ghat.\end{align*}
The symplectic Fourier transform $\mathcal{F}_{s}$ is related to the usual Fourier transform by the obvious relation $\mathcal{F}_{s} F(x,\omega) = \mathcal{F} F(\omega,-x)$ for all $(x,\omega)\in G\times\ghat$. Furthermore, the domain and range of the symplectic Fourier transform are functions over the same group, $G\times\ghat$, and $\mathcal{F}_{s}^{-1} = \mathcal{F}_{s}$.

For more on harmonic analysis on locally compact abelian groups see the book by Reiter and Stegeman \cite{rest00}. Other books are the one by Folland \cite{fo16} and Hewitt and Ross \cite{hero70,hero79}.

\subsection{The Feichtinger algebra}
\label{sec:fei-alg}
For any LCA group $G$ the \emph{Feichtinger algebra}  $\SO(G)$ \cite{fe81-2,ja19,lo80} (sometimes denoted by $\mathbf{M}^{1}(G)$) is the set of functions given by 
%
\[ \SO(G) = \big\{ f\in L^{2}(G) \, : \, \mathcal{V}_{f}f \in L^{1}(G\times\ghat) \big\} 
.\]
For the definition of $\mathcal{V}_{f}f$ see \eqref{eq:def-STFT}. Any non-zero function $g\in \SO(G)$ can be used to define a norm on $\SO(G)$, 
\begin{equation} \label{eq:norm} \Vert \cdot \Vert_{\SO(G),g} : \SO(G) \to \R^{+}_{0}, \ \Vert f \Vert_{\SO(G),g} = \Vert \mathcal{V}_{g}f \Vert_{1}. \end{equation}
All norms defined in this way are equivalent \cite[Proposition 4.10]{ja19} and they turn $\SO(G)$ into a Banach space \cite[Theorem 4.12]{ja19}. The usefulness of the Feichtinger algebra $\SO(G)$ lies in the fact that it behaves very much like the Schwartz-Bruhat space $\mathscr{S}(G)$ (also, one has the inclusion $\mathscr{S}(G)\subset \SO(G)$, see \cite[Theorem 9]{fe81-2}). Among its properties, we mention the following.
\begin{lemma}\label{le:s0-properties} \label{le:SO-properties} 
\begin{enumerate}[(i)]
\item All functions in $\SO(G)$ are absolutely integrable, continuous, and vanish at infinity. 
\item If $G$ is discrete, then $(\SO(G), \Vert \cdot \Vert_{\SO})=(\ell^{1}(G), \Vert \cdot \Vert_{1})$.
\item Time-frequency shifts $\pi(\chi)$, $\chi\in G\times\ghat$, are an isometry on $\SO(G)$. The Fourier transform is a continuous bijection from $\SO(G)$ onto $\SO(\ghat)$. Similarly, the symplectic Fourier transform is a continuous bijection on $\SO(G\times\ghat)$.
\item $\SO(G)$ is continuously embedded into $L^{p}(G)$ for all $p\in[1,\infty]$. 
In fact, if $1/p+1/q=1$, then
\[ \Vert f \Vert_{p} \le \Vert g \Vert_{q}^{-1} \Vert f \Vert_{\SO,g} \ \ \text{for all} \ \ f\in \SO(G).\]
\item $\SO(G)$ is a Banach algebra with respect to convolution and point-wise multiplication. 
\item For any closed subgroup $H$ of $G$, the restriction operator $R_{H} : f \mapsto f\big\vert_{H}$ is a bounded and surjective operator from $\SO(G)$ onto $\SO(H)$
\item For any discrete subgroup $H$ of $G$, the restriction operator $R_{H} : f \mapsto f\big\vert_{H}$ is a bounded and surjective operator from $\SO(G)$ onto $\ell^{1}(H)$
\item For any $f,g\in\SO(G)$ the short-time Fourier transform $\mathcal{V}_{g}{f}$ is a function in $\SO(G\times\ghat)$. Also, there exists is a constant $c>0$ such that $\Vert\mathcal{V}_{g}f\Vert_{\SO} \le c \, \Vert f \Vert_{\SO} \, \Vert g \Vert_{\SO}$ for all $f,g\in \SO(G)$. 
\item The Poisson (summation) formula holds pointwise for all functions in $\SO(G)$. For our purposes it is useful to state the Poisson formula for functions on phase space and with the symplectic Fourier transform: for all $F\in \SO(G\times\ghat)$ and any closed subgroup $\Lambda$ in $G\times\ghat$ it holds that
\[ \int_{\Lambda} F(\lambda) \, d\lambda = \int_{\Lambda^{\circ}} \mathcal{F}_{s} F(\lambda^{\circ}) \, d\lambda^{\circ}. \]
If the quotient group $(G\times\ghat)/\Lambda$ is compact, then the Poisson formula takes the form
\[ \int_{\Lambda} F(\lambda) \, d\lambda = \frac{1}{\s(\Lambda)} \sum_{\lambda^{\circ}\in\Lambda^{\circ}} \mathcal{F}_{s} F(\lambda^{\circ}) . \]
\item If $H$ is a discrete abelian group and $G$ is any LCA group, then $f\in\SO(G\times H)$ if and only if $f(\cdot, k)\in \SO(G)$ for each $k\in H$ and $\sum_{k\in H} \Vert f(\cdot, k) \Vert_{\SO(G)} < \infty$. Moreover, $\Vert f \Vert_{\SO(G)\otimes \ell^{1}(H)} = \sum_{k\in H} \Vert f(\cdot, k)\Vert_{\SO(G)}$ is a norm on $\SO(G\times H)$ that is equivalent to the norm on $\SO(G\times H)$ defined by \eqref{eq:norm}.
\end{enumerate}
\end{lemma}
\begin{proof}(i). This follows from \cite[Definition 1]{fe81-2}. Alternatively, see \cite[Lemma 4.19]{ja19}. (ii). see 
\cite[Remark 3]{fe81-2} or \cite[Lemma 4.11]{ja19}. (iii). \cite[Example 5.2(i,iii,v)]{ja19}. (iv). That $\SO$ is continuously embedded into $L^{p}$ follows from the fact that that $\SO(G) = W(\mathcal{F}L^{1},L^{1})$ (\cite[Remark 6]{fe81-2}) together with the inclusions in \cite[Lemma 1.2(iv)]{fe81-1} and the fact that $W(L^{p},L^{p}) = L^{p}$ \cite[Lemma 1.2(i)]{fe81-1}. For the inequality see \cite[Lemma 4.11]{ja19}.  (v). $\SO$ is a Segal algebra (\cite[Theorem 1]{fe81-2}) and any Segal algebra is a convolution algebra \cite[\S 4]{re71}. By (iii) this implies that it is also an algebra under pointwise multiplication. Alternatively, see \cite[Corollary 4.14]{ja19}. (vi). See \cite[Theorem 7.C]{fe81-2} or \cite[Theorem 5.7(ii)]{ja19}. (vii). This follows from (vi) together with (ii). (viii). \cite[Theorem 5.3(ii)]{ja19}. (ix). That the Poisson formula holds for functions in $\SO$ is stated in \cite[Remark 15]{fe81-1}. Alternatively, see \cite[Theorem 5.7(iii), Example 5.11]{ja19}. (x). This follows from \cite[Theorem 7.7]{ja19}.
\end{proof}
\section{Heisenberg modules}
\label{sec:module}

Let $\Lambda$ be a closed subgroup of $G\times\ghat$. For two functions on $\Lambda$ we define the twisted convolution 
\[ F_{1}\,\natural \, F_{2}(\lambda)=\int_{\Lambda} F_{1}(\lambda')\,F_{2}(\lambda-\lambda')\,\cocy(\lambda',\lambda-\lambda') \, d\lambda',\] 
and the twisted involution $F^*(\lambda)=\cocy(\lambda,\lambda)\overline{F(-\lambda)}$.

For functions on $\La^{\circ}$ we define a twisted convolution and involution as before but with the complex conjugate of the cocycle,
\begin{equation} \label{eq:0804a} F_{1} \, \natural \, F_{2}(\lambda^{\circ})=\int_{\Lambda^{\circ}} F_{1}(\lambda^{\circ'}) \, F_{2}(\lambda^{\circ}-\lambda^{\circ'}) \, \overline{\cocy(\lambda^{\circ'},\lambda^{\circ}-\lambda^{\circ'})} \, d\lambda^{\circ'},\end{equation}

\begin{equation} \label{eq:0804b} F^{*}(\lambda^{\circ}) = \overline{c(\lambda^{\circ},\lambda^{\circ}) \, F(-\lambda^{\circ})}.\end{equation}
The reason for the distinction between functions on $\Lambda$ and $\Lambda^{\circ}$ will be made clear in Remark \ref{rem:l-r-module}.

For any closed subgroup $\La$ the space $L^1(\La)$ is an involutive Banach algebra with respect to the twisted convolution and involution, denoted by $L^1(\La,\cocy)$. The twisted group $C^*$-algebra $C^*(\La,\cocy)$ is the enveloping $C^*$-algebra of the twisted group algebra $L^1(\La,\cocy)$. The same is true for $L^{1}(\La^{\circ},\overline{\cocy})$.

In \cite{ri88} Rieffel constructed Hilbert $C^*$-modules over the twisted group $C^*$-algebra $C^*(\Lambda,\cocy)$ based on the Schwartz-Bruhat space $\mathscr{S}(G)$ of $G$ because it behaves well under the Fourier transform  and restriction to (closed) subgroups. Also the Poisson (summation) formula holds for functions in $\mathscr{S}(G)$ and closed subgroups of $G$.

As proposed by Feichtinger in \cite{fe81-2} there is an alternative to the Schwartz-Bruhat space on a locally compact abelian group: the Banach algebra $\SO(G)$ (cf.\ Section \ref{sec:fei-alg}). In \cite{lu09} the Feichtinger algebra $\SO(\mathbb{R}^m)$ has been used to extend the results in \cite{ri88} from smooth noncommutative tori to twisted group algebras for lattices in $\mathbb{R}^{2m}$. In this section we show that this is also possible in the case of closed subgroups $\La$ in $G\times\ghat$, which also is Rieffel's setup in \cite{ri88}. 

We begin with the result that $\SO(\La)$ is a Banach algebra with respect to the twisted convolution from above. To our knowledge this result is new in case $\Lambda$ is a non-discrete and proper subgroup of $G\times\ghat$. In case $\Lambda = G \times \ghat$ and in particular, if $\Lambda = G\times\ghat = \R^{2m}$, the result of Lemma \ref{le:SO-twist-conv-algebra} is mentioned in the last paragraph of \cite{re84}.
\begin{lemma}\label{le:SO-twist-conv-algebra}
For any closed subgroup $\Lambda$ of the time-frequency plane $G\times\ghat$ the space $\SO(\Lambda)$ is an involutive Banach algebra with respect to the twisted convolution 
\[ \natural: \SO(\Lambda) \times \SO(\Lambda) \to \SO(\Lambda), \ F_{1}\,\natural \, F_{2}(\lambda)=\int_{\Lambda} F_{1}(\lambda')\, F_{2}(\lambda-\lambda')\,\cocy(\lambda',\lambda-\lambda') \, d\lambda'\]
and the twisted involution $\phantom{l}^{*} : \SO(\Lambda)\to\SO(\Lambda), \ F^*(\lambda)=\cocy(\lambda,\lambda)\overline{F(-\lambda)}$.
\end{lemma}
The same statement holds for functions on $\Lambda^{\circ}$ with respect to the convolution and involution defined in \eqref{eq:0804a} and \eqref{eq:0804b} with a very similar proof.
\begin{proof}[Proof of Lemma \ref{le:SO-twist-conv-algebra}] If $\Lambda$ is discrete, then $\SO(\Lambda)=\ell^{1}(\Lambda)$ (by Lemma \ref{le:s0-properties}(ii)) and the result follows easily. If $\Lambda$ is not discrete, then we have to work a bit harder. 
Note that the operator 
\[ \Phi:\SO(\Lambda\times\Lambda)\to\SO(\Lambda\times\Lambda), \ \Phi(F)(\lambda_{1},\lambda_{2}) = F(\lambda_{1},\lambda_{2}) \, \cocy(\lambda_{1},\lambda_{2}), \ (\lambda_{1},\lambda_{2})\in \Lambda\times\Lambda,\]
satisfies the assumption of \cite[Example 5.2(vi)]{ja19}. As stated in that reference $\Phi$ is a \emph{multiplication by a second degree character} and a well-defined, linear, and bounded bijection on $\SO(\Lambda\times\Lambda)$.
Similarly, the operator
\[ \alpha : \SO(\Lambda\times\Lambda) \to \SO(\Lambda\times\Lambda) , \ \alpha(F)(\lambda_{1},\lambda_{2}) = F(\lambda_{1},\lambda_{2}-\lambda_{1})\]
is of the form as in \cite[Example 5.2(ii)]{ja19} and therefore a well-defined, linear, and bounded bijection on $\SO(\Lambda\times\Lambda)$.
Furthermore, by \cite[Theorem 5.7(i)]{ja19}
 the operator 
\[ P : \SO(\Lambda\times\Lambda) \to \SO(\Lambda), \ (P F)(\lambda) = \int_{\Lambda} F(\lambda',\lambda) \, d\lambda'\]
is a well-defined, linear, and bounded surjection onto $\SO(\Lambda)$.
Composing these three operators implies that
\[ T = P \circ \alpha \circ \Phi : \SO(\Lambda\times\Lambda) \to \SO(\Lambda), \ T F(\lambda) = \int_{\Lambda} F(\lambda',\lambda-\lambda') \, \cocy(\lambda',\lambda-\lambda') \, d\lambda\]
is a well-defined, linear, and bounded surjection.
In particular, there is a constant $c_{1}$ that only depends on $\Lambda$ such that
\begin{equation} \label{eq:0704a} \Vert T(F) \Vert_{\SO(\Lambda)} \le c_{1} \, \Vert F \Vert_{\SO(\Lambda\times\Lambda)} \ \ \text{for all} \ \ F\in \SO(\Lambda\times\Lambda).\end{equation}
If $F_{1},F_{2}\in\SO(\Lambda)$, then by \cite[Theorem 5.3(i)]{ja19} the function $F_{1}\otimes F_{2} = (\lambda_{1},\lambda_{2})\mapsto F_{1}(\lambda_{1}) \cdot F_{2}(\lambda_{2})$ belongs to $\SO(\Lambda\times\Lambda)$ and there is a constant $c_{2}$ such that 
\begin{equation} 
\label{eq:0704b} \Vert F_{1}\otimes F_{2} \Vert_{\SO(\Lambda\times\Lambda)} \le c_{2} \, \Vert F_{1} \Vert_{\SO(\Lambda)} \, \Vert F_{2} \Vert_{\SO(\Lambda)}. \end{equation}
Observe that
\[ T(F_{1}\otimes F_{2})(\lambda) = \int_{\Lambda} F_{1}(\lambda') \, F_{2}(\lambda-\lambda') \, \cocy(\lambda',\lambda-\lambda') \, d\lambda = F_{1}\,\natural\, F_{2}(\lambda).\]
If we combine \eqref{eq:0704a} and \eqref{eq:0704b}, then we find that there is a constant $c>0$ (that depends only on $\Lambda$) such that
\[ \Vert F_{1} \, \natural \, F_{2} \Vert_{\SO(\Lambda)} = \Vert T (F_{1}\otimes F_{2}) \Vert_{\SO(\Lambda)} \le c \, \Vert F_{1} \Vert_{\SO(\Lambda)} \, \Vert F_{2} \Vert_{\SO(\Lambda)} \ \ \text{for all} \ \ F_{1},F_{2}\in\SO(\Lambda).\]
This shows that $\SO(\Lambda)$ is a Banach algebra under the twisted convolution. Similarly, one can show that the twisted involution is a composition of the operator
\[ \Psi:\SO(\Lambda) \to \SO(\Lambda) , \ \Psi F(\lambda) = \cocy(\lambda,\lambda) \, F(\lambda)\]
and the involution $F(\,\cdot\,) \mapsto \overline{F(-\,\cdot \, )}$. The operator $\Psi$ is, just as $\Phi$, a multiplication by a second degree character and thus an isomorphism on $\SO(\Lambda)$. Also the just described involution is an isomorphism on $\SO(\Lambda)$, see \cite[Example 5.2(viii)]{ja19}.
It is a matter of routine to verify that $(F_{1}\, \natural \, F_{2})^{*} = F_{2}^{*} \, \natural \, F_{1}^{*}$. 
\end{proof}

The time-frequency shifts $\pi(\chi)$ for $\chi = (x,\omega)\in G\times\ghat$ give an irreducible projective representation of $G\times\ghat$ on $L^{2}(G\times\ghat)$. Note that the restriction of this projective representation to a closed subgroup $\Lambda$ of $G\times\ghat$ defines a reducible projective representation of $\Lambda$. We will use that this projective representation of $\Lambda$ is faithful. We include the argument for the sake of completeness.

\begin{lemma}
Suppose $\Lambda$ is a closed subgroup of $G\times\ghat$. Then $\chi\mapsto \pi(\chi)$ is a faithful representation of $\Lambda$. Hence the integrated representation gives a non-degenerate representation of the twisted group algebra $(L^1(\La),\cocy)$ and of $(\SO(\La),\cocy)$.
\end{lemma}
\begin{proof} We adapt the proof for the Euclidean case given in \cite[Proposition 2.6]{gr07-2}.
The desired claim to show is: if $\int_\Lambda a(\lambda)\pi(\lambda)d\lambda=0$ for some $a\in L^\infty(\Lambda)$, then $a=0$. 
Equivalently, we have to prove that 
\[\int_\Lambda a(\lambda)\langle \pi(\lambda)\pi(\chi)f,\pi(\chi)g\rangle \,d\lambda=0 \ \ \text{for all} \ \ f,g\in\SO(G)\] 
implies $a=0$. By $\pi(\chi)^*\pi(\lambda)\pi(\chi)=\cocy_s(\lambda,\chi)\pi(\lambda)$ for $\lambda\in\Lambda$ and $\chi\in G\times\ghat$, the assumption is equivalent to: 
\[
\int_\Lambda a(\lambda)\langle \pi(\lambda)f,g\rangle\cocy_s(\lambda,\chi)\, d\lambda=0
\]
for all $\chi\in G\times\ghat$ and $f,g\in\SO(G)$.
By the uniqueness of the Fourier transform on $\SO(\Lambda)$ we conclude that $a(\lambda)\langle \pi(\lambda)f,g\rangle=0$ for all $\lambda\in \Lambda$ and $f,g\in\SO(G)$. Hence $a=0$ on $\Lambda$, the desired claim. Since $\pi$ is faithful, we get that the integrated representation is non-degenerate representation of $(\SO(\Lambda),\natural)$ and $(L^1(\Lambda),\natural)$, respectively.  
\end{proof}
Given a closed subgroup $\Lambda$ of time time-frequency plane $G\times\ghat$ and its adjoint group $\Lambda^\circ$, we are interested in the relation between the Banach algebras $(\SO(\Lambda),\natural)$ and $(\SO(\Lambda^\circ),\natural)$. We use the integrated Schr\"odinger representation to define the following two Banach algebras
\begin{align*} 
\lmodule & = \big\{ {\bf a} \in \mathsf{B}(L^{2}(G)) \, : \, {\bf{a}} = \int_{\Lambda} a(\lambda) \, \pi(\lambda) \, d\lambda, \ a\in \SO(\Lambda)\big\}, \\
\rmodule & = \big\{ {\bf b} \in \mathsf{B}(L^{2}(G)) \, : \, {\bf{b}} = \int_{\Lambda^{\circ}} b(\lambda^{\circ}) \, \pi(\lambda^{\circ})^{*} \, d\lambda^{\circ}, \ b\in \SO(\Lambda^{\circ})\big\}. \end{align*}
Indeed, the norm $\Vert {\bf{a}} \Vert_{\lmodule} = \Vert  a \Vert_{\SO}$ (where ${\bf a}$ and $a$ are related as in the definition of $\lmodule$) turns $\lmodule$ into an involutive Banach algebra with respect to composition of operators and where the involution is the $L^{2}$-adjoint. Similarly, $\rmodule$ becomes an involutive Banach algebra.

We define \emph{traces} on both $\lmodule$ and $\rmodule$ by
\[ \tr_{\lmodule} : \lmodule \to \C,  \ \tr_{\lmodule} (\mathbf{a}) = a(0) \ , \ \ \ \ \tr_{\rmodule} : \rmodule \to \C,  \ \tr_{\rmodule} (\mathbf{b}) = b(0). \]
These are well-defined as all functions in $\SO$ are continuous (Lemma \ref{le:s0-properties}(i)). Furthermore, $\tr_{\lmodule}$ and $\tr_{\rmodule}$ are bounded operators (this follows from the continuous embedding of $\SO$ into $L^{\infty}$, Lemma \ref{le:s0-properties}(iv)).

\begin{remark} In the definition of $\rmodule$ the measure on $\Lambda^{\circ}$ is the measure that is orthogonal to the measure on $\Lambda$, cf.\ Remark \ref{rem:orth-measure-for-cocompact}. Hence, if $\Lambda$ is a co-compact subgroup of $G\times\ghat$ (e.g., a lattice), then
\[ \rmodule = \big\{ {\bf b} \in \mathsf{B}(L^{2}(G)) \, : \, {\bf{b}} = \frac{1}{s(\Lambda)} \sum_{\lambda^{\circ}\in\Lambda^{\circ}} b(\lambda^{\circ}) \, \pi(\lambda^{\circ})^{*} \, d\lambda^{\circ}, \ b\in \SO(\Lambda^{\circ})\big\}. \]
In that case, still, $\Vert { \bf b} \Vert_{\rmodule} = \Vert b \Vert_{\SO}$ and \emph{not} $\Vert {\bf b} \Vert_{\rmodule} = s(\Lambda)^{-1} \Vert b \Vert_{\SO}$!
\end{remark}

In the remainder of the section we will occasionally refer to results that we establish in Section \ref{sec:mws-gabor} within the realm of time-frequency analysis. This is because some of the proofs here rely on those facts, but also in order to point out the connection between Heisenberg modules and Gabor analysis.

We denote by $C^*(\Lambda,\cocy)$ and $C^*(\Lambda,\overline{\cocy})$ the twisted group $C^*$-algebras of $\Lambda$ and $\Lambda^\circ$, respectively. The choice of the cocycle $\overline{\cocy}$ will become clear in a moment. Rieffel showed in \cite{ri88}:
\begin{theorem}[Rieffel] \label{th:Rieffel}
The twisted group $C^*$-algebras $C^*(\Lambda,\cocy)$ and $C^*(\Lambda,\overline{\cocy})$ are Morita equivalent. 
\end{theorem}
We briefly sketch the main steps in Rieffel's proof since it allows us to introduce the equivalence bimodule whose structure we are mainly interested in this paper. In addition, we choose the pre-Hilbert $C^*$-bimodule as Feichtinger's algebra $\SO(G)$ instead of the Schwartz-Bruhat space in \cite{ri88}.
\begin{proof}[Proof of Theorem \ref{th:Rieffel}]
We let elements of $\lmodule$ act from the left on functions in $L^{2}(G)$ by 
\[ \mathbf{a} \cdot f := \int_{\Lambda} a(\lambda) \pi(\lambda) f \, d\lambda, \ \ f\in L^{2}(G), \ \mathbf{a}\in \lmodule.\]
Operators in $\rmodule$ act from the right on $L^{2}(G)$,
\[ f \cdot \mathbf{b} := \int_{\Lambda^{\circ}} b(\lambda^\circ)\, \pi(\lambda^{\circ})^* f \, d\lambda^{\circ}, \ \ f\in L^{2}(G), \ \mathbf{b}\in \rmodule.\]
We define $\lmodule$- and $\rmodule$-valued inner products in the following way:
\begin{align*} \lhs{\,\cdot\,}{\,\cdot\,} & : \SO(G)\times \SO(G) \to \lmodule, \ \lhs{f}{g} = \int_{\Lambda} \langle f, \pi(\lambda) g \rangle \, \pi(\lambda) \, d\lambda,\\
\rhs{\,\cdot\,}{\,\cdot\,} & :\SO(G)\times\SO(G)\to \rmodule, \ \rhs{f}{g} = \int_{\Lambda^{\circ}} \langle g, \pi(\lambda^{\circ})^{*}f\rangle \,\pi(\lambda^{\circ})^{*} \, d\lambda^{\circ}. \end{align*}
That these are well-defined follows from Lemma \ref{le:s0-properties}(vi)+(viii).
In terms of the left and right action the equality in \eqref{eq:figa} (for $\NumDim=\NumWin=1$), which is the generalization of \eqref{eq:janssen-intro}, yields 
\begin{equation} \label{eq:2302c} \lhs{f}{g} \cdot h = f\cdot\rhs{g}{h} \ \ \text{for all} \ \ f,g,h\in\SO(G).\end{equation}
In other words, the $\lmodule$- and $\rmodule$-valued inner products satisfy the associativity condition. The completion of $\SO(G)$ with respect to $\|g\|_\Lambda:=\|\lhs{g}{g}\|^{1/2}_{\textnormal{op},L^{2}}$ is a full left Hilbert $A$-module and a full right $B$-module for $\|g\|_{\Lambda^\circ}:=\|\rhs{g}{g}\|^{1/2}_{\textnormal{op},L^{2}}$. We denote this equivalence bimodule by $\mathcal{E}$. 
\end{proof}
Observe that
\[ \tr_{\lmodule} \big( \lhs{f}{g}\big) = \tr_{\rmodule} \big(\rhs{g}{f}\big) = \langle f, g\rangle \ \ \text{for all} \ \ f,g\in\SO(G). \]

The equivalence bimodule $\mathcal{E}$ between $C^*(\Lambda,\cocy)$ and $C^*(\Lambda,\overline{\cocy})$ is referred to as a \emph{Heisenberg} module as it is related to the representation theory of the Heisenberg group.
\begin{remark}\label{rem:l-r-module}
The distinction between the twisted convolutions of functions in $\SO(\Lambda)$ and on $\SO(\Lambda^{\circ})$ is that they make $\lmodule$ a \emph{left} module over $\SO(G)$ and $\rmodule$ a \emph{right} module over $\SO(G)$. Specifically,
\[ \mathbf{a}_{1} \cdot (\mathbf{a}_{2} \cdot f) = (\mathbf{a}_{1} \cdot \mathbf{a}_{2}) \cdot f, \ \ (f\cdot \mathbf{b}_{1}) \cdot \mathbf{b}_{2} = f \cdot (\mathbf{b}_{1}\cdot \mathbf{b}_{2}).\]
\end{remark}
On the Hilbert $C^*$-module $\mathcal{E}$ (see the proof of Theorem \ref{th:Rieffel} for the definition of $\mathcal{E}$) we may define (left) rank-one operators $S_{g,h}$ by $f\mapsto \lhs{f}{g} \cdot h$ for $f,g,h\in\mathcal{E}$. Concretely, we have that 
\[ S_{g,h}f=\int_\La \langle f,\pi(\la)g\rangle\pi(\la)h\,d\mu_\La(\la). \]
This operator is also called a mixed frame operator of the Gabor systems $\{\pi(\la)g\}_{\la\in\La}$ and $\{\pi(\la)h\}_{\la\in\La}$. This observation is the very reason for the link between Heisenberg modules and Gabor analysis \cite{lu09}.
In Gabor analysis the associativity condition 
\[ \lhs{f}{g} \cdot h = f\cdot\rhs{g}{h} \ \ \text{for} \ \ f,g,h\in\SO(G)\]
is referred to as the \emph{Janssen representation} of the operator $S_{g,h}$ and lies at the heart of the duality theory for Gabor systems that we expand on in Section \ref{sec:mws-gabor}. Properties of Gabor systems enter also naturally in the discussion of the Hilbert $C^*$-module $\mathcal{E}$. Specifically, the notion of a \emph{Bessel system} used in frame theory has a link with the \emph{module norm} used in Hilbert $C^{*}$-module theory.


\begin{lemma} \label{le:module-norm-and-bessel}
Suppose $\Lambda$ is a closed subgroup of $G\times\ghat$. For a function $g\in \SO(G)$ the module norm $\|g\|_\Lambda$ equals the square root of the optimal Bessel bound of the Gabor system $\{\pi(\lambda)g\}_{\lambda\in\Lambda}$.
\end{lemma}
\begin{proof}
  Recall that $\lhs{g}{g}$ is a positive element in $\lmodule$ for any $g\in\mathcal{E}$. Hence 
  \[\|g\|_\La^2=\sup\Big\{ \frac{\langle\lhs{g}{g}f,f\rangle}{\|f\|_2^2} \, : \, f\in \SO(G)\backslash\{0\} \Big\},\] 
  and by the associativity condition we have that 
  \[\langle\lhs{g}{g}f,f\rangle=\int_{\Lac}|\langle f,\pi(\lac)g\rangle|^2 \, d\mu_{\Lac}(\lac).\] 
  This implies that $\|g\|_\La<\infty$ if and only if there exists a constant $B>0$ such that
  \begin{equation} \label{eq:1505b}\int_{\Lac}|\langle f,\pi(\lac)g\rangle|^2\, d\mu_{\Lac}(\lac)\le B \, \|f\|_2^2 \ \ \text{for all} \ \ f\in \SO(G).
  \end{equation}
  I.e., $\{\pi(\lambda^{\circ})g\}_{\lambda^{\circ}\in\Lambda^{\circ}}$ is a Bessel system (relate this to Remark \ref{rem:bessel-vs-finite-module-norm} and its preceding paragraph). Moreover, the lowest possible value of $B$ in \eqref{eq:1505b} is exactly $\Vert g \Vert_{\La}^{2}$. The same arguments relate the Bessel bound of the system $\{\pi(\la)g \}_{\la\in\La}$ and $\Vert g \Vert_{\Lac}$. The equality $\|g\|_\Lambda=\|g\|_{\Lambda^\circ}$ \cite[Proposition 3.1]{ri79-1} implies the desired relation between $\Vert g \Vert_{\Lambda}$ and the optimal Bessel bound of the Gabor system $\{\pi(\lambda)g\}_{\la\in\La}$.
\end{proof}
\begin{remark} \label{rem:bessel-duality} The proof of Lemma \ref{le:module-norm-and-bessel} implies the \emph{Bessel duality} for Gabor systems: the two Gabor systems $\{\pi(\lambda) g\}_{\lambda}$ and $\{\pi(\lambda^{\circ})g\}_{\lambda^{\circ}\in\Lambda^{\circ}}$ have the same Bessel bound (if $\Lac$ is equipped with the correctly normalized measure, see Remark \ref{rem:orth-measure-for-cocompact} and its preceding paragraph). We will give a second proof of the Bessel duality within the realm of time-frequency analysis, see Proposition \ref{pr:bessel-duality}. In the ``classical '' Gabor setting $G=\R$ and $\La = \alpha\Z\times\beta \Z$ this result was established in the time-frequency community independently by Daubechies, Landau, Landau \cite[Theorem 4.3]{dalala95}, Janssen \cite[Proposition 3.1]{ja95}, and Ron and Shen \cite[Theorem 2.2]{rosh97}. Here this statement follows by use of the connection of Gabor systems to Rieffel's theory on the module norm of equivalence bimodules \cite{ri79-1}. 
\end{remark}
\begin{definition}
 We say that the functions $\{g_1,...,g_n\}$ are a (finite) \emph{standard module frame} \cite{frla02} for the Hilbert $C^*(\La,c)$-module $\mathcal{E}$ if there exist positive constants $A$ and $B$ such that 
\[A\,\lhs{f}{f}\le\sum_{j=1}^n\lhs{f}{g_j}\lhs{g_j}{f} \le\, B\lhs{f}{f} \ \ \text{for all} \ \ f\in\mathcal{E}.\]
\end{definition}
Note that this means that $\mathcal{E}$ is finitely generated and projective. 
Since $\mathcal{E}$ is an equivalence bimodule for $\lmodule$ and $\rmodule$, then one has that $\{g_1,...,g_n\}$
is a standard module frame for $\mathcal{E}$ if and only if 
\[A\,I\le \rhs{g_1}{g_1}+\cdots+\rhs{g_n}{g_n}\le B\,I, \]
where we have used the associativity condition to rewrite the standard module frame condition. A different way to look at the standard module frame condition is by taking the trace $\tr_{\lmodule}$ and note that it then becomes 
\[A\,\|f\|_2^2\le\sum_{j=1}^n\int_{\Lac}|\langle f,\pi(\mu)g_j\rangle|^2 d\mu_{\Lac}(\lac)\le B\,\|f\|_2^2 \ \ \text{for all} \ \ f\in L^2(G),\]
which is the well-known condition for the system $\{\pi(\la)g_1\}_{\la\in\La}\cup\cdots\cup\{\pi(\la)g_n\}_{\la\in\La}$ being a multi-window Gabor frame for $L^2(G)$ (cf.\ Section \ref{sec:mws-gabor}, Definition \ref{def:frame}).

One of our main results is a characterization of those Heisenberg modules that are finitely generated and projective. 
\begin{theorem}\label{th:S0-fin-gen}
$\SO(G)$ is a finitely generated projective $\lmodule$-module if and only if $\Lambda^\circ$ is a discrete subgroup (equivalently, the quotient grop $(G\times\ghat)/\La$ is compact).		
\end{theorem}
\begin{remark}
In Theorem \ref{th:existence-mwsframes-rn} we give an elementary proof of this result in case $G=\R^{m}$.
\end{remark}
Let us rephrase the statement of Theorem \ref{th:S0-fin-gen} in terms of multi-window Gabor frames.
\begin{theorem} \label{th:exist}
  Let $\La$ be a closed subgroup of $G\times\ghat$. There exist functions $g_1,\ldots,g_n$ in $\SO(G)$ such that $\{\pi(\la)g_1\}_{\la\in\La}\cup\ldots\cup\{\pi(\la)g_n\}_{\la\in\La}$ is a multi-window Gabor frame for $L^2(G)$ if and only if $\Lac$ is discrete (equivalently, $(G\times\ghat)/\La$ is compact).
\end{theorem}
\begin{remark} Because \emph{super} Gabor systems are a special case of Gabor systems (see Definition \ref{def:mwg} and the consequent remarks) Theorem \ref{th:exist} shows that there exists multi-window super Gabor frames for $L^{2}(G\times\Z_{\NumDim})$ with respect to any closed subgroup $\Lambda$ of the time-frequency plane $G\times\ghat$ where the quotient group $(G\times\ghat)/\La$ is compact and $\NumDim<\infty$.
\end{remark}
\begin{proof}[Proof of Theorem \ref{th:S0-fin-gen}]
$(\Rightarrow)$ Suppose $\SO(G)$ is projective and finitely generated. Then there exist $g_1,\ldots,g_n$ in $\SO(G)$ such that 
\[f=\lhs{f}{g_1}g_1+\ldots+\lhs{f}{g_n}g_n \ \ \text{for all} \ \ f\in\SO(G).\]
Equivalently, we have 
\[\lhs{f}{f}=\lhs{f}{g_1}\cdot \lhs{g_1}{f}+\ldots+\lhs{f}{g_n}\cdot\lhs{g_n}{f} \ \ \text{for all} \ \ f\in\SO(G).\]
If we take take the trace, $\tr_{\lmodule}$, of the preceding equality, then we obtain the equality 
\[\langle f,f\rangle=\sum_{j=1}^n\int_{\La}|\langle f,\pi(\la)g_j\rangle|^2\,d\mu_\La(\La) \ \ \text{for all} \ \ f\in\SO(G).\]
Hence the functions $g_1,\ldots,g_n$ generate a tight multi-window Gabor frame for $L^2(G)$. By \cite[Theorem 5.1]{jale16-2} (which is Lemma \ref{le:nec-cond}(i) here) this statement implies that $(G\times\ghat)/\La$ is compact. Equivalently, $\Lac$ is discrete.
\\~\\
$(\Leftarrow)$
Suppose $\Lac$ is discrete. Then $\SO(\Lac)=\ell^1(\Lac)$ and thus $C^*(\Lac,\overline{\cocy})$ is unital. Following the proof of \cite[Proposition 2.9]{ri88} we see that the right ideal $\{\rhs{f}{g}:\,f,g\in\SO(G)\}$ is a dense ideal in $C^*(\Lac,\overline{\cocy})$ and by the first part of the proof of \cite[Proposition 2.1]{ri81} the range of $\rhs{.}{.}$ must contain the identity element of $C^*(\Lac,\overline{\cocy})$. Hence there exist $g_1,...,g_n$ in $\mathcal{E}$ such that 
\[f=\lhs{f}{g_1}g_1+\cdots+\lhs{f}{g_n}g_n \ \ \text{for all} \ \ f\in\mathcal{E}.\]
This implies that $\mathcal{E}$ is a finitely generated projective $C^*(\La,c)$-module. Note that  $\rmodule$ is inverse-closed in $C^*(\Lambda^\circ,\overline{\cocy})$ by \cite{grle04}, since $\Lac$ is discrete. By  \cite[Proposition 3.7]{ri88} we deduce that $\SO(G)$ is a finitely generated projective $\lmodule$-module.
 \end{proof}

\subsection{Extending the module theory to matrices}
\label{sec:module-mws}

In this section we generalize the definition of the $\lmodule$- and $\rmodule$-valued inner products to take on values in the matrix algebra $\textnormal{M}_{\NumDim\cdot\NumWin}(\lmodule)$ and $\textnormal{M}_{\NumDim\cdot\NumWin}(\rmodule)$, respectively, where $\NumDim,\NumWin\in\N$. This extension will exactly realize the multi-window super Gabor system theory that we establish in Section \ref{sec:mws-gabor}. 

We need to fix some notation: For $\NumDim\in\N$ we let $\Z_{\NumDim}$ denote the group $\Z/\NumDim\Z$. We will work with functions in $L^{2}$ and $\SO$ over the group $G\times\Z_{\NumDim}\times\Z_{\NumWin}$. Since $\NumDim$ and $\NumWin$ is finite, these functions can be thought of as matrix valued functions, $L^{2}(G; \C^{\NumDim\times\NumWin})$, $\SO(G; \C^{\NumDim\times\NumWin})$. For $f\in L^{2}(G\times\Z_{\NumDim}\times\Z_{\NumWin})$ we define $f_{k,j} := f(\,\cdot\, ,k,j)$. Also, for $k\in\Z_{\NumDim}$, we let $f_{
k,\bullet}$ be the function in $L^{2}(G\times\Z_{\NumWin})$ given by $(f
_{k,j})_{j\in\Z_{\NumWin}}$.

For functions in $\SO(G\times\Z_{\NumDim}\times\Z_{\NumWin})$ we define the $\textnormal{M}_{\NumDim\cdot\NumWin}(\lmodule)$-valued inner-product
\begin{align*} & \mlhs{\,\cdot\,}{\,\cdot\,} : \SO(G\times\Z_{\NumDim}\times\Z_{\NumWin})\times \SO(G\times\Z_{\NumDim}\times\Z_{\NumWin}) \to \textnormal{M}_{\NumDim\cdot\NumWin}(\lmodule)\end{align*}
which maps two functions $f$ and $g$ into the block diagonal matrix
\[ \mlhs{f}{g} = \textnormal{diag}(\underbrace{\mathbf{A}, \mathbf{A}, \ldots, \mathbf{A}}_{\textnormal{repeated $\NumDim$-times}}),\]
where $\mathbf{A}$ is the $\lmodule$-valued $\NumWin\times\NumWin$-matrix given by
\begin{align*}
& \mathbf{A} = \mathlarger{\sum}_{k\in\Z_{\NumDim}} \begin{bmatrix} 
\lhs{f_{k,1}}{g_{k,1}} & \lhs{f_{k,1}}{g_{k,2}} & \cdots & \lhs{f_{k,1}}{g_{k,n}} \\
\lhs{f_{k,2}}{g_{k,1}} & \lhs{f_{k,2}}{g_{k,2}} & \cdots & \lhs{f_{k,2}}{g_{k,n}} \\
\vdots & \vdots & \ddots & \vdots \\
\lhs{f_{k,\NumWin}}{g_{k,1}} & \lhs{f_{k,\NumWin}}{g_{k,2}} & \cdots & \lhs{f_{k,\NumWin}}{g_{k,n}} 
\end{bmatrix}.\end{align*}
The left action that $\mlhs{f}{g}$ has on a function $h\in L^{2}(G\times\Z_{\NumDim}\times\Z_{\NumWin})$ can be represented as a matrix-vector multiplication. We define

\[ \mlhs{f}{g}\cdot h \equiv \begin{bmatrix} ( 
\mlhs{f}{g}\cdot h )_{1,\bullet} \\ 
( \mlhs{f}{g}\cdot h )_{2,\bullet} \\
\vdots \\
( \mlhs{f}{g}\cdot h )_{\NumDim,\bullet}\end{bmatrix}:= 
\def\arraystretch{1.2} \begin{bmatrix} \multicolumn{1}{c}{\mathbf{A}} & \mathbf{0} & \cdots & \mathbf{0} \\ 
\multicolumn{1}{c}{\mathbf{0}} & \mathbf{A} & \cdots & \mathbf{0} \\
{\vdots} & \vdots & \ddots & \vdots \\
\multicolumn{1}{c}{\mathbf{0}} & \mathbf{0} & \cdots & {\mathbf{A}} \end{bmatrix} \cdot
\def\arraystretch{1.2}
\begin{bmatrix} h_{1,\bullet} \\ h_{2,\bullet} \\ \vdots \\  h_{\NumDim,\bullet} \end{bmatrix} = \begin{bmatrix} \mathbf{A} \cdot h_{1,\bullet} \\ \mathbf{A} \cdot h_{2,\bullet} \\ \vdots \\ \mathbf{A} \cdot h_{\NumDim,\bullet} \end{bmatrix},\]
Hence for all $f,g\in \SO(G\times\Z_{\NumDim}\times\Z_{\NumWin})$ and $h\in L^{2}(G\times\Z_{\NumDim}\times\Z_{\NumWin})$ we have 
\begin{equation} \label{eq:2302a} \big(\mlhs{f}{g}\cdot h\big)_{k,j} = \displaystyle\sum_{\substack{k'\in\Z_{\NumDim} \\ j'\in\Z_{\NumWin}}} \lhs{f_{k',j}}{g_{k',j'}}  \cdot h_{k,j'} \ \ \text{for all} \ \ (k,j)\in \Z_{\NumDim}\times\Z_{\NumWin}.\end{equation}

The $\rmodule$-valued inner product is generalized as follows: 
\[ \mrhs{\,\cdot\,}{\,\cdot\,} : \SO(G\times\Z_{\NumDim}\times\Z_{\NumWin})\times \SO(G\times\Z_{\NumDim}\times\Z_{\NumWin}) \to \textnormal{M}_{\NumDim\cdot\NumWin}(\rmodule)\]
is the map that takes two functions $f$ and $g$ into the diagonal block matrix 
\[ \mrhs{f}{g} = \textnormal{diag}(\underbrace{\mathbf{B}, \mathbf{B}, \ldots, \mathbf{B}}_{\textnormal{repeated $\NumWin$-times}}),\]
where $\mathbf{B}$ is the $\rmodule$-valued $\NumDim\times\NumDim$ matrix
\[ \mathbf{B} = \mathlarger{\sum}_{j\in\Z_{\NumWin}} \begin{bmatrix} 
\rhs{f_{1,j}}{g_{1,j}} & \rhs{f_{1,j}}{g_{2,j}} & \ldots & \rhs{f_{1,j}}{g_{\NumDim,j}} \\ 
\rhs{f_{2,j}}{g_{1,j}} & \rhs{f_{2,j}}{g_{2,j}} & \ldots & \rhs{f_{2,j}}{g_{\NumDim,j}} \\
\vdots & \vdots & \ddots & \vdots \\
\rhs{f_{\NumDim,j}}{g_{1,j}} & \rhs{f_{\NumDim,j}}{g_{2,j}} & \ldots & \rhs{f_{\NumDim,j}}{g_{\NumDim,j}}
\end{bmatrix}.\]
The right action that $\mrhs{f}{g}$ has on a function $h\in L^{2}(G\times\Z_{\NumDim}\times\Z_{\NumWin})$ can be realized as a vector-matrix product
\begin{align*} h \cdot \mrhs{f}{g} & \equiv \begin{bmatrix} (h \cdot \mrhs{f}{g})_{\bullet,1} & (h \cdot \mrhs{f}{g})_{\bullet,2} & \ldots & (h \cdot \mrhs{f}{g})_{\bullet,\NumWin}\end{bmatrix} \\
& := \begin{bmatrix} h_{\bullet,1} & h_{\bullet,2} & \ldots & h_{\bullet,\NumWin}\end{bmatrix} \cdot \def\arraystretch{1.2} \begin{bmatrix} \multicolumn{1}{c}{\mathbf{B}} & \mathbf{0} & \cdots & \mathbf{0} \\ 
\multicolumn{1}{c}{\mathbf{0}} & \mathbf{B} & \cdots & \mathbf{0} \\
{\vdots} & \vdots & \ddots & \vdots \\
\multicolumn{1}{c}{\mathbf{0}} & \mathbf{0} & \cdots & \mathbf{B}  \end{bmatrix} \\
& = \begin{bmatrix}  h_{\bullet,1} \cdot \textbf{B} &  h_{\bullet,2} \cdot \textbf{B} & \ldots & h_{\bullet,\NumWin} \cdot \textbf{B} \end{bmatrix} .\end{align*}
That is, for all $f,g\in \SO(G\times\Z_{\NumDim}\times\Z_{\NumWin})$ and $h\in L^{2}(G\times\Z_{\NumDim}\times\Z_{\NumWin})$
\begin{equation} \label{eq:2302b} (h \cdot \mrhs{f}{g})_{k,j} = \displaystyle\sum_{\substack{k'\in\Z_{\NumDim}\\j'\in \Z_{\NumWin}}} h_{k',j} \cdot \rhs{f_{k',j'}}{g_{k,j'}}, \ \ (k,j)\in \Z_{\NumDim}\times\Z_{\NumWin}.\end{equation}

By use of \eqref{eq:2302c} it is immediate from \eqref{eq:2302a} and \eqref{eq:2302b} that the matrix algebra valued inner-products satisfy Rieffel's associativity condition
\begin{equation} \label{eq:associativity-for-mws} \mlhs{f}{g}\cdot h = f \cdot \mrhs{g}{h} \ \ \text{for all} \ \ f,g,h\in \SO(G\times\Z_{\NumDim}\times\Z_{\NumWin}).\end{equation}

We follow the definition of traces of matrices and define the trace of an element in $\textnormal{M}_{\NumDim\times\NumWin}(\lmodule)$ and $\textnormal{M}_{\NumDim\times\NumWin}(\rmodule)$ to be the properly normalized sum of the trace along their diagonal,
\begin{align*} 
& \tr_{\textnormal{M}(\lmodule)} : \textnormal{M}_{d\cdot n} (\lmodule) \to \C , \ \tr_{\textnormal{M}(\lmodule)}( \mathbf{A} ) = \frac{1}{\NumDim} \sum_{i} \tr_{\lmodule} (\mathbf{A}_{ii}), \ \ \mathbf{A}\in\textnormal{M}_{d\cdot n} (\lmodule),  \\
& \tr_{\textnormal{M}(\rmodule)} : \textnormal{M}_{d\cdot n} (\rmodule) \to \C , \ \tr_{\textnormal{M}(\rmodule)}( \mathbf{B} ) = \frac{1}{\NumWin} \sum_{i} \tr_{\rmodule} (\mathbf{B}_{ii}), \ \ \mathbf{B}\in\textnormal{M}_{d\cdot n} (\rmodule).\end{align*}
The block-structure of $\mlhs{\cdot}{\cdot}$ and $\mrhs{\cdot}{\cdot}$ allows us to calculate their traces easily,
\begin{equation} \tr_{\textnormal{M}(\lmodule)} \big( \mlhs{f}{g} \big)  = \tr_{\textnormal{M}(\rmodule)} \big( \mrhs{g}{f} \big) = \langle f, g\rangle = \sum_{\substack{k\in\Z_{\NumDim} \\ j\in\Z_{\NumWin}}} \langle \mvfun{f}{k}{j},\mvfun{g}{k}{j}\rangle \ \ \text{for all} \ \ f,g\in \SO(G\times\Z_{\NumDim}\times\Z_{\NumWin}). \end{equation}
\begin{remark}\label{rem:besselbound-mws-module}
Similar as in the proof of Theorem \ref{th:Rieffel} we can define the module norm on $\SO(G\times\Z_{\NumDim}\times\Z_{\NumWin})$: \[ \|g\|_\Lambda=\big\Vert \mlhs{g}{g}\big\Vert^{1/2}_{\textnormal{op},L^{2}}, \ \|g\|_{\Lambda^\circ}=\big\Vert\mrhs{g}{g}\big\Vert^{1/2}_{\textnormal{op},L^{2}}.\]
They relate to the Bessel bound of the corresponding multi-window super Gabor system as in Lemma~\ref{le:module-norm-and-bessel}. The equality $\Vert g \Vert_{\Lambda} = \Vert g \Vert_{\Lac}$ established by Rieffel in \cite{ri79-1} exactly states the equivalence shown in Proposition \ref{pr:bessel-duality}. See Remark \ref{rem:bessel-duality} for more on this.
\end{remark}

It is clear that for $\NumDim=\NumWin=1$ the above theory reduces to the situation  in the proof of Theorem \ref{th:Rieffel}.

We are now in the position to formulate the main result of this section: the description of the duality theory for Gabor systems that we establish in Section \ref{sec:mws-gabor} in terms of the $\lmodule$- and $\rmodule$-matrix-valued inner-products (the generalization of Theorem \ref{th:duality-for-mws-module-intro} from the introduction). 
\begin{theorem}[Main result] \label{th:duality-for-mws-module} Let $\Lambda$ be a closed co-compact subgroup of $G\times\ghat$. For any pair of functions $g,h\in\SO(G\times\Z_{\NumDim}\times\Z_{\NumWin})$ the following statements are equivalent.
\begin{enumerate}[(i)]
\item $f = \mlhs{f}{g} \cdot h \ $ for all $\ f\in \SO(G\times\Z_{\NumDim}\times\Z_{\NumWin})$,
\item For all $f\in L^{2}(G\times\Z_{\NumDim})$ and all $k\in\Z_{\NumDim}$
\[ \vvfun{f}{k} = \sum_{j\in\Z_{\NumWin}} \int_{\La} \Big( \sum_{k'\in\Z_{\NumDim}} \big\langle \vvfun{f}{k'}, \pi(\lambda) \mvfun{g}{k'}{j} \big\rangle \Big)  \, \pi(\lambda) \mvfun{h}{k}{j} \, d\la .\]
\item The operator $\mrhs{g}{h}$ is the identity on $L^{2}(G\times\Z_{\NumDim}\times\Z_{\NumWin})$.
\item The $\rmodule$-valued $\NumDim\times\NumDim$ matrix
\[ \mathlarger{\sum}_{j\in\Z_{\NumWin}} \begin{bmatrix} 
\rhs{g_{1,j}}{h_{1,j}} & \rhs{g_{1,j}}{h_{2,j}} & \ldots & \rhs{g_{1,j}}{h_{\NumDim,j}} \\ 
\rhs{g_{2,j}}{h_{1,j}} & \rhs{g_{2,j}}{h_{2,j}} & \ldots & \rhs{g_{2,j}}{h_{\NumDim,j}} \\
\vdots & \vdots & \ddots & \vdots \\
\rhs{g_{\NumDim,j}}{h_{1,j}} & \rhs{g_{\NumDim,j}}{h_{2,j}} & \ldots & \rhs{g_{\NumDim,j}}{h_{\NumDim,j}}
\end{bmatrix} \]
equals the identity matrix.
\item $\{\mvfun{g}{}{j}\}_{j\in\Z_{\NumWin}}$ and $\{\mvfun{h}{}{j}\}_{j\in\Z_{\NumWin}}$ generate dual multi-window super Gabor frames for $L^{2}(G\times\Z_{\NumDim})$ with respect to $\Lambda$.
\item $\{\mvfun{g}{k}{}\}_{k\in\Z_{\NumDim}}$ and $\{\mvfun{h}{k}{}\}_{k\in\Z_{\NumDim}}$ generate dual multi-window super Gabor Riesz sequences for $L^{2}(G\times\Z_{\NumDim})$ with respect to $\Lambda^{\circ}$.
\item The $\lmodule$-valued $\NumWin\times\NumWin$-matrix
\[ \mathlarger{\sum}_{k\in\Z_{\NumDim}} \begin{bmatrix} 
\lhs{g_{k,1}}{h_{k,1}} & \lhs{g_{k,1}}{h_{k,2}} & \cdots & \lhs{g_{k,1}}{h_{k,n}} \\
\lhs{g_{k,2}}{h_{k,1}} & \lhs{g_{k,2}}{h_{k,2}} & \cdots & \lhs{g_{k,2}}{h_{k,n}} \\
\vdots & \vdots & \ddots & \vdots \\
\lhs{g_{k,\NumWin}}{h_{k,1}} & \lhs{g_{k,\NumWin}}{h_{k,2}} & \cdots & \lhs{g_{k,\NumWin}}{h_{k,n}} 
\end{bmatrix}\]
is an idempotent operator from $L^{2}(G\times\Z_{\NumWin})$ onto $V:=\overline{\textnormal{span}}\{ \bigoplus\limits_{j\in\Z_{\NumWin}} \pi(\lambda^{\circ})\mvfun{g}{k}{j} \}_{\lambda^{\circ}\in\La^{\circ},k\in\Z_{\NumDim}}$.
\item The operator $\mlhs{g}{h}$ is an idempotent operator from $L^{2}(G\times\Z_{\NumDim}\times\Z_{\NumWin})$ onto $\bigoplus\limits_{k\in\Z_{\NumDim}} V$. 
\item For all $f\in V$ and all $j\in\Z_{\NumWin}$
\[ \vvfun{f}{j} = \frac{1}{\s(\Lambda)} \sum_{k\in\Z_{\NumDim}} \sum_{\lac\in\Lac} \Big( \sum_{j'\in\Z_{\NumWin}} \big\langle \vvfun{f}{j'} , \pi(\lac) \mvfun{h}{j'}{k} \big\rangle \Big) \, \pi(\lac) \mvfun{g}{j}{k}.\]
\item $f = g \cdot \mrhs{h}{f}$ for all $ f\in \SO(G\times\Z_{\NumDim}\times\Z_{\NumWin})\cap \bigoplus\limits_{k\in\Z_{\NumDim}} V$. 
\end{enumerate}
The closure in (vii) is with respect to the $L^2(G\times\Z_{\NumWin})$-norm. All statements are equivalent to the ones where $g$ and $h$ are interchanged.
\end{theorem}
\begin{proof} 
(i) $\Leftrightarrow$ (ii). The definition of $\mlhs{\cdot}{\cdot}$ is such that $f = \mlhs{f}{g}\cdot h$ is equivalent to the statement that for all $f\in \SO(G\times\Z_{\NumDim}\times\Z_{\NumWin})$ 
\[ \mvfun{f}{k}{j} = \sum_{j'\in\Z_{\NumWin}} \int_{\La} \Big( \sum_{k'\in\Z_{\NumDim}} \big\langle \mvfun{f}{k'}{j}, \pi(\lambda) \mvfun{g}{k'}{j'} \big\rangle \Big)  \, \pi(\lambda) \mvfun{h}{k}{j'} \, d\la \ \ \text{for all} \ \ (k,j)\in\Z_{\NumDim}\times\Z_{\NumWin}. \]
This is equivalent with the equality in (ii) for all $f\in\SO(G\times\Z_{\NumDim})$. By Lemma \ref{le:s0-implies-bessel} the operator 
\begin{align*} & S_{g,h,\La} : L^{2}(G\times\Z_{\NumDim})\to L^{2}(G\times\Z_{\NumDim}), \\
& S_{g,h} f(\,\cdot\,, k) = \sum_{j\in\Z_{\NumWin}} \int_{\La} \Big( \sum_{k'\in\Z_{\NumDim}} \big\langle \vvfun{f}{k'} , \pi(\la) \mvfun{g}{k'}{j} \big\rangle \Big) \pi(\la) \mvfun{h}{k'}{j} \, d\la\end{align*}
is continuous on $L^{2}(G\times\Z_{\NumDim})$. Hence the equality $S_{g,h,\La} f = f$ extends by density from $\SO(G\times\Z_{\NumDim})$ to all of $L^{2}(G\times\Z_{\NumDim})$. \\
(i) $\Leftrightarrow$ (iii). By \eqref{eq:associativity-for-mws} we have 
\[ \mlhs{f}{g}\cdot h = f\cdot \mrhs{g}{h} \ \ \text{for all} \ \ f,g,h\in \SO(G\times\Z_{\NumDim}\times\Z_{\NumWin}).\]
Hence (i) holds if and only if $\mrhs{g}{h}$ is the identity on $\SO(G\times\Z_{\NumDim}\times\Z_{\NumWin})$.
Since $\mrhs{g}{h}$ is a linear and continuous operator on $L^{2}(G\times\Z_{\NumDim}\times\Z_{\NumWin})$ and $\SO$ is dense in $L^{2}$ it is clear that $\mrhs{g}{h}$ must be the identity operator. \\
(iii) $\Leftrightarrow$ (iv). This follows by the block structure of $\mrhs{g}{h}$. \\
(v) $\Leftrightarrow$ (ii). Since $g$ and $h$ belong to $\SO(G\times\Z_{\NumDim}\times\Z_{\NumWin})$ it follows from Lemma \ref{le:s0-implies-bessel} that the frame operator associated to either multi-window super Gabor system generated by $g$ and $h$ and the subgroup $\La$ is continuous on $L^{2}$, i.e., the Gabor systems are Bessel systems. This, together with the equality in (ii) is the definition of what it means for the two multi-window super Gabor systems to be dual frames for $L^{2}(G\times\Z_{\NumDim})$ (see Section \ref{sec:frame-and-riesz}). \\
(v) $\Leftrightarrow$ (vi). As in (v), by Lemma \ref{le:s0-implies-bessel} either multi-window super Gabor system generated by $g$ and $h$ with respect to the subgroup $\Lac$ is a Bessel system. Furthermore Theorem \ref{th:wex-raz}, which is equivalent to the statement in (ii), implies that the two families of functions are biorthogonal. By the theory of Riesz sequences, Lemma \ref{le:riesz-sequence}, we arrive at (vi). The same arguments imply the converse implication.\\
(vi) $\Leftrightarrow$ (ix). This follows by the definition of what it means for the two multi-window super Gabor systems to be dual Riesz sequences for $L^{2}(G\times\Z_{\NumWin})$ (see Section \ref{sec:frame-and-riesz}). \\
(ix) $\Leftrightarrow$ (x). This follows with similar arguments as the equivalence between (i) and (ii). \\
(x) $\Leftrightarrow$ (viii). By \eqref{eq:associativity-for-mws} $(x)$ is equivalent with the equality 
\[ f = \mlhs{g}{h}\cdot f \ \ \text{for all} \ \ f \in \SO(G\times\Z_{\NumDim}\times\Z_{\NumWin}) \cap \bigoplus_{k\in\Z_{\NumDim}} V.\]
Since $\SO$ is dense in $L^{2}$ and $\mlhs{g}{h}$ is a linear and continuous operator on $L^{2}(G\times\Z_{\NumDim}\times\Z_{\NumWin})$ it is clear that this identity extends to all $f\in \bigoplus_{k\in\Z_{\NumDim}} V$. \\
(vii) $\Leftrightarrow$ (viii). This follows by the block-structure of $\mlhs{g}{h}$. \\
Since $\mrhs{g}{h}$ is the identity operator if and only if $\mrhs{h}{g}$ is the identity operator, it is clear from (iii) that all statements are equivalent to the ones where $g$ and $h$ are interchanged.
\end{proof}

Let us single out the important special cases of Theorem \ref{th:duality-for-mws-module}: the case of multi-window Gabor frames ($\NumWin\in\N$, $\NumDim = 1$) and the case of super Gabor frames ($\NumDim\in\N$, $\NumWin = 1$).

\begin{corollary}[The multi-window Gabor frame scenario] \label{cor:duality-multi} Let $\Lambda$ be a closed co-compact subgroup of $G\times\ghat$. For any collection of functions $\{\vvfun{g}{j}\}_{j\in\Z_{\NumWin}}$ and $\{\vvfun{h}{j}\}_{n\in\Z_{\NumWin}}$ in $\SO(G)$ the following statements are equivalent.
\begin{enumerate}[(i)]
\item $f = \sum\limits_{j\in\Z_{\NumWin}} \lhs{f}{\vvfun{g}{j}}\cdot \vvfun{h}{j}$ for all $f\in\SO(G)$.
\item $\sum\limits_{j\in\Z_{\NumWin}} \rhs{\vvfun{g}{j}}{\vvfun{h}{j}}$ is the identity operator on $L^{2}(G)$.
\item $\{\vvfun{g}{j}\}_{j\in\Z_{\NumWin}}$ and $\{\vvfun{h}{j}\}_{j\in\Z_{\NumWin}}$ generate dual multi-window Gabor frames with respect to $\La$ for $L^{2}(G)$.
\item $\{\oplus_{j\in\Z_{\NumWin}} \vvfun{g}{j}\}$ and $\{\oplus_{j\in\Z_{\NumWin}} \vvfun{h}{j}\}$ generate dual super Gabor Riesz sequences with respect to $\La^{\circ}$ for $L^{2}(G\times\Z_{\NumWin})$.
\item the $\lmodule$-valued $\NumWin\times\NumWin$-matrix
\[ \begin{bmatrix} \lhs{\vvfun{g}{1}}{\vvfun{h}{1}} & \lhs{\vvfun{g}{1}}{\vvfun{h}{2}} & \cdots & \lhs{\vvfun{g}{1}}{\vvfun{h}{\NumWin}} \\ 
\lhs{\vvfun{g}{2}}{\vvfun{h}{1}} & \lhs{\vvfun{g}{2}}{\vvfun{h}{2}} & \cdots & \lhs{\vvfun{g}{2}}{\vvfun{h}{\NumWin}} \\
\vdots & \vdots & \ddots & \vdots \\
\lhs{\vvfun{g}{\NumWin}}{\vvfun{h}{1}} & \lhs{\vvfun{g}{\NumWin}}{\vvfun{h}{2}} & \cdots & \lhs{\vvfun{g}{\NumWin}}{\vvfun{h}{\NumWin}}
\end{bmatrix}\]
is an idempotent operator from $L^{2}(G\times\Z_{\NumWin})$ onto $
V := \overline{\textnormal{span}}\big\{ \oplus_{j\in\Z_{\NumWin}} \pi(\lambda^{\circ}) \vvfun{g}{j} \big\}_{\lambda^{\circ}\in\La^{\circ}}$.
\item $\vvfun{f}{j} = \vvfun{g}{j} \cdot \sum\limits_{j'\in \Z_{\NumWin}} \rhs{\vvfun{h}{j'}}{\vvfun{f}{j'}}$, $j\in\Z_{\NumWin}$ for all $f\in \SO(G\times\Z_{\NumWin})\cap V$.
\end{enumerate}
\end{corollary}

\begin{corollary}[The super Gabor frame scenario] \label{cor:duality-super} Let $\Lambda$ be a closed co-compact subgroup of $G\times\ghat$. For any collection of functions $\{\vvfun{g}{k}\}_{k\in\Z_{\NumDim}}$ and $\{\vvfun{h}{k}\}_{k\in\Z_{\NumDim}}$ in $\SO(G)$ the following statements are equivalent.
\begin{enumerate}[(i)]
\item $\vvfun{f}{k} = \sum\limits_{k'\in \Z_{\NumDim}} \lhs{\vvfun{f}{k'}}{\vvfun{g}{k'}}\cdot\vvfun{h}{k}$, $k\in\Z_{\NumDim}$ for all $f\in \SO(G\times\Z_{\NumDim})$.
\item the $\rmodule$-valued $\NumDim\times\NumDim$-matrix
\[ \begin{bmatrix} \rhs{\vvfun{g}{1}}{\vvfun{h}{1}} & \rhs{\vvfun{g}{1}}{\vvfun{h}{2}} & \cdots & \rhs{\vvfun{g}{1}}{\vvfun{h}{\NumDim}} \\ 
\rhs{\vvfun{g}{2}}{\vvfun{h}{1}} & \rhs{\vvfun{g}{2}}{\vvfun{h}{2}} & \cdots & \rhs{\vvfun{g}{2}}{\vvfun{h}{\NumDim}} \\
\vdots & \vdots & \ddots & \vdots \\
\rhs{\vvfun{g}{\NumDim}}{\vvfun{h}{1}} & \rhs{\vvfun{g}{\NumDim}}{\vvfun{h}{2}} & \cdots & \rhs{\vvfun{g}{\NumDim}}{\vvfun{h}{\NumDim}}
\end{bmatrix} \]
is the identity operator on $L^{2}(G\times\Z_{\NumDim})$.
\item $\{\oplus_{k\in\Z_{\NumDim}} \vvfun{g}{k}\}$ and $\{\oplus_{k\in\Z_{\NumDim}} \vvfun{h}{k}\}$ generate dual super Gabor frames with respect to $\La$ for $L^{2}(G\times\Z_{\NumDim})$.
\item $\{\vvfun{g}{k}\}_{k\in\Z_{\NumDim}}$ and $\{\vvfun{h}{k}\}_{k\in\Z_{\NumDim}}$ generate dual multi-window Gabor Riesz sequences with respect to $\Lac$ for $L^{2}(G)$.
\item $\sum\limits_{k\in\Z_{\NumDim}} \lhs{\vvfun{g}{k}}{\vvfun{h}{k}}$ is an idempotent operator from $L^{2}(G)$ onto $V:=\overline{\textnormal{span}}\{ \pi(\lambda^{\circ}) \vvfun{g}{k} \}_{\lac\in\Lac,k\in\Z_{\NumDim}}$.
\item $f = \sum\limits_{k\in\Z_{\NumDim}} \vvfun{g}{k}\cdot\rhs{\vvfun{h}{k}}{f}$ for all $f\in\SO(G)\cap V$.
\end{enumerate}
\end{corollary}

\section{Gabor systems}
\label{sec:mws-gabor}

Herein after $\NumDim$ and $\NumWin$ are variables that can take values in the natural numbers $\N$ or they can take the value $\infty$. 
If $\NumDim\ne \infty$ then $\Z_{\NumDim}$ denotes the finite group $\{0,1,\ldots, \NumDim-1\}$ under addition modulo $\NumDim$ (as in Section \ref{sec:module-mws}) and if $\NumDim=\infty$ then $\Z_{\infty}$ denotes the integers $\Z$. 

We wish to investigate the properties of the following collection of functions.
\begin{definition} \label{def:mwg} Let $\Lambda$ be a closed subgroup of the time-frequency plane $G\times\ghat$ and let $g$ be a function in $ L^{2}(G\times\Z_{\NumWin})$. The $\NumWin$-\emph{multi-window Gabor system} generated by the collection of functions $\{ g(\, \cdot \, , j) \, : \, j\in \Z_{\NumWin}\,\}$ and the subgroup $\Lambda$ is the set of functions in $L^{2}(G)$ given by
\begin{equation} \label{eq:mw-system}\big\{ \pi(\lambda) g(\,\cdot\,, j ) \, : \, \lambda\in \Lambda, \, j\in\Z_{\NumWin} \, \big\}.\end{equation}
If we consider the \emph{special case} of an $\NumWin$-multi-window Gabor system in $L^{2}(G\times\Z_{\NumDim})$ generated by the functions $\{ g(\,\cdot\,, \, \cdot\, , j) \, : \, j\in \Z_{\NumWin} \}$ for some $g\in L^{2}(G\times\Z_{\NumDim}\times \Z_{\NumWin})$ and the closed subgroup of the \emph{specific form} $\Lambda\times\{0\}\subseteq (G\times\ghat)\times (\Z_{\NumDim}\times\widehat{\Z}_{\NumDim})$, 
\begin{equation} \label{eq:mws-system} \big\{ \big(  \pi(\lambda) g(\,\cdot\,, k, j) \big)_{k\in\Z_{\NumDim}} \, : \, \lambda\in\Lambda, \, j\in \Z_{\NumWin} \big\} \subset L^{2}(G\times\Z_{\NumDim}),\end{equation}
then we call such a collection of functions an $\NumWin$-multi-window $\NumDim$-\emph{super} Gabor system in $L^{2}(G\times\Z_{\NumDim})$. 
\end{definition}
\noindent\textbf{Notation.} If $f$ is a function in $L^{2}(G\times\Z_{\NumDim})$, then we shall write $\vvfun{f}{k}$ rather than $f(\,\cdot\,, k)$ for some $k\in\Z_{\NumDim}$. Similarly, for functions in $L^{2}(G\times\Z_{\NumDim}\times\Z_{\NumWin})$ we shall write $\mvfun{f}{k}{j}$ rather than $f(\,\cdot\, , k,j)$. Further, we let $\mvfun{f}{}{j}$ denote the function $f(\,\cdot\,,\,\cdot\, , j)\in L^{2}(G\times\Z_{\NumDim})$ for some $j\in\Z_{\NumWin}$. 

\begin{remark}
The name \emph{multi-window} indicates that we allow for more than one function to generate the Gabor system. The adjective \emph{super} indicates that we are interested in Gabor systems in $L^{2}(G\times\Z_{\NumDim})$ with $\NumDim> 1$ and where $\Lambda$ \emph{remains} a closed subgroup of $G\times\ghat$.  
\end{remark}
\begin{remark} \label{rem:vector-valued}
Observe that for $\NumDim<\infty$ there is no difference between the Hilbert space $L^{2}(G\times\Z_{\NumDim})$ and the Hilbert space of square integrable vector valued functions over $G$, $L^{2}(G; \C^{\NumDim})$. In this point of view, super Gabor frames are typically written as
\[ \{ \pi(\lambda) \vvfun{g}{1} \oplus \pi(\lambda) \vvfun{g}{2} \oplus \ldots \oplus  \pi(\lambda) \vvfun{g}{\NumDim} \}_{\la\in\Lambda}\subset \bigoplus_{k\in\Z_{\NumDim}} L^{2}(G).\] 
Similarly, for finite $\NumDim$ and $\NumWin$ the Hilbert space $L^{2}(G\times\Z_{\NumDim}\times\Z_{\NumWin})$ is nothing but the space of square integrable matrix valued functions over $G$, $L^{2}(G; \C^{\NumDim\times\NumWin})$. As stated in Lemma \ref{le:s0-properties}: a function $f$ belongs to $\SO(G\times\Z_{\NumDim}\times\Z_{\NumWin})$ if and only if $\mvfun{f}{k}{j}\in \SO(G)$ for all $(k,j)\in\Z_{\NumDim}\times\Z_{\NumWin}$ and $\sum_{k,j} \Vert \mvfun{f}{k}{j} \Vert_{\SO(G)} < \infty$ (which also defines an equivalent norm on $\SO(G\times\Z_{\NumDim}\times\Z_{\NumWin})$. 
\end{remark}

\begin{remark} Gabor systems for $\NumDim=\NumWin=1$ in $L^{2}(G)$, and in particular for $G=\R^{n}$, are well understood, see for example the books \cite{ch16} and \cite{gr01} and the references therein. Super Gabor systems were first studied in  \cite{ba00} and have since then also been considered in \cite{fu08,grly09}. Results on multi-window Gabor systems can be found in \cite{feko98,lu09,zezi97} and they are also mentioned briefly in \cite{ch16,gr01}.
\end{remark}

\begin{remark} In case $\NumDim$ is finite then it is not hard to realize that any multi-window Gabor system in $L^{2}(G\times\Z_{\NumDim})$ generated by a function $\tilde{g}\in L^{2}(G\times\Z_{\NumDim}\times\Z_{\NumWin})$ and where the time-frequency shifts are taken from a subgroup of the entire phase space $(G\times\ghat)\times (\Z_{\NumDim}\times\widehat{\Z}_{\NumDim})$ can be written as a multi-window super Gabor system in $L^{2}(G\times\Z_{\NumDim})$ for an appropriate choice of $g\in L^{2}(G\times\Z_{\NumDim}\times\Z_{\NumWin})$ and $\Lambda\subseteq G\times\ghat$. 
\end{remark}

\begin{remark}
Despite the fact that (multi-window) super Gabor systems are a special case of (multi-window) Gabor systems, it turns out that there is theoretic insight to be gained by considering them in their own right. 
Specifically, Theorem \ref{th:duality} shows that there is an intimate relation between the ``multi-window part'' and ``super part'' of a Gabor system, i.e., between the two parameters $\NumWin$ and $\NumDim$.
\end{remark}

\begin{remark} It is instructive to consider the results given here for the traditional setting $\NumWin=\NumDim=1$, $G=\R$ and $\Lambda = \alpha\Z \times \beta\Z$, $\alpha,\beta\ne 0$ (equipped with the Lebesgue and counting measure, respectively and in which case $\s(\La) = \vert \alpha\beta \vert$).
\end{remark}

To an $\NumWin$-multi-window $\NumDim$-super-Gabor system we associate the \emph{analysis operator}
\[ C_{g,\Lambda} : f \mapsto \Big\{ \sum_{k\in\Z_{\NumDim}} \langle \vvfun{f}{k} , \pi(\lambda)\mvfun{g}{k}{j}\rangle \Big\}_{\lambda\in \Lambda, j\in \Z_{\NumWin}}, \ f\in L^{2}(G\times\Z_{\NumDim})\]
and the \emph{synthesis operator} (the adjoint of the analysis operator),
\[ D_{g,\Lambda} : c \mapsto \sum_{j\in \Z_{\NumWin}} \int_{\Lambda} c(\lambda,j) \, \pi(\lambda) \mvfun{g}{}{j} \, d\mu_{\Lambda}(\lambda), \ c\in L^{2}(\Lambda\times \Z_{\NumWin}). \]
We wish to determine when 
\begin{subequations} \label{eq:frame}
\begin{equation}
C_{g,\Lambda} \ \text{is an injective bounded operator from} \ L^{2}(G\times\Z_{\NumDim}) \ \text{into}  \ L^{2}(\Lambda\times\Z_{\NumWin}) \ \text{with closed range}.
\end{equation}
Equivalently, there should exist constants $A,B>0$ such that
\begin{align} \label{eq:frame-ineq} A \, \Vert f \Vert_{2}^{2} \le & \Vert C_{g,\Lambda} f \Vert_{2}^{2} \le B \, \Vert f \Vert_{2}^{2} \ \ \textnormal{for all} \ f\in L^{2}(G\times\Z_{\NumDim}), \\
 \text{where} \ \ & \Vert C_{g,\Lambda} f \Vert_{2}^{2} =  \sum_{j\in \Z_{\NumWin}} \int_{\Lambda} \Big\vert \sum_{k\in\Z_{\NumDim}} \big\langle \vvfun{f}{k} , \pi(\lambda) \mvfun{g}{k}{j}\big\rangle \Big\vert^{2} \, d\mu_{\Lambda}(\lambda). \nonumber \end{align}
\end{subequations}
Similarly, we wish to determine when 
\begin{subequations} \label{eq:riesz}
\begin{equation} 
D_{g,\Lambda} \ \text{is an injective bounded operator from} \ L^{2}(\Lambda\times\Z_{\NumWin}) \ \text{into}  \ L^{2}(G\times\Z_{\NumDim}) \ \text{with closed range},
\end{equation}
or, equivalently, when do there exist constants $A,B>0$ such that
\begin{align} \label{eq:Riesz-ineq} A \, \Vert c \Vert_{2}^{2} \le & \Vert D_{g,\Lambda} c \Vert_{2}^{2} \le B \, \Vert c \Vert_{2}^{2} \ \ \textnormal{for all} \ \ c\in L^{2}(\Lambda\times\Z_{\NumWin}), \\ \text{where} \ \ & \Vert D_{g,\Lambda} c \Vert_{2}^{2} = \sum_{k\in\Z_{\NumDim}} \Big\Vert \sum_{j\in \Z_{\NumWin}} \int_{\Lambda} c(\lambda,j) \, \pi(\lambda) \mvfun{g}{k}{j} \, d\mu_{\Lambda}(\lambda) \, \Big\Vert_{L^{2}(G)}^{2}. \nonumber  \end{align}
\end{subequations}
If $g$ and $\Lambda$ are such that the associated analysis and synthesis operator is bounded, i.e., the upper inequality in either (and hence both) \eqref{eq:frame-ineq} and \eqref{eq:Riesz-ineq} is satisfied, then we call the (multi-window super) Gabor system a \emph{Bessel system}. Observe that a Gabor system is a Bessel system if and only if the upper inequality \eqref{eq:frame-ineq} holds for all $f$ in a dense subspace, e.g., for all $f\in \SO$. In that case the lower inequality holds also if and only if it is satisfied for a dense subspace.

\begin{remark} \label{rem:bessel-vs-finite-module-norm} The Bessel bound of a Gabor system is related to the the associated \emph{module norm} (see Lemma \ref{le:module-norm-and-bessel} and Remark \ref{rem:besselbound-mws-module}).
\end{remark}

As we shall see in Lemma \ref{le:s0-implies-bessel} it is easy to find $g\in L^{2}(G\times\Z_{\NumDim}\times\Z_{\NumWin})$ such that the generated Gabor system is a Bessel system for any closed subgroup $\Lambda$ and for any choice of $\NumDim$ and $\NumWin$ (e.g., take $g\in \SO(G\times\Z_{\NumDim}\times\Z_{\NumWin})$). In contrast, the lower inequalities are non-trivial to satisfy and entail some necessary conditions for the generating function $g$, the subgroup $\Lambda$, the dimension $\NumDim$ and the number of generators $\NumWin$.

\begin{lemma} \label{le:nec-cond} Consider the $\NumWin$-multi-window $\NumDim$-super Gabor system generated by a function $g\in L^{2}(G\times\Z_{\NumDim}\times\Z_{\NumWin})$ and a closed subgroup $\Lambda\subseteq G\times\ghat$.
\begin{enumerate}
\item[(i)] If $g$ is such that the system satisfies \eqref{eq:frame}, then $(G\times\ghat)/\Lambda$ is compact, $\NumDim<\infty$ and
\[ A \, \s(\Lambda) \, \NumDim \le \Vert g \Vert_{2}^{2} \le B \, \s(\Lambda) \, \NumDim.\]
If, in addition, $\Lambda$ is discrete and equipped with the counting measure, then $\s(\Lambda) \le \NumWin/\NumDim$.
\item[(ii)] If $g$ is such that the system satisfies \eqref{eq:riesz}, then $\Lambda$ is discrete, $\NumWin<\infty$ and
\[ A \, \NumWin \le \Vert g \Vert_{2}^{2} \le B \, \NumWin.\]
If, in addition, $(G\times\ghat)/\Lambda$ is compact and $\Lambda$ is equipped with the counting measure, then $\s(\Lambda) \ge \NumWin/\NumDim$.
\end{enumerate}
\end{lemma}
\begin{remark} Lemma \ref{le:nec-cond} states that a necessary condition for a multi-window Gabor system to be a frame is that the quotient group $(G\times \ghat/\La)$ is compact (this condition was first proven in \cite{jale16-2} for the single-window case). In Theorem \ref{th:S0-fin-gen}/\ref{th:exist} we show that this condition is also \emph{sufficient} for the existence of multi-window Gabor frames.  
\end{remark}

We state a proof of Lemma \ref{le:nec-cond}(i) in the Appendix. Lemma \ref{le:nec-cond}(ii) will be proven as part of the proof of Theorem \ref{th:duality}.

\begin{definition} \label{def:frame}
A multi-window super Gabor system which satisfies \eqref{eq:frame} is a \emph{frame} for $L^{2}(G\times\Z_{\NumDim})$. A multi-window super Gabor system which satisfies \eqref{eq:riesz} is a \emph{Riesz sequence} for $L^{2}(G\times\Z_{\NumDim})$. In that case the constants $A$ and $B$ are called the \emph{frame} and \emph{Riesz} bounds, respectively. If it is possible to take $A=B$, then the systems are said to be \emph{tight}.  
\end{definition}
In general, if one increases $\NumDim$, then it becomes increasingly difficult to find $g\in L^{2}(G\times\Z_{\NumDim}\times\Z_\NumWin)$ such that \eqref{eq:frame} holds. On the other hand, increasing the number of windows $\NumWin$ decreases the difficulty of finding $g$ such that \eqref{eq:frame} holds. For Riesz sequences the converse is true.

\begin{remark} Let $\Lambda$ be discrete, co-compact and equipped with the counting measure and take $g\in L^{2}(G\times\Z_{\NumDim}\times\Z_{\NumWin})$. Lemma \ref{le:nec-cond} implies that the \emph{minimum number} of windows for the resulting Gabor system to be a multi-window super \emph{frame} for $L^{2}(G\times\Z_\NumDim)$ is $\lceil \NumDim \, \s(\Lambda)\rceil$. Similarly, the \emph{maximum number} of windows for the resulting Gabor system to be a \emph{Riesz sequence} for $L^{2}(G\times\Z_{\NumDim})$ is $\lfloor \NumDim\, \s(\Lambda) \rfloor$. In particular, if a multi-window super Gabor system should be both a Riesz sequence and a frame for $L^{2}(G\times\Z_{\NumDim})$, then $\NumDim \, \s(\La) = \NumWin$ (as we shall see in a moment the Gabor system will then be a basis for $L^{2}(G\times\Z_{\NumDim})$). In Section \ref{sec:existence-mws-frames-for-R}  we show that for multi-window Gabor systems in $L^{2}(\R^{m})$ one can achieve frames for any lattice $\La\subset\R^{2m}$ with the minimum numbers of windows in $\SO$.
\end{remark}

\subsection{Frames and Riesz sequences}
\label{sec:frame-and-riesz}

The theory of frames and Riesz sequences for Hilbert spaces is well understood, see for example the book by Christensen \cite{ch16}. Below, in Lemma \ref{le:frame}-\ref{le:dual-riesz-char}, we summarize the most important aspects of this theory for multi-window Gabor systems. The statements for (multi-window) super Gabor systems follow as a special case, cf.\ Definition \ref{def:mwg}. 

\begin{lemma}[\textnormal{Frames}]\label{le:frame}
Let $g\in L^{2}(G\times\Z_{\NumWin})$ generate a multi-window  Gabor frame for $L^{2}(G)$ with respect to the closed 
subgroup $\Lambda$ of $G\times\ghat$ (by Lemma \ref{le:nec-cond}, the quotient group $(G\times\ghat)/\Lambda$ is necessarily compact).
\begin{enumerate}
\item[(i)] The \emph{frame operator} $S_{g,\Lambda} := D_{g,\Lambda}\circ C_{g,\Lambda}$ is a well-defined, linear, bounded, self-adjoint, positive and invertible operator on $L^{2}(G)$. Moreover, it commutes with time-frequency shifts from $\Lambda$,
\[ S_{g,\Lambda} \, \pi(\lambda) = \pi(\lambda) \, S_{g,\Lambda} \ \ \text{for all} \ \ \la\in\La.\]
In case it is a tight frame, then the frame operator is a multiple of the identity.
\item[(ii)] Let $h\in L^{2}(G\times\Z_{\NumWin})$ be defined by $\vvfun{h}{j} = S_{g,\Lambda}^{-1} \vvfun{g}{j}$, $ j\in\Z_{\NumWin}$. 
\begin{enumerate}
\item[(ii.a)] The function $h$ generates a multi-window Gabor frame for $L^{2}(G)$.
\item[(ii.b)] The operator $D_{g,\Lambda}\circ C_{h,\Lambda}$ as well as $D_{h,\Lambda}\circ C_{g,\Lambda}$ is the identity on $L^{2}(G)$.
Specifically, the latter leads to a resolution of the identity of the form
\begin{equation} \label{eq:frame-rep} \langle f_{1},f_{2}\rangle = \sum_{j\in\Z_{\NumWin}} \int_{\Lambda} \langle {f_{1}}, \pi(\lambda) \vvfun{g}{j}\rangle  \,   \langle \pi(\lambda) \vvfun{h}{j}, {f_{2}} \rangle \, d\mu_{\Lambda}(\lambda)\end{equation}
for all $f_{1},f_{2}\in L^{2}(G)$.
\end{enumerate}
\item[(iii)] Unless the Gabor frame generated by $g$ at the same time is a Riesz sequence for $L^{2}(G)$, then there are other functions $h$ besides the one constructed in (ii) such that, still, both $D_{g,\Lambda}\circ C_{h,\Lambda}$ and $D_{h,\Lambda}\circ C_{g,\Lambda}$ are the identity on $L^{2}(G)$.
\end{enumerate}
\end{lemma}
The function $h$ of Lemma \ref{le:frame}(ii) is called the \emph{canonical dual generator} of $g$. The other functions $h$ described in Lemma \ref{le:frame}(iii) are called \emph{dual generators} of $g$. The pair of frames generated  by a (canonical) pair of generators is called a pair of (canonical) \emph{dual} frames for $L^{2}(G)$. 
\begin{proof}[Proof of Lemma \ref{le:frame}](i). The statements about the frame operator can be found in, e.g., \cite[Theorem 5.1.5]{ch16}. The commutation relation is an important fact for Gabor systems and is easily verified: for any $\lambda'\in\Lambda$
\begin{align*} S_{g,\La} \, \pi(\lambda') \, f & = \sum_{j\in\Z_{\NumWin}} \int_{\La} \langle \pi(\la') f, \pi(\la) \vvfun{g}{j} \rangle \pi(\la) \vvfun{g}{j} \, d\lambda \\
& = \sum_{j\in\Z_{\NumWin}} \int_{\La} \langle f, \cocy(\la',\la'-\la) \, \pi(\la-\la') \vvfun{g}{j}\rangle \pi(\la) \vvfun{g}{j} \, d\la \\
& = \sum_{j\in\Z_{\NumWin}} \int_{\La} \langle f, \overline{\cocy(\la',\la)} \, \pi(\la) \vvfun{g}{j}\rangle \pi(\la+\la') \vvfun{g}{j} \, d\la \\
& = \sum_{j\in\Z_{\NumWin}} \int_{\La} \langle f, \overline{\cocy(\la',\la)} \, \pi(\la) \vvfun{g}{j}\rangle \overline{\cocy(\la',\la)} \pi(\la')\pi(\la) \vvfun{g}{j} \, d\la \\ & = \pi(\la') \, S_{g,\La} \, f.
\end{align*}(ii). This follows as in \cite[Theorem 5.1.6]{ch16} (iii). This follows from the characterization of all dual frames as can be found in \cite[Theorem 7.1.1]{ch16} and its adaptation to Gabor systems as in \cite[Proposition 12.3.6]{ch16}.  
\end{proof}
\begin{lemma} \label{le:dual-frame-char}
Let $g$ be a function in $L^{2}(G\times\Z_{\NumWin})$ and let $\Lambda$ be a closed subgroup of $G\times\ghat$. The following statements are equivalent:
\begin{enumerate}
\item[(i)] The multi-window Gabor system generated by $g$ and $\Lambda$ is a frame for $L^{2}(G)$,
\item[(ii)] The multi-window Gabor system generated by $g$ and $\Lambda$ is a Bessel system and there exists a function $h\in L^{2}(G\times\Z_{\NumWin})$ that also generates a multi-window Gabor Bessel system for $L^{2}(G)$ such that \eqref{eq:frame-rep} is satisfied.
\end{enumerate}
\end{lemma}
\begin{proof} See \cite[Lemma 6.3.2]{ch16}.
\end{proof}
Usually Gabor frames are constructed where $\Lambda$ is discrete (and by Lemma \ref{le:nec-cond} necessarily also co-compact -- thus a lattice). In that case multi-window super Gabor frames lead to \emph{series} representations of the form
\begin{equation} \label{eq:0504c} f_k = \sum_{j\in\Z_{\NumWin}} \sum_{\lambda\in\Lambda} \Big( \sum_{k'\in\Z_{\NumDim}} \big\langle \vvfun{f}{k'}, \pi(\lambda) \mvfun{g}{k'}{j}\big\rangle \Big) \,   \pi(\lambda) \mvfun{h}{k}{j} \ \ \text{for all} \ \ f\in L^{2}(G\times\Z_{\NumDim}),\, k\in\mathbb{Z}_d.\end{equation}
One should note that not all groups have (non-trivial) discrete or co-compact subgroups.

One goal of Gabor theory is to find pairs of functions $g$ and $h$ in $L^{2}(G\times\Z_{\NumWin})$ and subgroups $\Lambda$ of $G\times\ghat$ such that frame expansions of the form \eqref{eq:frame-rep} hold. 

In contrast to frames, multi-window (super) Gabor \emph{Riesz sequences} for $L^{2}(G\times\Z_{\NumDim})$ give representations for functions \emph{only} in the closure of their span. Again, we state this only for multi-window Gabor systems as the statement for super systems is a special case, cf.\ Definition \ref{def:mwg}.

\begin{lemma}[\textnormal{Riesz sequences}]\label{le:riesz-sequence} Let $g\in L^{2}(G\times\Z_{\NumWin})$ generate a multi-window Gabor Riesz sequence for $L^{2}(G)$ with respect to the closed subgroup $\Lambda$ of $G\times\ghat$ (by Lemma \ref{le:nec-cond}, the group $\Lambda$ is necessarily discrete and $\NumWin<\infty$). Furthermore we let $R(D_{g,\Lambda})$ denote the range of the synthesis operator $D_{g,\Lambda}$ (which by definition of a Riesz sequence is closed in $L^{2}(G)$). The following holds:
\begin{enumerate} 
\item[(i)] The Gabor system is an unconditional basis for $R(D_{g,\Lambda})$.
\item[(ii)] The \emph{frame operator} $S_{g,\Lambda} := D_{g,\Lambda}\circ C_{g,\Lambda}$ is a well-defined, linear, bounded, self-adjoint, positive and invertible operator on $R(D_{g,\Lambda})$ (the Gabor system is a frame for $R(D_{g,\Lambda})$).
\item[(iii)] The function $h\in L^{2}(G\times\Z_{\NumWin})$ defined by
\[ \vvfun{h}{j} = S_{g,\Lambda}^{-1} \vvfun{g}{j}, \ \ j\in\Z_{\NumWin},\]
is the unique function such that the generated multi-window Gabor system is a Riesz sequence for $L^{2}(G)$ with the property that $R(D_{h,\Lambda}) = R(D_{g,\Lambda})$. Furthermore the operators $D_{g,\Lambda}\circ C_{h,\Lambda}$ and $D_{h,\Lambda}\circ C_{g,\Lambda}$ are projections from $L^{2}(G)$ onto $R(D_{g,\Lambda})$, and so
\begin{equation}\label{eq:riesz-rep} 
f = D_{h,\Lambda} \circ C_{g,\Lambda} f = D_{g,\Lambda} \circ C_{h,\Lambda} f \ \ \text{for all} \ \ f\in R(D_{g,\Lambda}).\end{equation}
\item[(iv)] Unless $R(D_{g,\Lambda})= L^{2}(G\times\Z_{\NumDim})$ (which is the case exactly if the Gabor system also is a frame for $L^{2}(G)$), then there are other functions $h\in L^{2}(G\times\Z_{\NumWin})$ that generate multi-window Gabor Riesz sequences for $L^{2}(G)$ such that $R(D_{h,\Lambda}) \ne R(D_{g,\Lambda})$ and the following holds.
\begin{enumerate}
\item[(iv.a)] $D_{h,\Lambda} \circ C_{g,\Lambda}$ is a projection from $L^{2}(G)$ onto $R(D_{h,\Lambda})$, 
\begin{align} \label{eq:0504a} 
& f = D_{h,\Lambda} \circ C_{g,\Lambda} f \ \ \text{for all} \ \ f\in R(D_{h,\Lambda}).\end{align}
\item[(iv.b)] $D_{g,\Lambda} \circ C_{h,\Lambda}$ is a projection from $L^{2}(G)$ onto $R(D_{g,\Lambda})$,
\begin{align} \label{eq:0504b} 
f = D_{g,\Lambda} \circ C_{h,\Lambda} f \ \ \text{for all} \ \ f\in R(D_{g,\Lambda}).\end{align}
\end{enumerate}
For $g$ and $h$ to be two functions in $L^{2}(G\times\Z_{\NumWin}\times\Z_{\NumDim})$ that generate multi-window Gabor Riesz sequences such that \eqref{eq:0504a} and \eqref{eq:0504b} hold, it is necessary and sufficient that the two families of functions are biorthonormal, that is,
\begin{equation} \label{eq:def:biorth-multi} \langle \pi(\lambda) \vvfun{g}{j},\pi(\lambda') \vvfun{h}{j'} \rangle = \begin{cases} 1 & \text{if } \lambda'=\lambda, \ j'=j,\\ 0 & \text{otherwise}, \end{cases} \ \ \text{for all} \ \ \lambda,\lambda'\in\Lambda, \ j,j'\in\Z_{\NumWin}.  \end{equation}
\end{enumerate}

\end{lemma}
\begin{proof} See Sections 3.6, 3.7 and Section 7.1 in \cite{ch16}.\end{proof}
In case of multi-window \emph{super} Gabor Riesz sequences, the biorthonormality relations in  \eqref{eq:def:biorth-multi} take the form
\begin{equation} \label{eq:def:biorth-multi-super} \sum_{k\in\Z_{\NumDim}} \langle \pi(\lambda) \mvfun{g}{k}{j},\pi(\lambda') \mvfun{h}{k}{j'} \rangle = \begin{cases} 1 & \text{if } \lambda'=\lambda, \ j'=j,\\ 0 & \text{otherwise}, \end{cases} \ \ \text{for all} \ \ \lambda,\lambda'\in\Lambda, \ j,j'\in\Z_{\NumWin}.  \end{equation}

\begin{lemma} \label{le:dual-riesz-char}
Let $g$ be a function in $L^{2}(G\times\Z_{\NumWin})$ and let $\Lambda$ be a discrete subgroup of $G\times\ghat$. The following statements are equivalent:
\begin{enumerate}
\item[(i)] The multi-window Gabor system generated by $g$ and $\Lambda$ is a Riesz sequence for $L^{2}(G)$,
\item[(ii)] The multi-window Gabor system generated by $g$ and $\Lambda$ is a Bessel system and there exists a function $h\in L^{2}(G\times\Z_{\NumWin})$ that also generates a multi-window Gabor Bessel system for $L^{2}(G)$ such that the two families of functions such that \eqref{eq:def:biorth-multi} is satisfied.
\end{enumerate}
\end{lemma}
\begin{proof}See Theorem 6.6 in \cite{jale16-1}
\end{proof}


\subsection{The continuous case} \label{sec:continuous-case}

If one picks the subgroup $\Lambda$ to be the entire time-frequency plane $G\times\ghat$, then the task of constructing (multi-window super) Gabor frames is trivial.

\begin{lemma} \label{le:1009} Any non-zero function $g\in L^{2}(G\times\Z_{\NumWin})$ generates a multi-window Gabor frame for $L^{2}(G)$ with respect to the subgroup $\Lambda = G\times\ghat$ and the associated frame operator satisfies
\[ S_{g,\Lambda} f = \sum_{j\in\Z_{\NumWin}} \int_{G\times\ghat} \langle f, \pi(\chi) \vvfun{g}{j} \rangle \pi(\chi) \vvfun{g}{j} \, d\chi = \Vert g \Vert_{2}^{2} \, f \ \ \text{for all} \ \ f\in L^{2}(G).\]
Moreover, any pair of functions $g,h\in L^{2}(G\times\Z_{\NumWin})$ such that $\langle g,h\rangle = 1$ generate dual multi-window Gabor frames for $L^{2}(G)$ with respect to the subgroup $\Lambda = G\times\ghat$, so that
\[ \langle f_{1},f_{2}\rangle = \sum_{j\in\Z_{\NumWin}} \int_{G\times\ghat} \langle f_{1}, \pi(\chi) \vvfun{g}{j} \rangle \, \langle \pi(\chi) \vvfun{h}{j} , f_{2}\rangle \, d\chi \ \ \text{for all} \ \ f_{1},f_{2}\in L^{2}(G).\]
\end{lemma}
\begin{proof} The results follow by applying \eqref{eq:STFT}.
\end{proof}
\begin{remark} \label{rem:continuous-mw-case-generator-in-so} 
It is clear that any
non-zero function $g\in \SO(G\times\Z_{\NumWin})$ generates a multi-window Gabor frame for $L^{2}(G)$ with respect to the subgroup $\La = G\times\ghat$, and that in this case the frame operator, $S_{g,\La} : f \mapsto \Vert g \Vert_{2}^{2} \, f$ also is invertible on $\SO(G)$.
\end{remark}

\begin{lemma} \label{le:1604} For any $\NumDim < \infty$ there exists a function $g\in L^{2}(G\times\Z_{\NumDim})$ that generates a super Gabor frame for $L^{2}(G\times\Z_{\NumDim})$ with respect to the subgroup $\Lambda = G\times\ghat$. Moreover, any two functions $g,h\in L^{2}(G\times\Z_{\NumDim})$ such that $\langle \vvfun{g}{k}, \vvfun{h}{k'}\rangle = \delta_{k,k'}$ for all $k,k'\in\Z_{\NumDim}$ generate dual super Gabor frames for $L^{2}(G\times\Z_{\NumDim})$ with respect to the subgroup $\Lambda = G\times\ghat$.
\end{lemma}
\begin{proof}
Let $g$ be any function in $L^{2}(G\times\Z_{\NumDim})$ such that the functions $\{ \vvfun{g}{k}\}_{k\in\Z_{\NumDim}}$ are linearly independent. Apply the Gram-Schmidt procedure to generate a function $\tilde{g}\in L^{2}(G\times\Z_{\NumDim})$ so that $\{ \vvfun{\tilde{g}}{k}\}_{k\in\Z_{\NumDim}}$ are an orthonormal collection of functions in $L^{2}(G)$. That is, we define $\vvfun{\tilde{g}}{0} := \vvfun{g}{0} \, \Vert \vvfun{g}{0} \Vert_{2}^{-1}$ and  
\[ \vvfun{{g}'}{k} := \vvfun{g}{k} - \sum_{k' = 0}^{k-1} \langle \vvfun{g}{k} , \vvfun{\tilde{g}}{k'} \rangle \, \vvfun{\tilde{g}}{k'}, \ \ \vvfun{\tilde{g}}{k} :=  \vvfun{g'}{k} \, \Vert \vvfun{g'}{k} \Vert^{-1}_{2}, \ \ \text{for} \ \ k = 1,2,\ldots, \NumDim-1. \]
Of course, any other orthonormalization procedure also works. Then, using \eqref{eq:STFT}, we find that
\begin{align*} & \int_{G\times\ghat} \big\vert \sum_{k\in\Z_{\NumDim}} \langle \vvfun{f}{k} , \pi(\chi) \vvfun{\tilde{g}}{k} \rangle  \big\vert^{2} \, d\chi = \sum_{k,k'\in\Z_{\NumDim}} \int_{G\times\ghat} \langle \vvfun{f}{k} , \pi(\chi) \vvfun{\tilde{g}}{k} \rangle  \,  \langle \pi(\chi) \vvfun{\tilde{g}}{k'}  , \vvfun{f}{k'} \rangle  \, d\chi \\
& \stackrel{\eqref{eq:STFT}}{=} \sum_{k,k'\in\Z_{\NumDim}} \langle \vvfun{\tilde{g}}{k'} , \vvfun{\tilde{g}}{k} \rangle \, \langle \vvfun{f}{k} , \vvfun{f}{k'} \rangle = \sum_{k,k'\in\Z_{\NumDim}} \delta_{k,k'}\, \langle \vvfun{f}{k} , \vvfun{f}{k'} \rangle = \sum_{k\in\Z_{\NumDim}} \Vert f_{k} \Vert_{2}^{2} = \Vert f \Vert_{2}^{2},
\end{align*}
for all $f\in L^{2}(G\times\Z_{\NumDim})$. We conclude that $\tilde{g}$ generates a super Gabor frame for $L^{2}(G\times\Z_{\NumDim})$ with respect to the subgroup $\Lambda = G\times\ghat$ and with frame bounds $A = B = 1$ and where the frame operator is the identity.
The moreover part follows also using \eqref{eq:STFT}: for all $f^{1},f^{2}\in L^{2}(G\times\Z_{\NumDim})$
\begin{align*} & \int_{G\times\ghat} \Big( \sum_{k\in\Z_{\NumDim}} \langle \vvfun{f^{1}}{k} , \pi(\chi) \vvfun{g}{k} \rangle \Big) \, \Big( \sum_{k'\in\Z_{\NumDim}} \langle \pi(\chi) \vvfun{h}{k'} , \vvfun{f^{2}}{k'} \rangle \Big) \, d\chi \\
& = \sum_{k,k'\in\Z_{\NumDim}} \langle \vvfun{h}{k'} , \vvfun{g}{k} \rangle \, \langle \vvfun{f^{1}}{k} ,\vvfun{f^{2}}{k'} \rangle = \sum_{k,k'\in\Z_{\NumDim}} \delta_{k,k'} \langle \vvfun{f^{1}}{k} ,\vvfun{f^{2}}{k'} \rangle = \langle f^{1}, f^{2}\rangle.\end{align*}
\end{proof}

\begin{remark} \label{rem:continuous-super-case-generator-in-so} 
In the proof of Lemma \ref{le:1604}, if one starts with a function $g\in \SO(G\times\Z_{\NumDim})$ where $\{\vvfun{g}{k}\}_{k\in\Z_{\NumDim}}$ is a collection of linearly independent functions, then the described Gram-Schmidt procedure will generate a collection of functions in $\SO(G)$ which is orthonormal in $L^{2}(G)$. Therefore, for any $\NumDim<\infty$ the proof produces a function in $\SO(G\times\Z_{\NumDim})$ that generates a super Gabor frame for $L^{2}(G\times\Z_{\NumDim})$ with respect to the subgroup $\La=G\times\ghat$ and where the associated Gabor frame operator $S_{g,\La} : f\mapsto f$ also is invertible on $\SO(G\times\Z_{\NumDim})$.
\end{remark}

\subsection{The duality between Gabor frames and Gabor Riesz sequences}
\label{sec:duality}
As stated in the introduction, there is an intimate relation between the frame and Riesz sequence of a Gabor system generated by a function and with time-frequency shifts from the subgroup $\Lambda$ and its adjoint group $\Lambda^{\circ}$. We now extend this relation, the duality principle for Gabor systems, to a duality between multi-window super Gabor systems. Its proof will be given in Section \ref{sec:proof-duality}. 


\begin{proposition} \label{pr:bessel-duality} Let $g$ be a function in $L^{2}(G\times\Z_{\NumDim}\times \Z_{\NumWin})$ and let $\Lambda$ be a closed subgroup of $G\times\ghat$. The following statements are equivalent.
\begin{enumerate} 
\item[(i)] The $\NumWin$-multi-window $\NumDim$-super Gabor system in $L^{2}(G\times\Z_{\NumDim})$ generated by the functions $\mvfun{g}{}{j}$, $j\in\Z_{\NumWin}$ and with time-frequency shifts along $\Lambda$ is a Bessel system with bound $B$,
\begin{align*} \sum_{j\in \Z_{\NumWin}} \int_{\Lambda} \Big\vert \sum_{k\in\Z_{\NumDim}} \big\langle \vvfun{f}{k} , \pi(\lambda) \mvfun{g}{k}{j}\big\rangle \Big\vert^{2} \, d\mu_{\Lambda}(\lambda) \le B \, \Vert f \Vert_{2}^{2} \ \ \textnormal{for all} \ f\in L^{2}(G\times\Z_{\NumDim}). \end{align*}
\item[(ii)] The $\NumDim$-multi-window $\NumWin$-super Gabor system in $L^{2}(G\times\Z_{\NumWin})$ generated by the functions $\mvfun{g}{k}{}$, $k\in\Z_{\NumDim}$ and with time-frequency shifts along $\Lambda^{\circ}$ is a Bessel system with bound $B$,
\begin{align*} \sum_{k\in \Z_{\NumDim}} \int_{\Lambda^{\circ}} \Big\vert \sum_{j\in\Z_{\NumWin}} \big\langle \vvfun{f}{j} , \pi(\lambda^{\circ}) \mvfun{g}{k}{j}\big\rangle \Big\vert^{2} \, d\mu_{\Lambda^{\circ}}(\lambda^{\circ}) \le B \, \Vert f \Vert_{2}^{2} \ \ \textnormal{for all} \ f\in L^{2}(G\times\Z_{\NumWin}). \end{align*}
\end{enumerate}
\end{proposition}
\noindent See Lemma \ref{le:module-norm-and-bessel}, Remark \ref{rem:bessel-duality}, and Remark \ref{rem:besselbound-mws-module} for references and comments on this result.

\begin{theorem}[Main result] \label{th:duality} Let $\Lambda$ be a closed subgroup in the time-frequency plane $G\times\ghat$ such that the quotient group $(G\times\ghat)/\Lambda$ is compact and let $g$ be a function in $L^{2}(G\times\Z_{\NumDim}\times\Z_{\NumWin})$ with $\NumDim<\infty$. The following statements are equivalent.
\begin{enumerate}
\item[(i)] The functions $\mvfun{g}{}{j}$, $j\in \Z_{\NumWin}$ generate an $\NumWin$-multi-window $\NumDim$-super Gabor frame for $L^{2}(G\times\Z_{\NumDim})$ with time-frequency shifts along $\Lambda$. Specifically, for all functions $f\in L^{2}(G\times\Z_{\NumDim})$, 
\[ A \, \Vert f \Vert_{2}^{2} \le \sum_{j\in \Z_{\NumWin}} \int_{\Lambda} \Big\vert \sum_{k\in\Z_{\NumDim}} \big\langle \vvfun{f}{k}, \pi(\lambda) \mvfun{g}{k}{j}\big\rangle \Big\vert^{2} \, d\mu_{\Lambda}(\lambda) \le B \, \Vert f \Vert_{2}^{2}.\]
\item[(ii)] The functions $\mvfun{g}{k}{}$, $k\in \Z_{\NumDim}$ generate a $\NumDim$-multi-window $\NumWin$-super Gabor Riesz sequence for $L^{2}(G\times\Z_{\NumWin})$ with time-frequency shifts along $\Lambda^{\circ}$. Specifically, for all sequences $c\in \ell^{2}(\Lambda^{\circ}\times\Z_{\NumDim})$,
\[ \frac{A}{\s(\Lambda)} \sum_{\substack{\lambda^{\circ}\in \Lambda^{\circ}\\k\in\Z_{\NumDim}}} \vert c(\lambda^{\circ},k)\vert^{2} \le \sum_{j\in\Z_{\NumWin}} \Big\Vert \frac{1}{\s(\Lambda)} \sum_{\substack{\lambda^{\circ}\in\Lambda^{\circ}\\k\in \Z_{\NumDim} }} c(\lambda^{\circ},k) \, \pi(\lambda^{\circ}) \mvfun{g}{k}{j} \, \Big\Vert_{2}^{2}\le \frac{B}{\s(\Lambda)} \sum_{\substack{\lambda^{\circ}\in \Lambda^{\circ}\\k\in\Z_{\NumDim}}} \vert c(\lambda^{\circ},k)\vert^{2} .\]
\end{enumerate}
\end{theorem}
 For references to the duality principle, please see the last paragraph of the introduction.
 
\begin{remark} In statement (i) of Proposition \ref{pr:bessel-duality} and Theorem \ref{th:duality} the multi-window Gabor system is generated by the functions $\{ \mvfun{g}{}{j} \, : \, j\in \Z_{\NumWin} \, \}$ in $L^{2}(G\times\Z_{\NumDim})$, where as in statement (ii) of either result the multi-window Gabor system is generated by the functions $\{ \mvfun{g}{k}{} \, : \, k\in \Z_{\NumDim} \, \}$ in $L^{2}(G\times\Z_{\NumWin})$. 
Hence the duality principle extends from the single window case with a duality between Gabor system generated by the subgroups $\Lambda$ and $\Lambda^{\circ}$ to a duality between the parameters $\NumWin$ and $\NumDim$. \end{remark}
\begin{remark}
The additional assumptions on $\Lambda$ and $d$ in Theorem \ref{th:duality} are no loss of generality: we know from Lemma \ref{le:nec-cond} that these assumptions are necessary for both statement (i) and (ii) to be true. The assumptions need not be put in Proposition \ref{pr:bessel-duality} because the Bessel property (the upper bound in \eqref{eq:frame-ineq} and \eqref{eq:Riesz-ineq}) does \emph{not} impose conditions on either $\Lambda$, $\NumDim$ or $\NumWin$.
\end{remark}


In order to prove the above results we need to build up some theory. In particular we need to establish facts about the \emph{fundamental identity of Gabor analysis} and to prove the \emph{Wexler-Raz biorthogonality relations}.

\subsection{The fundamental identity and the Wexler-Raz relations}
\label{sec:figa-wex-raz}

Proofs of the statements in this section are relayed to the Appendix.

The \emph{fundamental identity of} (multi-window super) \emph{Gabor analysis} relates the analysis and synthesis operator for Gabor systems generated by the subgroup $\Lambda$ and its adjoint subgroup, $\Lambda^{\circ}$ in the following way.
\begin{align}
& \sum_{j\in\Z_{\NumWin}} \int_{\Lambda} \Big( \sum_{k\in\Z_{\NumDim}} \big\langle \vvfun{f^{1}}{k}, \pi(\lambda) \mvfun{g}{k}{j}\big\rangle \Big) \, \Big( \sum_{l\in \Z_{\NumDim}} \big\langle \pi(\lambda) \mvfun{h}{l}{j},  \vvfun{f^{2}}{l} \big\rangle \Big)  \, d\mu_{\Lambda}(\lambda) \nonumber \\
& \hspace{2cm} = \int_{\Lambda^{\circ}} \sum_{ k,l\in\Z_{\NumDim} } \Big( \big\langle \pi(\lambda^{\circ}) \vvfun{f^{1}}{k},  \vvfun{f^{2}}{l} \big\rangle \,  \sum_{j\in\Z_{\NumWin}}\big\langle \mvfun{h}{l}{j}, \pi(\lambda^{\circ}) \mvfun{g}{k}{j} \big\rangle \Big) \, d\mu_{\Lambda^{\circ}}(\lambda^{\circ}). \label{eq:figa}
\end{align}
This equality is the result of an application of the symplectic Poisson formula (see Lemma \ref{le:s0-properties}(ix)) applied to a certain function (see Lemma \ref{le:janssen-function}). As such, we have to be careful about its validity as it does not hold for arbitrary functions $f^{1},f^{2}\in L^{2}(G\times\Z_{\NumDim})$, $g,h\in L^{2}(G\times\Z_{\NumDim}\times\Z_{\NumWin})$ and closed subgroups $\Lambda\subseteq G\times\ghat$. The following lemma states some sufficient conditions for \eqref{eq:figa} to hold.

\begin{lemma} \label{le:figa-suff-cond} Let $f^{1},f^{2}\in L^{2}(G\times\Z_{\NumDim})$, $g,h\in L^{2}(G\times\Z_{\NumDim}\times\Z_{\NumWin})$ and let $\Lambda$ be a closed subgroup in the time-frequency plane $G\times\ghat$. 
\begin{enumerate}
\item[(i)] If the multi-window super Gabor systems generated by $g$ and $h$ satisfy the upper frame inequality with respect to $\Lambda$ as stated in \eqref{eq:frame-ineq} and if the functions $f_{1},f_{2},g$ and $h$ are such that 
\begin{equation} \label{eq:cond-A} \int_{\Lambda^{\circ}} \Big\vert \sum_{\substack{ k,l\in\Z_{\NumDim} \\ j\in\Z_{\NumWin} }} \big\langle \mvfun{h}{l}{j}, \pi(\lambda^{\circ}) \mvfun{g}{k}{j} \big\rangle \, \big\langle \pi(\lambda^{\circ}) \vvfun{f^{1}}{k}, \vvfun{f^{2}}{l} \big\rangle \Big\vert \, d\mu_{\Lambda^{\circ}}(\lambda^{\circ}) < \infty,\end{equation}
then \eqref{eq:figa} holds.
\item[(ii)] If any two of the following conditions are satisfied, then \eqref{eq:figa} holds for any closed subgroup $\Lambda$. 
\begin{itemize}
\item [(a)] $\vvfun{f^{1}}{k} \in \SO(G)$ for all  $k\in\Z_{\NumDim}$ and $\sum_{k\in\Z_{\NumDim}} \Vert \vvfun{f^{1}}{k} \Vert_{\SO(G)}^{2} < \infty$,
\item [(b)] $\vvfun{f^{2}}{k} \in \SO(G)$ for all  $k\in\Z_{\NumDim}$ and $\sum_{k\in\Z_{\NumDim}} \Vert \vvfun{f^{2}}{k} \Vert_{\SO(G)}^{2} < \infty$,
\item[(c)] $\mvfun{g}{k}{j}\in \SO(G)$ for all $(k,j)\in\Z_{\NumDim}\times \Z_{\NumWin}$ and $\sum_{k\in\Z_{\NumDim}, j\in\Z_{\NumWin}} \Vert \mvfun{g}{k}{j}\Vert_{\SO}^{2} <\infty$,
\item[(d)] $\mvfun{h}{k}{j}\in \SO(G)$ for all $(k,j)\in\Z_{\NumDim}\times \Z_{\NumWin}$ and $\sum_{k\in\Z_{\NumDim}, j\in\Z_{\NumWin}} \Vert \mvfun{h}{k}{j}\Vert_{\SO}^{2} <\infty$.
\end{itemize} 
\end{enumerate}
\end{lemma}

\begin{lemma} \label{le:s0-implies-bessel} If $g$ and $h$ are functions in $L^{2}(G\times\Z_{\NumDim}\times\Z_{\NumWin})$ such that 
\[ \mvfun{g}{}{j},\mvfun{h}{}{j}\in \SO(G\times\Z_{\NumDim}) \ \text{for all} \ \ j\in\Z_{\NumWin} \]
and $\sum_{j\in\Z_{\NumWin} } \Vert \mvfun{g}{}{j} \Vert_{\SO}^{2}$, $\sum_{j\in\Z_{\NumWin} } \Vert \mvfun{h}{}{j} \Vert_{\SO}^{2}$ are finite, 
then, for any closed subgroup $\Lambda$ in $G\times\ghat$,  the (mixed) frame operator 
\[ (S_{g,h,\La} f)(\cdot, k) = \sum_{j\in\Z_{\NumWin}} \int_{\La} \Big( \sum_{k'\in\Z_{\NumDim}} \big\langle \vvfun{f}{k'}, \pi(\la) \mvfun{g}{k'}{j}\big\rangle \Big) \, \pi(\la) \mvfun{h}{k}{j}, \ k\in\Z_{\NumDim} \]
is a linear and bounded operator on both $L^{2}(G\times\Z_{\NumWin})$ and on $\SO(G\times\Z_{\NumWin})$. 
\end{lemma}
\begin{remark} Lemma \ref{le:s0-implies-bessel} implies that for any choice of $\La$ and any $g\in \SO(G\times\Z_{\NumDim}\times\Z_{\NumWin})$ the module norm (see Lemma \ref{le:module-norm-and-bessel} and Remark \ref{rem:besselbound-mws-module}) $\Vert g \Vert_{\La}$ is finite. 
\end{remark}

Our understanding of the fundamental identity of Gabor analysis allows us to easily show the Wexler-Raz relations that characterize pairs of functions $g,h\in L^{2}(G\times\Z_{\NumDim}\times\Z_{\NumWin})$ that generate dual $\NumWin$-multi-window $\NumDim$-super Gabor frames for $L^{2}(G\times\Z_{\NumDim})$.

\begin{theorem} \label{th:wex-raz} Let $\Lambda$ be a closed subgroup in the time-frequency plane $G\times\ghat$ such that $(G\times\ghat)/\Lambda$ is compact, $\NumDim<\infty$ and let $g,h\in L^{2}(G\times\Z_{\NumDim}\times\Z_{\NumWin})$ be such that the respective multi-window super Gabor systems satisfy the upper frame inequality. Then the following statements are equivalent.
\begin{enumerate}
\item[(i)] For all functions $f^{1},f^{2}\in L^{2}(G\times\Z_{\NumDim})$ 
 \[ \langle f^{1},f^{2}\rangle = \sum_{j\in\Z_{\NumWin}} \int_{\Lambda} \Big( \sum_{k\in\Z_{\NumDim}} \big\langle \vvfun{f^{1}}{k}, \pi(\lambda) \mvfun{g}{k}{j}\big\rangle \Big) \, \Big( \sum_{l\in \Z_{\NumDim}} \big\langle \pi(\lambda) \mvfun{h}{l}{j}, \vvfun{f^{2}}{l} \big\rangle \Big)  \, d\mu_{\Lambda}(\lambda),\]
i.e., the functions $g$ and $h$ generate dual multi-window super Gabor frames for $L^{2}(G\times\Z_{\NumDim}\times\Z_{\NumWin})$.
\item[(ii)] For all $\lambda^{\circ}\in \Lambda^{\circ}$ and $k,l\in\Z_{\NumDim}$ the functions $g$ and $h$ satisfy
\[ \sum_{j\in \Z_{\NumWin}} \big\langle \mvfun{h}{l}{j} , \pi(\lambda^{\circ}) \mvfun{g}{k}{j} \big\rangle = \begin{cases} \s(\Lambda) & \text{if } \lambda^{\circ} = 0, \ k= l, \\ 0 & \text{otherwise}.\end{cases}\]
\end{enumerate}
\end{theorem}
\begin{remark} If the equality in Theorem \ref{th:wex-raz}(i) should only hold for functions in $\SO(G\times\Z_{\NumDim})$ rather than all of $L^{2}(G\times\Z_{\NumDim})$, then the assumption that the functions $g$ and $h$ should satisfy the upper frame inequality can be omitted from the theorem.
\end{remark}

\subsection{Proof of the duality principles} 
\label{sec:proof-duality}

With the fundamental identity and the Wexler-Raz relations in place we can now prove the duality principle for multi-window super Gabor systems for $L^{2}(G\times\Z_{\NumDim})$.

The proof becomes more transparent if $\NumDim$ and/or $\NumWin$ is equal to $1$.

\begin{proof}[Proof of Proposition \ref{pr:bessel-duality}, Theorem \ref{th:duality} and Lemma \ref{le:nec-cond}(i)]
We begin with the proof that the statement in Theorem \ref{th:duality}(ii) implies Theorem \ref{th:duality}(i). At the same time we prove Proposition \ref{pr:bessel-duality} and Lemma \ref{le:nec-cond}(ii). Let therefore $\Lambda^{\circ}$ be a closed (not-necessarily discrete) subgroup in the time-frequency plane $G\times\ghat$ and $\NumDim,\NumWin\in \N\cup\{\infty\}$. Assume that $g\in L^{2}(G\times\Z_{\NumDim}\times\Z_{\NumWin})$ is such that the functions $\mvfun{g}{k}{}$, $k\in\Z_{\NumDim}$ generate a $\NumDim$-multi-window $\NumWin$-super Gabor system with respect to time-frequency shifts along $\Lambda^{\circ}$ that satisfies \eqref{eq:riesz}. That is, we assume that there exists constants $A,B>0$ such that, for all $c\in L^{2}(\Lambda^{\circ}\times\Z_{\NumDim})$,
\begin{align} 
& A \sum_{k\in\Z_{\NumDim}}\int_{\Lambda^{\circ}} \vert c(\lambda^{\circ},k) \vert^{2} \, d\mu_{\Lambda^{\circ}}(\lambda^{\circ}) \nonumber \\
& \hspace{2cm} \le \sum_{j\in\Z_{\NumWin}} \Big\Vert \sum_{k\in\Z_{\NumDim}} \int_{\Lambda^{\circ}} c(\lambda^{\circ},k) \, \pi(\lambda^{\circ}) \mvfun{g}{k}{j} \, d\mu_{\Lambda^{\circ}}(\lambda^{\circ}) \, \Big\Vert_{2}^{2} \label{eq:1204b}\\
& \hspace{4cm} \le B \sum_{k\in\Z_{\NumDim}}\int_{\Lambda^{\circ}} \vert c(\lambda^{\circ},k) \vert^{2} \, d\mu_{\Lambda^{\circ}}(\lambda^{\circ}).\nonumber \end{align}
Following the idea of Janssen \cite{ja95} (essentially equations (3.4)-(3.6) in the proof of \cite[Theorem 3.1]{ja95}) we construct a function $c\in L^{2}(\Lambda^{\circ}\times\Z_{\NumDim})$, which, when inserted into \eqref{eq:1204b}, allows us to deduce the frame inequalities of Theorem \ref{th:duality}(i).
To this end let $h$ be a function in $\SO(G)$ such that $\Vert h \Vert_{2} = 1$. Furthermore, let $f$ be any function in $\SO(G\times\Z_{\NumDim})$. Given any $\chi\in G\times\ghat$ we define
\[ c(\lambda^{\circ},k) 
= \langle \pi(\chi) h, \pi(\lambda^{\circ}) \vvfun{f}{k}\rangle, \ \lambda^{\circ}\in \Lambda^{\circ}, \ k\in \Z_{\NumDim}. \]
Since (a) $\SO$ is continuously embedded into $L^{2}$, (b) the restriction operator is linear and bounded from $\SO(G\times\ghat)$ onto $\SO(\Lambda^{\circ})$ and (c) $\Vert \mathcal{V}_{g}f \Vert_{\SO} \le C \Vert f \Vert_{\SO} \, \Vert g\Vert_{\SO}$ for some $C>0$,
we can establish that
\[ \Vert c \Vert_{2}^{2} = \sum_{k\in\Z_{\NumDim}} \int_{\Lambda^{\circ}} \vert \langle \pi(\chi) h, \pi(\lambda^{\circ}) \vvfun{f}{k} \rangle \vert^{2} \, d\mu_{\Lambda^{\circ}}(\lambda^{\circ}) \le D \, \Vert f \Vert_{\SO(G\times\Z_{d})} \, \Vert h \Vert_{\SO(G)} < \infty, \]
for some $D>0$ that depends on $\Lambda$. Hence $c$ is an element in $L^{2}(\Lambda^{\circ}\times\Z_{\NumDim})$.
For later use we note that
\begin{equation} \label{eq:1204a} \Vert c \Vert_{2}^{2} = \sum_{k\in\Z_{\NumDim}}\int_{\Lambda^{\circ}} \vert c(\lambda^{\circ},k) \vert^{2} \, d\mu_{\Lambda^{\circ}}(\lambda^{\circ}) = \sum_{k\in\Z_{\NumDim}} \int_{\Lambda^{\circ}} \vert \langle \pi(\chi-\lambda^{\circ}) h,  \vvfun{f}{k} \rangle \vert^{2} \, d\mu_{\Lambda^{\circ}}(\lambda^{\circ}).
\end{equation}
Let us take a closer look at the middle expression in \eqref{eq:1204b} with our choice of $c$.
\begin{align}
& \quad \, \sum_{j\in\Z_{\NumWin}} \Big\Vert \sum_{k\in\Z_{\NumDim}} \int_{\Lambda^{\circ}} c(\lambda^{\circ},k) \, \pi(\lambda^{\circ}) \mvfun{g}{k}{j}  \, d\mu_{\Lambda^{\circ}} \, \Big\Vert_{2}^{2} \nonumber \\
& = \sum_{j\in\Z_{\NumWin}} \Big\langle \sum_{k\in\Z_{\NumDim}} \int_{\Lambda^{\circ}} c(\lambda^{\circ},k) \, \pi(\lambda^{\circ}) \mvfun{g}{k}{j} \, d\mu_{\Lambda^{\circ}}(\lambda^{\circ}) , \sum_{l\in\Z_{\NumDim}} \int_{\Lambda^{\circ}} c(\tilde{\lambda}^{\circ},l) \, \pi(\tilde{\lambda}^{\circ}) \mvfun{g}{l}{j} \, d\mu_{\Lambda^{\circ}}(\tilde{\lambda}^{\circ})  \Big\rangle \nonumber \\
& = \sum_{j\in\Z_{\NumWin}} \sum_{k,l\in\Z_{\NumDim}} \int_{\Lambda^{\circ}} \int_{\Lambda^{\circ}} c(\lambda^{\circ},k) \, \overline{c(\tilde{\lambda}^{\circ},l)}  \, \big\langle  \pi(\lambda^{\circ}) \mvfun{g}{k}{j} , \pi(\tilde{\lambda}^{\circ}) \mvfun{g}{l}{j}  \big\rangle \, d\mu_{\Lambda^{\circ}}(\lambda^{\circ}) \, \, d\mu_{\Lambda^{\circ}}(\tilde{\lambda}^{\circ}) \nonumber \\
& = \sum_{j\in\Z_{\NumWin}} \sum_{k,l\in\Z_{\NumDim}} \int_{\Lambda^{\circ}} \int_{\Lambda^{\circ}} c(\lambda^{\circ},k) \,\big\langle  \pi(\lambda^{\circ}) \mvfun{g}{k}{j} , \pi(\tilde{\lambda}^{\circ}) \mvfun{g}{l}{j} \big\rangle \big\langle \pi(\tilde{\lambda}^{\circ}) \vvfun{f}{l}, \pi(\chi) h\big\rangle \, d\lambda^{\circ} \, \, d\tilde{\lambda}^{\circ}. \label{eq:1204c}
\end{align}
We use the fundamental identity of Gabor analysis \eqref{eq:figa} (with $\NumDim=\NumWin=1$) to establish that for all $k,l\in\Z_{\NumDim}$ and $j\in\Z_{\NumWin}$
\begin{align*} & \quad \, \int_{\Lambda^{\circ}} \big\langle  \pi(\lambda^{\circ}) \mvfun{g}{k}{j} , \pi(\tilde{\lambda}^{\circ}) \mvfun{g}{l}{j} \big\rangle   \, \big\langle \pi(\tilde{\lambda}^{\circ}) \vvfun{f}{l}, \pi(\chi) h\big\rangle \, d\mu_{\Lambda^{\circ}}(\tilde{\lambda}^{\circ}) \\
& = \int_{\Lambda} \big\langle \vvfun{f}{l}, \pi(\lambda) \mvfun{g}{l}{j} \big\rangle \, \big\langle \pi(\lambda) \pi(\lambda^{\circ}) \mvfun{g}{k}{j}, \pi(\chi) h\big\rangle \, d\mu_{\Lambda}(\lambda).\end{align*}
Indeed, two of the involved functions belong to $L^{2}(G)$, where as two belong to $\SO(G)$. Lemma \ref{le:figa-suff-cond}(ii) implies that we can use \eqref{eq:figa}. Returning to \eqref{eq:1204c} we thus have that
\begin{align}
& \quad \, \sum_{j\in\Z_{\NumWin}} \Big\Vert \sum_{k\in\Z_{\NumDim}} \int_{\Lambda^{\circ}} c(\lambda^{\circ},k) \, \pi(\lambda^{\circ}) \mvfun{g}{k}{j} \, d\mu_{\Lambda^{\circ}} \, \Big\Vert_{2}^{2}  \nonumber \\
& = \sum_{j\in\Z_{\NumWin}} \sum_{k,l\in\Z_{\NumDim}} \int_{\Lambda^{\circ}}  c(\lambda^{\circ},k) \,\int_{\Lambda} \big\langle \vvfun{f}{l}, \pi(\lambda) \mvfun{g}{l}{j} \big\rangle \, \big\langle \pi(\lambda) \pi(\lambda^{\circ}) \mvfun{g}{k}{j} , \pi(\chi) h\big\rangle \, d\mu_{\Lambda}(\lambda) \, d\lambda^{\circ} \nonumber \\
& = \sum_{j\in\Z_{\NumWin}} \sum_{k,l\in\Z_{\NumDim}} \int_{\Lambda^{\circ}}  \,\int_{\Lambda} \big\langle \vvfun{f}{l}, \pi(\lambda) \mvfun{g}{l}{j} \big\rangle \, \big\langle \pi(\lambda) \mvfun{g}{k}{j} , \pi(\chi-\lambda^{\circ}) h\big\rangle \, \big\langle \pi(\chi-\lambda^{\circ})h, \vvfun{f}{k} \big\rangle \, d\lambda \, d\lambda^{\circ} \label{eq:1204d}.
\end{align}
Combining \eqref{eq:1204b}, \eqref{eq:1204a} and \eqref{eq:1204d} yields the inequalities
\begin{align*}
& \quad \, A \sum_{k\in\Z_{\NumDim}} \int_{\Lambda^{\circ}} \vert \langle \pi(\chi-\lambda^{\circ}) h,  \vvfun{f}{k} \rangle \vert^{2} \, d\mu_{\Lambda^{\circ}}(\lambda^{\circ})  \\
& \hspace{0cm} \le \sum_{j\in\Z_{\NumWin}} \sum_{k,l\in\Z_{\NumDim}} \int_{\Lambda^{\circ}}  \,\int_{\Lambda} \langle \vvfun{f}{l}, \pi(\lambda) \mvfun{g}{l}{j} \rangle \, \langle \pi(\lambda) \mvfun{g}{k}{j} , \pi(\chi-\lambda^{\circ}) h\rangle \, \langle \pi(\chi-\lambda^{\circ})h, \vvfun{f}{k} \rangle \, d\lambda \, d\lambda^{\circ} \\
& \hspace{0cm} \le B \sum_{k\in\Z_{\NumDim}} \int_{\Lambda^{\circ}} \vert \langle \pi(\chi-\lambda^{\circ}) h,  \vvfun{f}{k} \rangle \vert^{2} \, d\mu_{\Lambda^{\circ}}(\lambda^{\circ}).
\end{align*}
Integrating over the quotient group $(G\times\ghat)/\Lambda^{\circ}$ with respect to $\chi$ and using \eqref{eq:2503a} implies the inequalities
\begin{align*}
& \quad \, A \sum_{k\in\Z_{\NumDim}} \int_{G\times\ghat} \vert \langle \pi(\chi) h,  \vvfun{f}{k} \rangle \vert^{2} \, d\mu_{G\times\ghat}(\chi) \\
& \hspace{0cm} \le \sum_{j\in\Z_{\NumWin}} \sum_{k,l\in\Z_{\NumDim}} \int_{\Lambda} \langle \vvfun{f}{l}, \pi(\lambda) \mvfun{g}{l}{j} \rangle \, \int_{G\times\ghat} \langle \pi(\lambda) \mvfun{g}{k}{j} , \pi(\chi) h\rangle \, \langle \pi(\chi)h, \vvfun{f}{k} \rangle \, d\lambda \, d\mu_{G\times\ghat}(\chi) \\
& \hspace{0cm} \le B \sum_{k\in\Z_{\NumDim}} \int_{G\times\ghat} \vert \langle \pi(\chi) h,  \vvfun{f}{k} \rangle \vert^{2} \, d\mu_{G\times\ghat}(\chi).
\end{align*}
Since $\Vert h \Vert_{2}= 1$ an application of \eqref{eq:STFT} allows us to conclude that
\begin{align}
 A \sum_{k\in\Z_{\NumDim}} \Vert \vvfun{f}{k} \Vert_{2}^{2} & \le \sum_{j\in\Z_{\NumWin}} \sum_{k,l\in\Z_{\NumDim}} \int_{\Lambda} \langle \vvfun{f}{l}, \pi(\lambda) \mvfun{g}{l}{j} \rangle \, \langle \pi(\lambda) \mvfun{g}{k}{j} , \vvfun{f}{k} \rangle \, d\mu_{\Lambda}(\lambda)  \nonumber  \\
& = \sum_{j\in\Z_{\NumWin}} \int_{\Lambda} \Big\vert \sum_{k\in\Z_{\NumDim}} \langle \vvfun{f}{k}, \pi(\lambda) \mvfun{g}{k}{j} \rangle \Big\vert^{2} \, d\mu_{\Lambda}(\lambda) \le B \sum_{k\in\Z_{\NumDim}} \Vert \vvfun{f}{k} \Vert_{2}^{2}. \label{eq:2102b} 
\end{align}
This is the statement of Theorem \ref{th:duality}(i). We have thus shown that Theorem \ref{th:duality}(ii) implies Theorem \ref{th:duality}(i). Observe that Lemma \ref{le:nec-cond}(i) implies that $(G\times \ghat)/\Lambda$ must be compact and $d$ must be finite. In particular $\Lambda^{\circ}$ had to be discrete to begin with. This proves the first part of Lemma \ref{le:nec-cond}(ii). Concerning the ``in addition'' part of Lemma \ref{le:nec-cond}(ii) let us assume that $(G\times\ghat)/\Lambda^{\circ}$ is compact and that $\Lambda^{\circ}$ is equipped with the counting measure. Note that in that case the sizes of $\Lambda$ and $\Lambda^{\circ}$ are related by $s(\Lambda) s(\Lambda^{\circ}) = 1$ (for a proof of this see, e.g., \cite[Lemma 2.1.4]{jale16-2}). The result from Lemma \ref{le:nec-cond}(i) implies that $(s(\Lambda^{\circ}))^{-1} = \s(\Lambda) \le \NumWin/\NumDim$. I.e., $\s(\Lambda^{\circ}) \ge \NumDim/\NumWin$. This finishes the proof of Lemma \ref{le:nec-cond}(ii) (note that in that result $\Lambda^{\circ}$ is replaced by $\Lambda$ and the roles of $\NumDim$ and $\NumWin$ are interchanged).

As we just showed, the assumption that the $\NumDim$-multi-window $\NumWin$-super Gabor system satisfies \eqref{eq:riesz} implies that $\Lambda^{\circ}$ is discrete. In that case \eqref{eq:cocompact-orth-measure} states that $\int_{\Lambda^{\circ}} \ \ldots \ d\mu_{\Lambda^{\circ}}(\lambda^{\circ}) = \s(\Lambda)^{-1}\sum_{\lambda^{\circ}\in\Lambda^{\circ}}$. This entails that our initial assumption \eqref{eq:1204b} has the form
\[ \frac{A}{\s(\Lambda)} \sum_{\substack{\lambda^{\circ}\in \Lambda^{\circ}\\k\in\Z_{\NumDim}}} \vert c(\lambda^{\circ},k)\vert^{2} \le \sum_{j\in\Z_{\NumWin}} \Big\Vert \frac{1}{\s(\Lambda)} \sum_{\substack{\lambda^{\circ}\in\Lambda^{\circ}\\k\in \Z_{\NumDim} }} c(\lambda^{\circ},k) \, \pi(\lambda^{\circ}) \mvfun{g}{k}{j} \, \Big\Vert_{2}^{2}\le \frac{B}{\s(\Lambda)} \sum_{\substack{\lambda^{\circ}\in \Lambda^{\circ}\\k\in\Z_{\NumDim}}} \vert c(\lambda^{\circ},k)\vert^{2} ,\]
which is the statement in Theorem \ref{th:duality}(ii).

Concerning the proof of Proposition \ref{pr:bessel-duality} observe that the upper inequality in the assumed inequalities \eqref{eq:1204b} states that the synthesis operator of the $\NumDim$-multi-window $\NumWin$-super Gabor system generated by $\mvfun{g}{k}{}$, $k\in \Z_{\NumDim}$ with time-frequency shifts from $\Lambda^{\circ}$ is bounded by $\sqrt{B}$. Equivalently, its adjoint, the analysis operator, has the same operator bound. This is exactly Proposition \ref{pr:bessel-duality}(ii). Under this assumption we showed that the upper inequality in \eqref{eq:2102b} holds. I.e., the analysis operator of the $\NumWin$-multi-window $\NumDim$-super Gabor frame generated by the functions $\mvfun{g}{}{j}$, $j\in\Z_{\NumWin}$ with time-frequency shifts along $\Lambda$ is bounded by $\sqrt{B}$. This is the statement of Proposition \ref{pr:bessel-duality}(i). The converse implication is proven with the same steps as before but with the roles of $\NumDim$ and $\NumWin$ as well as $\Lambda$ and $\Lambda^{\circ}$ interchanged. This proves Proposition \ref{pr:bessel-duality}.

It is only left to prove that Theorem \ref{th:duality}(i) implies Theorem \ref{th:duality}(ii). Let therefore $(G\times\ghat)/\Lambda$ be compact, $\NumDim< \infty$ and assume that Theorem \ref{th:duality}(i) holds. We build again on the ideas of Janssen \cite{ja95} (see line 3 in the proof of Theorem 3.1 in \cite{ja95}) and will construct certain functions in $L^{2}(G\times\Z_{d})$ which upon insertion into the assumed frame inequalities,
\begin{equation} \label{eq:1709b} A \, \Vert f \Vert_{2}^{2} \le \sum_{j\in \Z_{\NumWin}} \int_{\Lambda} \Big\vert \sum_{k\in\Z_{\NumDim}} \big\langle \vvfun{f}{k} , \pi(\lambda) \mvfun{g}{k}{j}\big\rangle \Big\vert^{2} \, d\mu_{\Lambda}(\lambda) \le B \, \Vert f \Vert_{2}^{2},\end{equation}
enable us to deduce the Riesz sequence inequalities of Theorem \ref{th:duality}(ii). 
By assumption the multi-window Gabor system generated by the functions $\mvfun{g}{}{j}$, $j\in\Z_{\NumWin}$ in $L^{2}(G\times\Z_{\NumDim})$ with time-frequency shifts along $\Lambda\times\{0\}\subseteq (G\times\ghat)\times(\Z_{\NumDim}\times\widehat{\Z_{\NumDim}})$ is a frame for $L^{2}(G\times\Z_{\NumDim})$. 
Let $S$ be the associated frame operator on $L^{2}(G\times\Z_{\NumDim})$. The functions $S^{-1/2}\mvfun{g}{}{j}$, $j\in\Z_{\NumWin}$ generate a multi-window super Gabor frame for $L^{2}(G\times\Z_{\NumDim})$ with frame bounds $A=B=1$ and such that for all $f^{1},f^{2}\in L^{2}(G\times\Z_{\NumDim})$
\begin{equation} \label{eq:1709a} \langle f^{1}, f^{2}\rangle = \sum_{j\in \Z_{\NumWin}} \int_{\Lambda} \Big( \sum_{k\in \Z_{\NumDim}} \big\langle \vvfun{f^{1}}{k}, \pi(\lambda) S^{-1/2} \mvfun{g}{k}{j}  \big\rangle \Big) \, \Big( \sum_{l\in \Z_{\NumDim}} \big\langle \pi(\lambda) S^{-1/2} \mvfun{g}{l}{j}, \vvfun{f^{2}}{l}  \big\rangle \Big) \, d\mu_{\Lambda}(\lambda). \end{equation}
The characterization of dual frames by Theorem \ref{th:wex-raz} implies that for all $k,l\in\Z_{\NumDim}$ and $\lambda^{\circ}\in\Lambda^{\circ}$
\begin{equation} \label{eq:1903bb} \sum_{j\in \Z_{\NumWin}}  \langle S^{-1/2}\mvfun{g}{l}{j}, \pi(\lambda^{\circ})S^{-1/2}\mvfun{g}{k}{j} \rangle = \begin{cases} \s(\Lambda) \, & \text{if} \ (\lambda^{\circ},k)=(0,l),\\ 0 & \text{otherwise}. \end{cases}\end{equation}

For the rest of the proof we fix a finite sequence $c\in \ell^{2}(\Lambda^{\circ}\times\Z_{\NumDim})$. For each $m\in\Z_{\NumWin}$ we define the function $f^{m}\in L^{2}(G\times\Z_{\NumDim})$ where
\[ \vvfun{f^{m}}{k} = d^{-1/2} \sum_{\substack{\lambda^{\circ}\in \Lambda^{\circ} \\ l \in \Z_{\NumDim}}} \overline{c(\lambda^{\circ}, k)} \, \pi(\lambda^{\circ})^{*} \, S^{-1/2} \mvfun{g}{l}{m}, \ k\in \Z_{\NumDim}.\]

We now insert these functions into the frame inequalities \eqref{eq:1709b} and sum over $m\in \Z_{\NumWin}$. 
The first observation is that
\begin{align}
& \sum_{m\in \Z_{\NumWin}} \Vert f^{m} \Vert_{2}^{2} = \sum_{\substack{m\in\Z_{\NumWin} \\ k\in\Z_{\NumDim}} } \big\langle \vvfun{f^{m}}{k} , \vvfun{f^{m}}{k}\big\rangle \nonumber  \\
& = \frac{1}{d}  \sum_{\substack{\lambda^{\circ}, \tilde{\lambda}^{\circ}\in \Lambda^{\circ} \\ l,\tilde{l} \in \Z_{\NumDim} \\ k\in\Z_{\NumDim} }} \overline{c(\lambda^{\circ}, k) } \, c(\tilde{\lambda}^{\circ}, k) \, \sum_{m\in\Z_{\NumWin}} \big\langle \pi(\lambda^{\circ})^{*} \, S^{-1/2} \mvfun{g}{l}{m}, \pi(\tilde{\lambda}^{\circ})^{*} \, S^{-1/2} \mvfun{g}{\tilde{l}}{m} \big\rangle \nonumber \\
& \stackrel{\eqref{eq:1903bb}}{=} \frac{\s(\Lambda)}{d} \sum_{\substack{k,l\in \Z_{\NumDim} \\\lambda^{\circ}\in \Lambda^{\circ} } } \vert c(\lambda^{\circ}, k) \vert^{2} = \s(\Lambda) \sum_{\substack{ k\in\Z_{\NumDim} \\ \lambda^{\circ}\in \Lambda^{\circ} }} \vert c(\lambda^{\circ}, k) \vert^{2} < \infty. \label{eq:1709c}
\end{align}
The calculation for the expression in the middle of the inequalities in \eqref{eq:1709a} is a bit more involved.

\begingroup
\allowdisplaybreaks
\begin{align}
& \sum_{j,m\in\Z_{\NumWin}} \int_{\Lambda} \Big\vert \sum_{k\in\Z_{\NumDim}} \big\langle \vvfun{f^{m}}{k}, \pi(\lambda) \mvfun{g}{k}{j}\big\rangle \Big\vert^{2} \, d\mu_{\Lambda}(\lambda) \nonumber \\
& = \sum_{\substack{ j,m\in \Z_{\NumWin} \nonumber \\ k,\tilde{k}\in \Z_{\NumDim} } } \int_{\Lambda} \big\langle \vvfun{f^{m}}{k} , \pi(\lambda) \mvfun{g}{k}{j}\big\rangle \, \big\langle \pi(\lambda) \mvfun{g}{\tilde{k}}{j}, \vvfun{f^{m}}{\tilde{k}} \big\rangle \, d\mu_{\Lambda}(\lambda) \nonumber \\
& = \frac{1}{d} \sum_{\substack{ j\in\Z_{\NumWin} \nonumber \\ k,\tilde{k} \in\Z_{\NumDim} \\ \lambda^{\circ}, \tilde{\lambda}^{\circ}\in \Lambda^{\circ}}} \overline{c(\lambda^{\circ}, k)} \, c(\tilde{\lambda}^{\circ},\tilde{k}) \nonumber  \\
& \hspace{2cm} \sum_{m\in\Z_{\NumWin}} \int_{\Lambda} \Big( \sum_{\tilde{l}\in\Z_{\NumDim}} \big\langle \pi(\tilde{\lambda}^{\circ}) \mvfun{g}{\tilde{k}}{j}, \pi(\lambda) S^{-1/2} \mvfun{g}{\tilde{l}}{m} \big\rangle \Big)   \nonumber  \\
& \hspace{4cm} \cdot \Big( \sum_{l\in\Z_{\NumDim}} \big\langle \pi(\lambda) S^{-1/2} \mvfun{g}{l}{m}, \pi(\lambda^{\circ}) \mvfun{g}{k}{j} \big\rangle \Big) \, d\mu_{\Lambda}(\lambda) \nonumber \\
& \stackrel{\eqref{eq:1709a}}{=} \sum_{\substack{ j\in\Z_{\NumWin} \\ k,\tilde{k} \in\Z_{\NumDim} \nonumber  \\ \lambda^{\circ}, \tilde{\lambda}^{\circ}\in \Lambda^{\circ}}} \overline{c(\lambda^{\circ}, k)} \, c(\tilde{\lambda}^{\circ},\tilde{k}) \big\langle \pi(\tilde{\lambda}^{\circ}) \mvfun{g}{\tilde{k}}{j},\pi(\lambda^{\circ}) \mvfun{g}{k}{j} \big\rangle \nonumber \\
& = \sum_{j\in\Z_{\NumWin}} \Big\Vert \sum_{\substack{\lambda^{\circ}\in\Lambda^{\circ} \\ k\in\Z_{\NumDim} }} c(\lambda^{\circ}) \, \pi(\lambda^{\circ}) \mvfun{g}{k}{j} \, \Big\Vert_{2}^{2}. \label{eq:1709d}
\end{align}
\endgroup

Combining the assumed frame inequalities \eqref{eq:1709b} and the established equalities in \eqref{eq:1709c} and \eqref{eq:1709d} yields the desired inequalities of Theorem \ref{th:duality}(ii).

\end{proof}


\section{Gabor frames with generators in $\SO(\R^{m}\times\Z_{\NumDim}\times\Z_{\NumWin})$}
\label{sec:existence-mws-frames-for-R}

In this section we go from the general setting on groups to the much more specific settings of multi-window super Gabor frames for $L^{2}(\R^{m}\times\Z_{\NumDim})$ that are generated by a function in the Feichtinger algebra $\SO(\R^{m}\times\Z_{\NumDim}\times\Z_{\NumWin})$. By Lemma \ref{le:nec-cond}(i) we know that the desire to construct multi-window super Gabor frames for $L^{2}(\R^{m}\times\Z_{\NumDim})$ implies that $\NumDim<\infty$ and that the subgroup of the 
time-frequency plane $\Lambda\subset \R^{2d}$ is co-compact, hence $\La$ must have the form
\begin{equation} \label{eq:co-compact-in-Rm} \Lambda = \textnormal{A} ( \R^{m_{1}} \times \Z^{m_{2}} ), \ \ \text{for some} \ \textnormal{A}\in \textnormal{GL}_{2m}(\R) \ \text{and where} \ m_{1}+m_{2} = 2m.\end{equation}
Furthermore, if $m_{1}=0$, then $\Lambda$ is discrete and Lemma \ref{le:nec-cond}(i) imposes the additional condition $\NumDim\, s(\Lambda) \le \NumWin$ on the minimum number $\NumWin$ of generators of the multi-window (super) Gabor frame, where $s(\Lambda)$ is the measure of the fundamental domain of $\La$.

We will give an elementary proof of the following important result, which is a special case of Theorem \ref{th:exist}. Again, it might be instructive to consider this for the classical case $m=1$, $\NumDim= 1$, $\La = a\Z\times b \Z$.

\begin{theorem} \label{th:existence-mwsframes-rn} For any closed and co-compact subgroup $\La$ in $\R^{2m}$ and any $\NumDim<\infty$ there exists a function $g\in \SO(\R^{m}\times\Z_{\NumDim}\times\Z_{\NumWin})$, $n<\infty$, that generates a multi-window super Gabor frame for $L^{2}(\R^{m}\times\Z_{\NumDim})$ with respect to the subgroup $\La$, and such that the associated frame operator  is invertible on $\SO(\R^{m}\times\Z_{\NumDim})$.  
\end{theorem}

\begin{corollary} For any choice of $m,m_{1},m_{2}\in\N$ such that $m_{1}+m_{2}=2m$, $A\in \textnormal{GL}_{m}(\R)$, and any $\NumDim<\infty$ there exist functions $g,h\in\SO(\R^{m}\times\Z_{\NumDim}\times\Z_{\NumWin})$, $\NumWin<\infty$ that generate dual multi-window super Gabor frames for $L^{2}(\R^{m}\times\Z_{\NumDim})$ with respect to the closed subgroup $\La = A(\R^{m_{1}}\times\Z^{m_{2}})$. That is, for all $f\in L^{2}(\R^{m}\times\Z_{\NumDim})$,
\[ \vvfun{f}{k} = \sum_{j\in\Z_{\NumWin}} \int_{\R^{m_1}} \sum_{\lambda\in\Z^{m_{2}}} \Big( \sum_{k'\in\Z_{\NumDim}} \langle \vvfun{f}{k'} , \pi( A (x,\la) ) \mvfun{g}{k'}{j} \rangle \Big) \ \pi(A(x,\la)) \mvfun{h}{k}{j} \, dx \ \ \text{for all} \ \ k\in\Z_{\NumDim}\]
or, in terms of the $\lmodule$-valued inner products from Section \ref{sec:module},
\[ f = \sum_{j\in \Z_{\NumDim}} \lhs{f}{\mvfun{g}{}{j}} \cdot \mvfun{h}{}{j} \ \ \text{for all} \ \ f\in L^{2}(\R^{m}\times\Z_{\NumDim}). \]
\end{corollary}


\begin{proof}[Proof of Theorem \ref{th:existence-mwsframes-rn}] In case $\Lambda$ is the full time-frequency plane $\La = \R^{2m}$, then the statements follow from Remark \ref{rem:continuous-mw-case-generator-in-so} and Remark \ref{rem:continuous-super-case-generator-in-so} in Section \ref{sec:continuous-case}. From now on we assume that $\La$ has the form as in \eqref{eq:co-compact-in-Rm} with $m_{2}>0$. Let $g$ be a function in $\SO(\R^{m}\times\Z_{\NumDim})$ such that $\{\vvfun{g}{k}\}_{k\in\Z_{\NumDim}}$ is an orthonormal (with respect to the inner-product in $L^{2}(\R^{m})$) collection of functions. This can be achieved, e.g., with the Gram-Schmidt procedure as in the proof of Lemma \ref{le:1604}. For each $N\in\N$ we define $\Lambda_{N}$ to be the subgroup of $\R^{2m}$ given by
\begin{equation} \label{eq:Lambda-N} \Lambda_{N} = \textnormal{A} \big(\R^{m_{1}} \times (N^{-1} \Z)^{m_{2}}\big).\end{equation}
Observe that $\Lambda \subseteq \Lambda_{N}$ (with equality for $N=1$), that $\s(\Lambda_{N}) = N^{-m_{2}} \s(\Lambda)$, and that the quotient $\Lambda_{N}/\Lambda$ contains $N^{m_{2}}$ elements. For later use, we let $\{\chi_{j} \, : \, j\in \Z_{N^{m_{2}}}\}\subset \R^{2m}$ be a collection of $N^{m_{2}}$ coset representatives for the quotient group $\La_{N}/\La$. 
The adjoint group of $\La_{N}$ is 
\[ \Lambda_{N}^{\circ} = \begin{pmatrix} \mathbf{0} & \mathbf{I}_{m} \\ -\mathbf{I}_{m} & \mathbf{0} \end{pmatrix} (\textnormal{A}^{-1})^{\top} \big(\{0\}^{m_{1}}\times (N\Z)^{m_{2}} \big)\]
and we have the inclusion $\Lambda_{N}^{\circ} \subseteq \Lambda^{\circ}$ with equality for $N=1$.
The subgroup $\La_{N}^{\circ}$ becomes increasingly sparse with increasing $N$. 

Since all $\vvfun{g}{k}$, $k\in\Z_{\NumDim}$ belong to $\SO(\R^{m})$  it follows from Lemma \ref{le:s0-properties}(vii) and (viii) that the sequence $\{ \langle \vvfun{g}{k} , \pi(\lambda^{\circ}) \vvfun{g}{k'} \rangle \}_{\lambda^{\circ}\in\Lambda^{\circ}}$ is absolutely summable for any $k,k'\in\Z_{\NumDim}$. We can therefore find an $N$ that is sufficiently large so that
\begin{equation} \label{eq:1704b}  \sum_{\substack{\lambda^{\circ}\in \Lambda^{\circ}_{N} \\ \lambda^{\circ} \ne 0 }} \vert \langle \vvfun{g}{k}, \pi(\lambda^{\circ}) \vvfun{g}{k'} \rangle \vert < d^{-1} \ \ \text{for all} \ \ k,k'\in\Z_{\NumDim}.\end{equation}
With an $N$ fixed so that \eqref{eq:1704b} holds, we define $\tilde{g}$ to be the function in $\SO(\R^{m}\times\Z_{\NumDim})$ given by $\tilde{g} = \sqrt{\s(\Lambda_{N})} \, g$. The frame operator for the super Gabor system generated by $\tilde{g}$ with respect to the subgroup $\Lambda_{N}$ is the operator
\begin{align*} & S_{\tilde{g}, \Lambda_{N}} : L^{2}(\R^{m} \times \Z_{\NumDim}) \to L^{2}(\R^{m}\times\Z_{\NumDim}), \\
& S_{\tilde{g},\Lambda_{N}} f(\,\cdot \, , k) = \int_{\Lambda_{N}} \Big( \sum_{k'\in\Z_{\NumDim}} \langle \vvfun{f}{k'} , \pi(\lambda) \vvfun{\tilde{g}}{k'} \rangle \Big) \, \pi(\lambda) \vvfun{\tilde{g}}{k} \, d\lambda, \ k\in\Z_{\NumDim}.\end{align*}
Since $g\in \SO(\R^{m}\times\Z_{\NumDim})$ we know from Lemma \ref{le:s0-implies-bessel} that this operator is bounded. Also Lemma \ref{le:figa-suff-cond} states that we may apply \eqref{eq:figa}, and so
\begin{align} \nonumber & S_{\tilde{g},\Lambda_{N}} f(\,\cdot \, , k) = \frac{1}{\s(\Lambda_{N})} \sum_{\substack{\lambda^{\circ}\in\Lambda^{\circ} \\ k'\in\Z_{\NumDim}}}  \langle \vvfun{\tilde{g}}{k} , \pi(\lambda^{\circ}) \vvfun{\tilde{g}}{k'} \rangle \, \pi(\lambda^{\circ}) \vvfun{f}{k'} \\
& = \sum_{\substack{\lambda^{\circ}\in\Lambda^{\circ} \\ k'\in\Z_{\NumDim}}}  \langle \vvfun{g}{k} , \pi(\lambda^{\circ}) \vvfun{g}{k'} \rangle \, \pi(\lambda^{\circ}) \vvfun{f}{k'} = \vvfun{f}{k} + \sum_{\substack{\lambda^{\circ}\in\Lambda^{\circ}\backslash\{0\} \\ k'\in\Z_{\NumDim}}}  \langle \vvfun{g}{k} , \pi(\lambda^{\circ}) \vvfun{g}{k'} \rangle \, \pi(\lambda^{\circ}) \vvfun{f}{k'}  \label{eq:1704a}
\end{align}
In the last equality we used the orthonormality of the functions $\{\vvfun{g}{k}\}_{k\in\Z_{\NumDim}}$.
Recall from Lemma \ref{le:s0-properties}(x) that 
\[ \Vert \cdot \Vert_{\SO\otimes\ell^{1}} : \SO(\R^{m}\times\Z_{\NumDim}) \to \R_{0}^{+}, \ \Vert f \Vert_{\SO\otimes\ell^{1}} = \sum_{k\in\Z_{\NumDim}} \Vert \vvfun{f}{k} \Vert_{\SO(\R^{m})},\]
is a norm on $\SO(\R^{m}\times\Z_{\NumDim})$ that is equivalent to the norm on $\SO(\R^{m}\times\Z_{\NumDim})$ via \eqref{eq:norm}.
We will use this norm on $\SO(\R^{m}\times\Z_{\NumDim})$. We now show that $S_{\tilde{g},\Lambda_{\N}}$ is invertible on $\SO(\R^{m}\times\Z_{\NumDim})$ by proving that it is close to the identity operator $\textnormal{Id}$ on $\SO(\R^{m}\times\Z_{\NumDim})$.
Indeed,
\begin{align*}
& \Vert \textnormal{Id} - S_{\tilde{g},\Lambda_{N}} \Vert_{\textnormal{op},\SO\to\SO} = \sup_{\substack{f\in\SO(\R^{m}\times\Z_{\NumDim}) \\ \Vert f \Vert_{\SO\otimes\ell^{1}} = 1}} \big\Vert f - S_{\tilde{g},\Lambda_{N}} f \big\Vert_{\SO\otimes\ell^{1}} \\
& = \sup_{\substack{f\in\SO(\R^{m}\times\Z_{\NumDim}) \\ \Vert f \Vert_{\SO\otimes\ell^{1}} = 1}} \sum_{k\in\Z_{\NumDim}} \big\Vert \vvfun{f}{k} - S_{\tilde{g},\Lambda_{N}} f(\,\cdot\, , k) \big\Vert_{\SO} \\
& \stackrel{\eqref{eq:1704a}}{=} \sup_{\substack{f\in\SO(\R^{m}\times\Z_{\NumDim}) \\ \Vert f \Vert_{\SO\otimes\ell^{1}} = 1}} \sum_{k\in\Z_{\NumDim}} \Big\Vert \sum_{\substack{\lambda^{\circ}\in\Lambda^{\circ}\backslash\{0\} \\ k'\in\Z_{\NumDim}}}  \langle \vvfun{g}{k} , \pi(\lambda^{\circ}) \vvfun{g}{k'} \rangle \, \pi(\lambda^{\circ}) \vvfun{f}{k'} \Big\Vert_{\SO} \\
& \le \sup_{\substack{f\in\SO(\R^{m}\times\Z_{\NumDim}) \\ \Vert f \Vert_{\SO\otimes\ell^{1}} = 1}}  \sum_{\substack{\lambda^{\circ}\in\Lambda^{\circ}\backslash\{0\} \\ k,k'\in\Z_{\NumDim}}} \big\vert \langle \vvfun{g}{k} , \pi(\lambda^{\circ}) \vvfun{g}{k'} \rangle \big\vert \,  \big\Vert \pi(\lambda^{\circ}) \vvfun{f}{k'} \big\Vert_{\SO} \\
& \stackrel{\eqref{eq:1704b}}{<} \sup_{\substack{f\in\SO(\R^{m}\times\Z_{\NumDim}) \\ \Vert f \Vert_{\SO\otimes\ell^{1}} = 1}}  \sum_{k,k'\in\Z_{\NumDim}} d^{-1} \,  \big\Vert \pi(\lambda^{\circ}) \vvfun{f}{k'} \big\Vert_{\SO} \\
& = \sup_{\substack{f\in\SO(\R^{m}\times\Z_{\NumDim}) \\ \Vert f \Vert_{\SO\otimes\ell^{1}} = 1}}  \sum_{k'\in\Z_{\NumDim}} \, \Vert \vvfun{f}{k'} \Vert_{\SO} = 1.
\end{align*}
This shows that the the frame operator $S_{\tilde{g},\Lambda_{N}}$ is invertible on $\SO(\R^{m}\times\Z_{\NumDim})$. In turn, this implies that there are some constants $A,B>0$ such that
\begin{equation} \label{eq:1505a} A \, \Vert f \Vert_{2}^{2} \le \int_{\Lambda_{N}} \Big\vert \Big( \sum_{k\in\Z_{\NumDim}} \langle \vvfun{f}{k} , \pi(\lambda) \vvfun{\tilde{g}}{k} \rangle \Big) \Big\vert^{2} \, d\lambda \le B \, \Vert f \Vert_{2}^{2} \ \ \text{for all} \ \ f\in L^{2}(\R^{m}\times\Z_{\NumDim}).\end{equation}
We will now rewrite the expression in the middle from a single window super Gabor system with time-frequency shifts from $\La_{N}$ into a multi-window super Gabor system with time-frequency shifts from the subgroup $\Lambda$. We take $\NumWin:= N^{m_{2}}$ and use the earlier defined coset representative $\{\chi_{j} \, : \, j\in\Z_{\NumWin}\}$ of the quotient group $\Lambda_{N}/\Lambda$: 
\begin{align*}
\int_{\Lambda_{N}} \Big\vert \Big( \sum_{k\in\Z_{\NumDim}} \langle \vvfun{f}{k} , \pi(\lambda) \vvfun{\tilde{g}}{k} \rangle \Big) \Big\vert^{2} \, d\lambda = \sum_{j\in\Z_{\NumWin}} \int_{\Lambda} \Big\vert \Big( \sum_{k\in\Z_{\NumDim}} \langle \vvfun{f}{k} , \pi(\lambda) [\pi(\chi_{j})\vvfun{\tilde{g}}{k}] \rangle \Big) \Big\vert^{2} \ d\lambda.
\end{align*}
Using this in \eqref{eq:1505a} implies that the $\NumWin$-\emph{multi-window} $\NumDim$-super Gabor system generated by the function $h\in \SO(\R^{m}\times\Z_{\NumDim}\times\Z_{\NumWin})$, where $h(\, \cdot \, , k,j) = \pi(\chi_{j}) \vvfun{\tilde{g}}{k}$, and with time-frequency shifts from the subgroup $\La$ is a frame for $L^{2}(\R^{m})$.
\end{proof}

\begin{corollary} Any orthonormal collection of functions $\{g_{k}\}_{k=1}^{d}$ in $\SO(\R^{m})$ generates a super Gabor frame for $L^{2}(\R^{m}\times\Z_{\NumDim})$ with respect to a subgroup $\Lambda_{N}$ (as in \eqref{eq:Lambda-N}) for sufficiently large $N$.
\end{corollary}

We now show the existence of Gabor frames with an atom in $\SO(\R^m)$ for non-rational lattices with $s(\La)<1$.
The proof is based on the interpretation of Gabor frames as generators of projective modules over noncommutative tori which allows us to use deep theorems on their structure due to Rieffel \cite{ri88}.   

First we clarify what we mean by a non-rational lattice.
Given a lattice $\La$ in $\R^{2m}$ and let $\{e_1,...,e_{2m}\}$ be a basis for $\La$. We denote by $\Omega(z,z^\prime)$ the symplectic form of $z,z^\prime\in\R^{2m}$. Then the unitaries $\{\pi(e_i):\,i=1,..,2m\}$ satisfy $\pi(e_j)\pi(e_i)=e^{2\pi \Omega(e_i,e_j)i}\pi(e_i)\pi(e_j)$ for $i,j=1,..,2m$. The $C^*$-algebra generated by $\{\pi(e_i):\,i=1,..,2m\}$ is a noncommutative torus which is isomorphic to $C^*(\La,c)$.  The structure of a noncommutative torus depends crucially on the numbers $\Omega(e_i,e_j)$ being rational or irrational. We call the lattice $\La$ not rational if there exists at least one pair of indices $i,j\in\{1,..,2m\}$ such that $\Omega(e_i,e_j)$ is not rational. Note that the set of all not rational lattices is dense in the set of all lattices.

\begin{theorem}
\label{th:existence-lattice-less-than-1}
If $\La$ is a non-rational lattice in $\R^{2m}$ with $\s(\La)<1$, then there exists a function $g\in\SO(\R^{m})$ such that $\lhs{g}{g}$ is a projection in $(\ell^1(\La),\cocy)$. That is, there exists a function $g\in\SO(\R^{m})$ such that the Gabor system $\{\pi(\la)g\}_{\la\in\La}$ is a tight frame for $L^2(\R^m)$.
\end{theorem}
\begin{proof}

Theorem \ref{th:existence-mwsframes-rn} states that $\SO(\R^{m})$ is a finitely generated projective $\mathcal{A}$-module, i.e., there exists (dependent on $\La$) finitely many pairs of functions in $\SO(\R^{m})$ that generate dual multi-window Gabor frames for $L^{2}(\R^{m})$. Let us denote them as $g$ and $h$ in $\SO(\R^{m}\times\Z_{\NumWin})$ for some $\NumWin\in\N$. By Corollary \ref{cor:duality-multi} and/or Theorem \ref{th:wex-raz} we know that this is the case if and only if
\begin{equation} \label{eq:0709a} \sum_{j\in\Z_{\NumWin}} \rhs{\vvfun{g}{j}}{\vvfun{h}{j}} = \textnormal{Id}.\end{equation}
If we take the trace on either side of this equality, then we find that (see  Theorem \ref{th:wex-raz})
\[ \tr_{\rmodule}\Big( \textstyle \sum\limits_{j\in\Z_{\NumWin}} \rhs{\vvfun{g}{j}}{\vvfun{h}{j}} \Big) = \textstyle\sum\limits_{j\in\Z_{\NumWin}} \langle \vvfun{h}{j},\vvfun{g}{j}\rangle= \s(\Lambda).\]
Since $\s(\La)< 1$, we conclude that $\tr_{\rmodule}( \textstyle \sum_{j\in\Z_{\NumWin}} \rhs{\vvfun{g}{j}}{\vvfun{h}{j}})< 1$.
By Corollary \ref{cor:duality-multi} the statement in \eqref{eq:0709a} is equivalent to the statement that $P=(\lhs{g_i}{h_j})_{i,j=1}^{n}$ is an idempotent element of $\textnormal{M}_{\NumWin}(\lmodule)$ and that $\SO(\R^n)$ is isomorphic to $P\lmodule^n$. Furthermore, $\tr_{\textnormal{M}_{\NumWin}}(P) = \s(\La)<1$.

Since $\La$ is a \emph{non}-rational lattice, $\s(\La)<1$ and $P$ is a projection in $\textnormal{M}_{\NumWin}(\lmodule)$, the result of \cite[Corollary 7.10]{ri88} implies the existence of a projection $p\in \lmodule$ so that $\SO(\R^m)\cong p\lmodule$.
Since $\SO(\R^m)$ is a finitely generated projective $\lmodule$-module, it is self-dual. We can apply \cite[Proposition 7.3]{ri10-2}, which concludes that there exists a function $\tilde{g}\in\SO(\R^m)$ (which depends on the particular lattice $\La$) such that the projection $p$ is given by $p=\lhs{\tilde{g}}{\tilde{g}}$. By Theorem \ref{th:duality-for-mws-module} (for $\NumDim=\NumWin=1$) this is equivalent to the statement that $\{\pi(\lambda)\tilde{g}\}_{\la\in\La}$ is a (in fact, tight) Gabor frame for $L^{2}(\R^{m})$. 
\end{proof}

\begin{remark} Recall the following result by Feichtinger and Kaiblinger \cite[Theorem 1.1]{feka04}: for any function $g\in \SO(\R^m)$ the set of matrices $A\in \textnormal{GL}_{2m}(\R)$ for which $\{\pi(\lambda)\}_{\lambda\in\Lambda}$, $\Lambda = A \Z^{2m}$ is a Gabor frame for $L^{2}(\R^{m})$ is open in $\textnormal{GL}_{2m}(\R)$. This, together with Theorem \ref{th:existence-lattice-less-than-1}, implies that it is also possible to find a function $g\in \SO(\R^{m})$ which generates a Gabor frame for $L^{2}(\R^{m})$ with respect to certain rational lattices $\Lambda$ in $\R^{2m}$ with $\s(\La)<1$. However, this result does not allow us to deduce this fact for any rational lattice.
\end{remark}


\begin{corollary} If $\Lambda$ is a non-rational lattice in $\R^{2m}$ such that $\s(\La) \in [n-1,n)$ for some $n\in\N$, then there exists a multi-window Gabor system with exactly $n$ generators in $\SO(\R^{m})$ that forms a frame for $L^{2}(\R^{m})$. 
\end{corollary}
\begin{proof} Let $A\in \textnormal{GL}_{2m}(\R)$ be such that $\Lambda = A \Z^{2m}$. The lattice $\Lambda_{n} = A ( \Z^{2m-1}\times \frac{1}{n}\Z)$ is such that $\s(\Lambda_{n}) \in [\tfrac{n-1}{n},1)$ and so $\s(\Lambda_{n})< 1$. Furthermore, we can find $n$ coset representatives $\chi_{n}\in \R^{2m}$ of $\Lambda_{n}/\Lambda$ such that  for all $\lambda_{n}\in\Lambda_{n}$ there exists a unique $\lambda \in \Lambda$ and $\chi_{n}\in \R^{2m}$ such that $\lambda_{n} = \lambda + \chi_{n}$. Since $\s(\Lambda_{n})< 1$ and $\Lambda_{n}$ also is non-rational we know by Theorem \ref{th:existence-lattice-less-than-1} that there is some $g\in \SO(\R^{m})$ that generates a tight Gabor frame for $L^{2}(\R^{m})$ with respect to $\Lambda_{n}$. With the same arguments as in the proof of Theorem \ref{th:existence-mwsframes-rn} we can turn this into a statement about an $n$-multi-window Gabor system with respect to the lattice $\Lambda$,
\[ \sum_{\lambda_{n}\in\Lambda_{n}} \big\vert \big\langle f, \pi(\lambda) g \big\rangle \big\vert^{2} = \sum_{j=1}^{n} \sum_{\lambda\in\Lambda} \big\vert \big\langle f, \pi(\lambda) [\pi(\chi_{j}) g] \big\rangle\big\vert^{2}. \]
I.e., the functions $\{\pi(\chi_{j})g\}_{j=1,\ldots,n}$ in $\SO(\R^{m})$ generate a tight Gabor frame for $L^{2}(\R^{m})$ with respect to the lattice $\Lambda$. 
\end{proof}

\begin{proposition}
If $\La$ is a non-rational lattice in $\R^{2m}$ with $\s(\La)<1$, then there exist two functions $g_1,g_2\in\SO(\R^{m})$ such that for $i=1,2$ each Gabor system $\{\pi(\la)g_i\}_{\la\in\La}$ is a frame for $L^{2}(\R^{m})$ and, furthermore, $\overline{\textnormal{span}}\{\pi(\lac)g_1\}_{\lac\in\Lac} \oplus \overline{\textnormal{span}}\{\pi(\lac)g_2\}_{\lac\in\Lac}=L^2(\R^m)$.  \end{proposition}
\begin{proof}
By Theorem \ref{th:existence-lattice-less-than-1} there exists a function $g_{1}\in \SO(\R^{m})$ such that $\{\pi(\la) g_{1}\}_{\la\in\La}$ is a tight Gabor frame for $L^{2}(\R^{m})$. Equivalently (Theorem \ref{th:duality-for-mws-module}), $\lhs{g_{1}}{g_{1}}$ is a projection from $L^{2}(\R^{m})$ onto $V:= \overline{\textnormal{span}}\{\pi(\lambda^{\circ})g_{1}\}_{\lac\in\Lac}$. But then also $\textnormal{Id}-\lhs{g_{1}}{g_{1}}$ is a projection from $L^{2}(\R^{m})$ into $V^{\perp}$. Observe that $\tr_{\lmodule}(\textnormal{Id}-\lhs{g}{g}) = 1 - \s(\Lambda) < 1$. By the same argument as in the proof of Theorem \ref{th:existence-lattice-less-than-1} we can use a combination of Rieffel's results from \cite{ri88,ri10-2} to conclude that there exists a function $g_{2}\in\SO(\R^{m})$ such that $\lhs{g_{2}}{g_{2}}$ equals the projection $\textnormal{Id}-\lhs{g}{g}$. But then $\{\pi(\la)g_{2}\}_{\la\in\La}$ is a Gabor frame for $L^{2}(\R^{m})$ and $V^{\perp} = \overline{\textnormal{span}}\{\pi(\lac) g_{2}\}_{\lac\in\Lac}$. 
\end{proof}

\section{Appendix, proofs of Section \ref{sec:mws-gabor}} \label{sec:app}

The following proofs will be closely related to the ones that can be found in \cite{jale16-2}, which is concerned with these results for Gabor frames with the parameters $\NumDim = \NumWin = 1$.  

\begin{lemma} \label{le:janssen-function} Let $f^{1},f^{2}\in L^{2}(G\times\Z_{\NumDim})$ and $g,h\in L^{2}(G\times\Z_{\NumDim}\times\Z_{\NumWin})$. The function 
\begin{align*} \psi: G\times\ghat\to\C, \ \psi(\chi) & = \sum_{j\in\Z_{\NumWin}} \big(C_{g,G\times\ghat}f^{1}\big)(\chi, j) \, \overline{\big(C_{h,G\times\ghat}f^{2}\big)(\chi,j)} \\
& \equiv \sum_{j\in\Z_{\NumWin}} \Big( \sum_{k\in\Z_{\NumDim}} \big\langle \vvfun{f^{1}}{k}, \pi(\chi) \mvfun{g}{k}{j}\big\rangle \Big) \, \Big( \sum_{l\in \Z_{\NumDim}} \big\langle \pi(\chi) \mvfun{h}{l}{j}, \vvfun{f^{2}}{l} \big\rangle \Big) \end{align*}
belongs to $L^{1}(G\times\ghat)$ and its symplectic Fourier transform is 
\[ \mathcal{F}_{s} (\psi)(\chi) = \sum_{ k,l\in\Z_{\NumDim} } \Big( \big\langle \pi(\chi) \vvfun{f^{1}}{k},  \vvfun{f^{2}}{l} \big\rangle \,  \sum_{j\in\Z_{\NumWin}}\big\langle \mvfun{h}{l}{j}, \pi(\chi) \mvfun{g}{k}{j} \big\rangle \Big), \ \chi\in G\times\ghat. \]
The periodization of the function $\psi$ with respect to a closed subgroup $\Lambda$ of $G\times\ghat$ is the function
\begin{align*} \varphi(\chi) & = \big\langle  C_{g,\Lambda} \pi(\chi)f^{1} , C_{h,\Lambda} \pi(\chi)f^{2}\big\rangle \\
& \equiv \sum_{j\in\Z_{\NumWin}} \int_{\Lambda} \Big( \sum_{k\in\Z_{\NumDim}} \big\langle \pi(\chi) \vvfun{f^{1}}{k}, \pi(\lambda) \mvfun{g}{k}{j}\big\rangle \Big) \, \Big( \sum_{l\in \Z_{\NumDim}} \big\langle \pi(\lambda) \mvfun{h}{l}{j}, \pi(\chi) \vvfun{f^{2}}{l} \big\rangle \Big)  \, d\mu_{\Lambda}(\lambda) \end{align*}
which is an element in $L^{1}((G\times\ghat)/\Lambda)$. The symplectic Fourier transform of $\varphi$ is given by
\begin{align*} \mathcal{F}_{s}(\varphi)(\lambda^{\circ}) & = \mathcal{}{F}_{s}(\psi)(\lambda^{\circ}), \ \lambda^{\circ}\in\Lambda^{\circ}.
\end{align*}
\end{lemma}
\begin{proof}[Proof of Lemma \ref{le:janssen-function}]
It is a straight forward computation to show that $\psi\in L^{1}(G\times\ghat)$. Indeed,
\begin{align*}
& \int_{G\times\ghat} \vert \psi(\chi) \vert \, d\mu_{G\times\ghat}(\chi) \\
& \le \sum_{\substack{k,l\in\Z_{\NumDim}\\ j\in\Z_{\NumWin}}} \Big( \int_{G\times\ghat} \big\vert \big\langle \vvfun{f^{1}}{k}, \pi(\chi) \mvfun{g}{k}{j}\big\rangle \big\vert^{2}  \, d\mu_{G\times\ghat}(\chi) \Big)^{1/2} \\
& \hspace{3cm} \cdot \, \Big( \int_{G\times\ghat} \big\vert \big\langle \vvfun{f^{2}}{l}, \pi(\chi) \mvfun{h}{l}{j}\big\rangle \big\vert^{2}  \, d\mu_{G\times\ghat}(\chi) \Big)^{1/2} \\
& \stackrel{\eqref{eq:STFT}}{=} \sum_{\substack{k,l\in\Z_{\NumDim}\\ j\in\Z_{\NumWin}}} \Vert \vvfun{f^{1}}{k} \Vert_{2} \, \Vert \mvfun{g}{k}{j}\Vert_{2} \, \Vert \vvfun{f^{2}}{l}\Vert_{2} \, \Vert \mvfun{h}{l}{j}\Vert_{2} \\
& \le \Vert f^{1} \Vert_{2} \, \Vert f^{2} \Vert_{2} \, \Vert g \Vert_{2} \, \Vert h \Vert_{2} < \infty.
\end{align*}

We can thus consider the symplectic Fourier transform of $\psi$. Let $\chi = (x,\omega) \in G\times\ghat$ and $\chi' = (x',\omega')\in G\times\ghat$, then
\begingroup
\allowdisplaybreaks
\begin{align*} & \mathcal{F}_{s}\psi(\chi) = \int_{G\times\ghat} \psi(\chi') \, \cocy_{s}(\chi',\chi) \, d\mu_{G\times\ghat}(\chi') \\
& = \sum_{\substack{ k,l\in\Z_{\NumDim} \\ j\in\Z_{\NumWin} }} \int_{G\times\ghat} \big\langle  \vvfun{f^{1}}{k}, \pi(\chi') \mvfun{g}{k}{j}\big\rangle \, \big\langle \pi(\chi') \mvfun{h}{l}{j}, \vvfun{f^{2}}{l} \big\rangle\, \cocy_{s}(\chi',\chi) \, d\mu_{G\times\ghat}(\chi') \\
& = \sum_{\substack{ k,l\in\Z_{\NumDim} \\ j\in\Z_{\NumWin} }} \int_{G\times\ghat} \big\langle \vvfun{f^{1}}{k}, \pi(\chi') \mvfun{g}{k}{j}\big\rangle \, \big\langle \pi(\chi) \pi(\chi') \mvfun{h}{l}{j}, \pi(\chi) \vvfun{f^{2}}{l} \big\rangle \, \cocy_{s}(\chi',\chi) \, d\mu_{G\times\ghat}(\chi') \\
& = \sum_{\substack{ k,l\in\Z_{\NumDim} \\ j\in\Z_{\NumWin} }} \int_{G\times\ghat} \big\langle  \vvfun{f^{1}}{k}, \pi(\chi') \mvfun{g}{k}{j}\big\rangle \, \big\langle \pi(\chi') \pi(\chi) \mvfun{h}{l}{j}, \pi(\chi) \vvfun{f^{2}}{l}\big\rangle \, d\mu_{G\times\ghat}(\chi') \\
& \stackrel{\eqref{eq:STFT}}{=} \sum_{\substack{ k,l\in\Z_{\NumDim} \\ j\in\Z_{\NumWin} }}  \big\langle \vvfun{f^{1}}{k}, \pi(\chi) \vvfun{f^{2}}{l} \big\rangle \, \big\langle \pi(\chi) \mvfun{h}{l}{j}, \mvfun{g}{k}{j}\big\rangle.
\end{align*}
\endgroup
Since $\psi\in L^{1}(G)$ it follows from \eqref{eq:2503a} that its periodization $\varphi$ belongs to $L^{1}((G\times\ghat)/\Lambda)$. It is clear that the symplectic Fourier transform of the periodization of $\psi$ is the restriction of the symplectic Fourier transform of $\psi$ to $\Lambda^{\circ}$.
\end{proof}

\begin{proof}[Proof of Lemma \ref{le:figa-suff-cond}]
(i).  
Let $B$ be the common upper frame bound of the two Gabor systems generated by the functions $g$ and $h$. Let $\chi$ and $\chi_{0}$ be elements in $G\times\ghat$. Since the time-frequency shift operator $\pi(\chi)$ is continuous from $G\times \ghat$ into $L^{2}(G)$, i.e.,
\[ \lim_{\chi\to 0} \Vert \pi(\chi) f - f\Vert_{2} = 0 \ \ \text{for all} \ \ f\in L^{2}(G\times\Z_{\NumDim})\]
we can conclude, together with the upper frame inequality, that the function $\varphi$ in Lemma \ref{le:janssen-function} is continuous,
\begin{align*} & \lim_{\chi\to \chi_{0}} \vert \varphi(\chi) - \varphi(\chi_{0}) \vert  \\ 
& \le B \lim_{\chi\to\chi_{0}}  \, \big( \Vert \pi(\chi-\chi_{0}) f_{1} - f_{1} \Vert_{2} \, \Vert f_{2} \Vert_{2} + \Vert f_{1} \Vert_{2} \, \Vert \pi(\chi-\chi_{0}) f_{2} - f_{2} \Vert_{2} \big) = 0.
\end{align*}
The assumption \eqref{eq:cond-A} states that the symplectic Fourier transform of $\varphi$ is integrable. Hence $\varphi$ is continuous, integrable and its symplectic Fourier transform is integrable. We conclude that the Fourier inversion formula holds pointwise. In particular, $\varphi(0) = \int_{\Lambda^{\circ}} \mathcal{F}_{s}(\varphi) (\lambda^{\circ}) \, d\mu_{\Lambda^{\circ}}(\lambda^{\circ})$, which is \eqref{eq:figa}.  \\
(ii). No matter which two conditions are satisfied one can show that, for any $j\in\Z_{\NumWin}$ and $k,l\in\Z_{\NumDim}$ the function
\[ \psi_{j,k,l} : G\times\ghat\to \C, \psi(\chi)= \big\langle f_{1}(\, \cdot\, , k) , \pi(\chi) g(\,\cdot\, , k, j) \big\rangle \, \big\langle \pi(\chi) h(\,\cdot\,, l,j) , f_{2}(\,\cdot\,, l)\big\rangle\]
belongs to $\SO(G\times\ghat)$ (see, e.g., the appendix in \cite{jale16-2}). If (c) and (d) are satisfied, then, for some constant $c>0$,
\begin{equation} \label{eq:1209b} \Vert \psi_{j,k,l} \Vert_{\SO} \le c \, \Vert f_{1}(\,\cdot\, , k) \Vert_{2} \, \Vert f_{2}(\,\cdot\, , l) \Vert_{2} \, \Vert g(\,\cdot\, , k, j) \Vert_{\SO} \, \Vert h(\,\cdot\, , l , j) \Vert_{\SO}.\end{equation}
The estimate in \eqref{eq:1209b} is changed accordingly to which functions are assumed to be in $\SO(G)$.
Furthermore there exists a constant $\tilde{c}>0$ such that, no matter which of the assumptions (a)-(d) are satisfied, then
\begin{equation} \label{eq:1209c} \int_{\Lambda} \vert \psi_{j,k,l}(\lambda) \vert \, d\lambda \le \tilde{c} \ \Vert \psi_{j,k,l} \Vert_{\SO} \ \ \text{and} \ \ \int_{\Lambda} \vert \mathcal{F}_{s} \psi_{j,k,l}(\lambda^{\circ}) \vert \, d\mu_{\Lambda^{\circ}} \le \tilde{c} \, \Vert \psi_{j,k,l} \Vert_{\SO}. \end{equation}
Recall that for functions in $\SO$ the Poisson formula is valid, cf.\ Lemma \ref{le:SO-properties}. In particular $\int_{\Lambda} \psi_{j,k,l}(\lambda)  d\mu_{\Lambda}(\lambda) = \int_{\Lambda^{\circ}} \mathcal{F}_{s}(\psi_{j,k,l})(\lambda^{\circ}) \, d\mu_{\Lambda^{\circ}}(\lambda^{\circ})$. Written out this states that
\begin{align*} & \int_{\Lambda} \big\langle \vvfun{f^{1}}{k} , \pi(\lambda) \mvfun{g}{k}{j} \big\rangle \, \big\langle \pi(\lambda) \mvfun{h}{l}{j} , \vvfun{f^{2}}{l}\big\rangle \, d\mu_{\Lambda}(\lambda) \\
& \hspace{3cm} = \int_{\Lambda^{\circ}} \big\langle \mvfun{h}{l}{j}, \pi(\lambda^{\circ}) \mvfun{g}{k}{j} \big\rangle \, \big\langle \pi(\lambda^{\circ}) \vvfun{f^{1}}{k},  \vvfun{f^{2}}{l} \big\rangle \, d\mu_{\Lambda^{\circ}}(\lambda^{\circ}). \end{align*}
All we need to justify for \eqref{eq:figa} to hold is that we may sum over all $k,l\in\Z_{\NumDim}$ and $j\in\Z_{\NumWin}$ so that both sides are absolutely summable. Combining \eqref{eq:1209b}, \eqref{eq:1209c} and either two assumptions implies just that. 
\end{proof}

\begin{proof}[Proof of Lemma \ref{le:s0-implies-bessel}] For a moment let us only consider the multi-window case. By assumption and Lemma \ref{le:figa-suff-cond} we can use \eqref{eq:figa} for any $f\in L^{2}(G)$ (and in particular for any $f\in \SO(G)$), so that
\begin{align*}
S_{g,h,\La} f & =  \sum_{j\in\Z_{\NumWin}} \int_{\Lambda} \langle f, \pi(\lambda) \vvfun{g}{j} \rangle \, \pi(\lambda) \vvfun{h}{j} \, d\lambda \\
& \stackrel{\eqref{eq:figa}}{=} 
\int_{\Lambda^{\circ}} \sum_{j\in\Z_{\NumWin}} \langle \vvfun{h}{j}, \pi(\lambda^{\circ}) \vvfun{g}{j} \rangle \,  \pi(\lambda^{\circ}) f \ d\lambda^{\circ}.
\end{align*}
Because time-frequency shifts are isometries on both $\SO$ and $L^{2}$, we find in either case that
\[ \Vert S_{g,h,\Lambda} f \Vert \le \underbrace{\sum_{j\in\Z_{\NumWin}} \int_{\Lambda^{\circ}}  \big\vert \langle \vvfun{h}{j}, \pi(\lambda^{\circ}) \vvfun{g}{j} \rangle \big\vert \,  d\lambda^{\circ}}_{=:B} \ \Vert f \Vert. \]
The quantity $B$ is finite: combing different statements of Lemma \ref{le:s0-properties}, we establish that, for certain constants $c_{1},c_{2},c_{3}>0$,
\begin{align*}
B & = \sum_{j\in\Z_{\NumWin}} \big\Vert \mathcal{V}_{g_{j}}h_{j} \big\vert_{\Lambda^{\circ}} \big\Vert_{L^{1}(\La^{\circ})} \stackrel{\text{Lemma}\, \ref{le:s0-properties}\text{(iv)}}{\le} \, c_{1} \sum_{j\in\Z_{\NumWin}} \big\Vert \mathcal{V}_{g_{j}}h_{j} \big\vert_{\Lambda^{\circ}} \big\Vert_{\SO(\La^{\circ})} \\
& \stackrel{\text{Lemma}\,\ref{le:s0-properties}\text{(vi)}}{\le} c_{2} \sum_{j\in\Z_{\NumWin}} \big\Vert \mathcal{V}_{g_{j}}h_{j} \big\Vert_{\SO(G\times\ghat)} \\
& \stackrel{\text{Lemma}\,\ref{le:s0-properties}\text{(viii)}}{\le} c_{3} \Big( \sum_{j\in\Z_{\NumWin}} \Vert g_{j} \Vert_{\SO(G)}^{2} \Big)^{1/2} \Big( \sum_{j\in\Z_{\NumWin}} \Vert h_{j} \Vert_{\SO(G)}^{2} \Big)^{1/2}.
\end{align*}
The last term is finite by assumption. Since multi-window \emph{super} Gabor systems are a special case of multi-window Gabor systems the result follows from the previous proof.
\end{proof}

\begin{proof}[Proof of Theorem \ref{th:wex-raz}]
The assumption of (i) implies that $\varphi$ as in Lemma \ref{le:janssen-function} is constant with value $\langle f^{1},f^{2}\rangle$. Therefore, for all $\lambda^{\circ}\in\Lambda^{\circ}$,
\[ \mathcal{F}_{s}\varphi(\lambda^{\circ}) = \langle f^{1},f^{2}\rangle \, \begin{cases} \s(\Lambda)  & \lambda^{\circ} = 0, \\ 0 & \lambda^{\circ} \ne 0.\end{cases}\] 
At the same time Lemma \ref{le:janssen-function} tells us that
\[ \mathcal{F}_{s}(\varphi)(\lambda^{\circ}) = \sum_{ k,l\in\Z_{\NumDim} } \Big( \big\langle \pi(\lambda^{\circ}) \vvfun{f^{1}}{k},  \vvfun{f^{2}}{l} \big\rangle \,  \sum_{j\in\Z_{\NumWin}}\big\langle \mvfun{h}{l}{j}, \pi(\lambda^{\circ}) \mvfun{g}{k}{j} \big\rangle \Big).\]
Since these two expression for $\mathcal{F}_{s}(\varphi)$ must coincide for all choices of $f^{1},f^{2}\in L^{2}(G\times\Z_{\NumDim})$ we can conclude that (ii) must hold.

Conversely, note that by Lemma \ref{le:figa-suff-cond}(i) we are in the position to use the fundamental identity \eqref{eq:figa} for all $f^{1},f^{2}\in L^{2}(G\times\Z_{d})$. Indeed, the assumption (ii) implies that the in Lemma \ref{le:figa-suff-cond}(i) required assumption \eqref{eq:cond-A} is satisfied. It is straightforward to use the assumption in Theorem \ref{th:wex-raz}(ii) in the fundamental identity \eqref{eq:figa} to conclude that the two multi-window super Gabor systems are in fact dual frames for $L^{2}(G\times\Z_{\NumDim})$.
\end{proof}

\begin{proof}[Proof of Lemma \ref{le:nec-cond}(i)]
For a moment let $d=1$. We follow the same idea as in \cite{jale16-2}: a combination of Lemma \ref{le:1009} and the relation in \eqref{eq:2503a} implies that, for any $g\in L^{2}(G\times\Z_{\NumWin})$,
\[ \sum_{j\in\Z_{\NumWin}} \int_{(G\times\ghat)\Lambda} \int_{\Lambda} \big\vert \big\langle \pi(\dot{\chi}) f, \pi(\lambda) \vvfun{g}{j} \big\rangle \big\vert^{2} \, d\mu_{\Lambda}(\lambda) \, d\mu_{(G\times\ghat)/\Lambda}(\dot{\chi})  =  \Vert f \Vert_{2}^{2} \, \Vert g \Vert_{2}^{2} < \infty \ \ \text{for all} \ \ f\in L^{2}(G). \]
If $g$ and $\Lambda$ are such that the frame conditions \eqref{eq:frame} are satisfied, then the lower frame inequality in \eqref{eq:frame-ineq} implies that 
\[ A \, \Vert f \Vert_{2}^{2} \int_{(G\times\ghat)/\Lambda} \, d\mu_{(G\times\ghat)/\Lambda}(\dot{\chi}) \le \Vert f \Vert_{2}^{2} \, \Vert g \Vert_{2}^{2} < \infty \ \ \text{for all} \ \ f\in L^{2}(G). \]
Hence $\s(\Lambda):=\int_{(G\times\ghat)/\Lambda} \, d\mu_{(G\times\ghat)/\Lambda}(\dot{\chi})$ must be finite. This is the case if and only if the quotient group $(G\times\ghat)/\Lambda$ is compact. We conclude that $A \, \s(\Lambda) \le \Vert g \Vert_{2}^{2}$. The upper frame inequality implies that also $ \Vert g \Vert_{2}^{2}\le B \, \s(\Lambda)$. 
We now prove the ``in addition''-part, the inequality $\s(\Lambda) \le n$. If $\NumWin=\infty$, then it is clear that this inequality is satisfied. Assume therefore that $\NumWin$ is finite and that $\Lambda$ is discrete and equipped with the counting measure. The frame assumption for the multi-window Gabor system generated by $g\in L^{2}(G\times\Z_{\NumWin})$ implies that 
the frame operator $S_{g,\Lambda}=D_{g,\Lambda}\circ C_{g,\Lambda}$ is positive and invertible. In particular, we can consider the square root of the inverse frame operator $S_{g,\Lambda}^{-1/2}$. It is a general fact from frame theory (\cite[Theorem 6.1.1]{ch16}) that the multi-window Gabor system generated by $S_{g,\Lambda}^{-1/2}\vvfun{g}{j}$, $j\in\Z_{\NumWin}$ satisfies
\begin{equation} \label{eq:2102a} \sum_{j\in\Z_{\NumWin}} \sum_{\lambda\in\Lambda} \big\vert \big\langle f, \pi(\lambda) S^{-1/2} \vvfun{g}{j} \big\rangle \big\vert^{2} = \Vert f \Vert_{2}^{2} \ \ \text{for all} \ \ f\in L^{2}(G).\end{equation}
If we fix a $j\in\Z_{\NumWin}$ and take $f=S^{-1/2}_{g,\Lambda} \vvfun{g}{j}$, then  we find
\begin{align*} \Vert S^{-1/2}_{g,\Lambda} \vvfun{g}{j}\Vert_{2}^{4} & = \vert \langle S^{-1/2}_{g,\Lambda} \vvfun{g}{j} , S^{-1/2}_{g,\Lambda} \vvfun{g}{j}\rangle \vert^{2} \\
& \le \sum_{j'\in\Z_{\NumWin}} \sum_{\lambda\in\Lambda} \big\vert \big\langle S^{-1/2}_{g,\Lambda} \vvfun{g}{j}, \pi(\lambda) S^{-1/2} \vvfun{g}{j'} \big\rangle \big\vert^{2} \stackrel{\eqref{eq:2102a}}{=} \Vert S^{-1/2}_{g,\Lambda} \vvfun{g}{j} \Vert_{2}^{2}. \end{align*}
Hence $\Vert S^{-1/2}_{g,\Lambda} \vvfun{g}{j}\Vert_{2} \le 1$ for all $j\in \Z_{\NumWin}$ and $\sum_{j\in\Z_{\NumWin}} \Vert S^{-1/2}_{g,\Lambda} \vvfun{g}{j}\Vert_{2} \le \NumWin$. In order to finish the proof we need to use Theorem \ref{th:wex-raz} which is independent of the result we are proving here. Using this we establish that 
\[ \sum_{j\in\Z_{\NumWin}} \Vert S_{g,\Lambda}^{-1/2} \vvfun{g}{j} \Vert_{2} = \sum_{j\in\Z_{\NumWin}} \langle S_{g,\Lambda}^{-1} \vvfun{g}{j}, \vvfun{g}{j} \rangle \stackrel{\textnormal{Theorem} \ \ref{th:wex-raz}}{=} \s(\Lambda). \]
We conclude that $\s(\Lambda) \le \NumWin$.

Assume now that $d\in \N\cup\{\infty\}$. If a multi-window super Gabor system is a frame for $L^{2}(G\times\Z_{\NumDim})$ then, as we just showed, we need the quotient \[ \big[(G\times\ghat) \times (\Z_{\NumDim}\times\widehat{\Z}_{\NumDim})\big]/(\Lambda \times \{0\}) \cong (G\times\ghat)/\Lambda \, \times \, (\Z_{\NumDim}\times\widehat{\Z}_{\NumDim})\] to be compact. This is true if and only if $(G\times\ghat)/\Lambda$ is compact and  $\NumDim<\infty$. In that case $\s(\Lambda\times\{0\}) = \NumDim \, \s(\Lambda)$.
\end{proof}

\subsection*{Acknowledgments}
The work of M.S.J.\ was carried out during the tenure of the ERCIM ``Alain Bensoussan'' Fellowship Programme at NTNU. The authors thank K.\ Gr\"ochenig and J.\ T.\ van Velthoven for discussions and remarks on a gap in the proof of Theorem 5.4, that appeared in an earlier version of the manuscript.

\bibliographystyle{abbrv}

\end{document}